\documentclass[a4paper,10pt]{amsart}
\usepackage{verbatim}
\usepackage{amsfonts}
\usepackage{amsthm}
\usepackage{amsmath}
\usepackage{mathrsfs}
\usepackage{amssymb}
\usepackage{mathabx}
\usepackage[utf8]{inputenc}
\usepackage{enumerate}
\usepackage[all]{xy}
\usepackage{graphicx}
\usepackage{tikz}
\usepackage{tikz-cd}
\usepackage{hyperref}
\usepackage{wasysym}

\hypersetup{                    
	colorlinks=true,                
	breaklinks=true,                
	urlcolor= blue,                 
	linkcolor= blue,                
	citecolor= blue               
}

\theoremstyle{plain}
\newtheorem{theorem}{Theorem}[section]
\newtheorem{lemma}[theorem]{Lemma}
\newtheorem{corollary}[theorem]{Corollary}
\newtheorem{proposition}[theorem]{Proposition}

\newtheorem{definition-theorem}[theorem]{Definition / Theorem}

\newtheorem*{conjecture*}{Conjecture}
\newtheorem*{theorem*}{Theorem}
\newtheorem*{corollary*}{Corollary}

\theoremstyle{definition}
\newtheorem{definition}[theorem]{Definition}

\newtheorem{examples}[theorem]{Examples}

\newcommand{\N}{\mathbb N}
\newcommand{\Z}{\mathbb Z}
\newcommand{\Q}{\mathbb{Q}}

\newcommand{\C}{\mathbb C}

\newcommand{\U}{\mathcal{U}}

\newcommand{\Nc}{\mathcal{N}}
\newcommand{\M}{\mathcal{M}}

\newcommand{\Manoa}{M\=anoa}
\newcommand{\Hawaii}{Hawai\kern.05em`\kern.05em\relax i}

\newtheorem{thmx}{Theorem}

\newtheorem{corx}[thmx]{Corollary}
\newtheorem{prop}[theorem]{Proposition}

\theoremstyle{definition}
\newtheorem{defi}[theorem]{Definition}
\newtheorem{rem}[theorem]{Remark}
{\begin{proof}[Beweis]}
	{\end{proof}}
\newtheorem{ex}[theorem]{Example}

\newcommand{\norm}[1]{\lVert#1\rVert}   
\newcommand{\abs}[1]{\lvert#1\rvert}

\newcommand{\dad}{\mathrm{dad}}
\newcommand{\asdim}{\mathrm{asdim}}
\newcommand{\K}{\mathrm K}

\newcommand{\RR}{\mathbb R}

\newcommand{\CC}{\mathbb C}
\newcommand{\NN}{\mathbb N}
\newcommand{\ZZ}{\mathbb Z}
\newcommand{\FF}{\mathbb F}

\newcommand{\id}{\text{id}}

\newcommand{\Gmod}{G\textrm{-}\mathbf{mod}}

\usepackage{tikz}
\usetikzlibrary{matrix,arrows}
\usepackage{hyperref}

\title{Dynamic asymptotic dimension and Matui's HK conjecture}
\date{\today}

\author[C. B\"{o}nicke]{Christian B\"{o}nicke}
\address{\hskip-\parindent School of Mathematics and Statistics, University of Glasgow, University Gardens, Glasgow, G12 8QQ, UK}
\curraddr{School of Mathematics, Statistics and Physics, Herschel Building, Newcastle University, Newcastle upon Tyne, NE1 7RU, UK}
\email{christian.bonicke@newcastle.ac.uk}
\thanks{CB was supported by the Alexander von Humboldt foundation}

\author[C. Dell'Aiera]{Cl\'{e}ment Dell'Aiera}
\address{\hskip-\parindent ENS Lyon, UMPA, Department of Mathematics, 46 allée d’Italie, 69342 Lyon Cedex 07, France}
\email{clement.dellaiera@ens-lyon.fr}
\thanks{CD was partly supported by the US NSF (DMS 1564281).}

\author[J. Gabe]{James Gabe}
\address{\hskip-\parindent Department of Mathematics and Computer Science, University of Southern Denmark, Campusvej 55, DK-5230 Odense M, Denmark}
\email{gabe@imada.sdu.dk}
\thanks{JG was supported by the Carlsberg Foundation through an Internationalisation Fellowship, and by Australian Research Council grant DP180100595.}

\author[R. Willett]{Rufus Willett}
\address{\hskip-\parindent Department of Mathematics, University of \Hawaii~at \Manoa, 2565 McCarthy Mall, Honolulu, HI 96822, USA}
\email{rufus@math.hawaii.edu}
\thanks{RW was partly supported by the US NSF (DMS 1564281 and DMS 1901522).}

\begin{document}
	
	\begin{abstract} We prove that the homology groups of a principal ample groupoid vanish in dimensions greater than the dynamic asymptotic dimension of the groupoid (as a side-effect of our methods, we also give a new model of groupoid homology in terms of the Tor groups of homological algebra, which might be of independent interest).  As a consequence, the K-theory of the $C^*$-algebras associated with groupoids of finite dynamic asymptotic dimension can be computed from the homology of the underlying groupoid. In particular, principal ample groupoids with dynamic asymptotic dimension at most two and finitely generated second homology satisfy Matui's HK-conjecture.
	
	We also construct explicit maps from the groupoid homology groups to the K-theory groups of their $C^*$-algebras in degrees zero and one, and investigate their properties.
\end{abstract}

\maketitle

\tableofcontents

\section{Introduction}
Dynamic asymptotic dimension is a notion of dimension for actions of discrete groups on locally compact spaces, and more generally, for locally compact \'{e}tale groupoids introduced by the last named author with Guentner and Yu in \cite{Guentner:2014aa}. It is inspired by Gromov's theory of asymptotic dimension \cite[Section 1.3]{Gromov:1993tr}. At the same time it is strongly connected to other existing dimension theories for dynamical systems, for example the conditions introduced by Bartels, L\"{u}ck and Reich \cite{Bartels2018} or Kerr's tower dimension \cite{Kerr2020}.

The original article \cite{Guentner:2014aa} focused on the fine structure of $C^*$-algebras associated with \'{e}tale groupoids of finite dynamic asymptotic dimension, while later work by the same set of authors in \cite{Guentner:2014bh} presented some consequences to K-theory and topology.

In the present work we aim to explore the implications of dynamic asymptotic dimension for groupoid homology and its relation to the K-theory of groupoid $C^*$-algebras. A homology theory for \'{e}tale groupoids was introduced by Crainic and Moerdijk in \cite{Crainic2000}.  More recently, groupoid homology attracted a considerable amount of interest from the topological dynamics and operator algebras communities following the work of Matui \cite{Matui2012}. The main contribution of this article is the following:
\begin{thmx}\label{Thm:A}
	Let $G$ be a locally compact, Hausdorff, étale, principal, $\sigma$-compact, ample groupoid with dynamic asymptotic dimension at most $d$. Then $H_n(G)=0$ for $n>d$ and $H_d(G)$ is torsion-free.
\end{thmx}

Our proof of Theorem \ref{Thm:A} goes via a description of groupoid homology in terms of semi-simplicial spaces equipped with a $G$-action. As a byproduct this leads to a description of these homology groups in terms of the classical Tor groups of homological algebra, quite analogous to the well-known case of the homology of a discrete group.  While this may be known to experts, it seems worthwhile recording it as we are not aware of its appearance in the literature.

\begin{thmx}\label{Thm:B}
	Let $G$ be a locally compact, Hausdorff, étale, ample groupoid with $\sigma$-compact base space.  There is a canonical isomorphism $H_*(G)\cong \text{Tor}_*^{\Z[G]}(\Z[G^0],\Z[G^0])$.
\end{thmx}

Our result allows us to draw some significant consequences to the following conjecture formulated by Matui in \cite{Matui2016}.
\begin{conjecture*}[Matui]
	For a minimal, essentially principal, ample groupoid $G$ there are isomorphisms
	$$K_i(C_r^*(G))\cong \bigoplus_{n\geq 0}H_{2n+i}(G),\quad i=0,1$$
\end{conjecture*}
The conjecture has been confirmed in several interesting cases \cite{Farsi2019,Matui2016,Ortega2020,Yi20} even beyond the originally postulated minimal setting.  However, there are also counterexamples due to Scarparo \cite{Scarparo2020}, and Deeley \cite{Deeley22}.  Scarparo's counterexamples seem to be explained by the presence of torsion in isotropy groups, while Deeley's come from the presence of torsion phenomena in $K$-theory.  It seems very interesting to see where exactly the conjecture holds, and to elucidate the obstructions that exist.

Combining Theorem \ref{Thm:A} with a spectral sequence recently constructed by Proietti and Yamashita \cite{Proietti:2021wz}, we can confirm the HK-conjecture for a large class of low-dimensional groupoids:
\begin{corx}\label{intro iso}
	Let $G$ be a locally compact, Hausdorff, étale, principal, second countable, ample groupoid with dynamic asymptotic dimension at most $2$. If $H_2(G)$ is finitely generated then the HK-conjecture holds for $G$, i.e.
	$$K_0(C_r^*(G))\cong H_0(G)\oplus H_2(G),\quad K_1(C_r^*(G))\cong H_1(G).$$
\end{corx}
Note that having finite dynamic asymptotic dimension forces all the isotropy groups to be locally finite. Consequently, it is natural to restrict our attention to the class of principal groupoids to avoid the trouble caused by torsion in the isotropy groups.

The second theme of this work is an attempt to make the HK-conjecture more explicit. It is well-known and easy to see that there is a canonical homomorphism $$\mu_0:H_0(G)\rightarrow K_0(C_r^*(G)).$$ In general, not much seems to be known about this map, and there are only partial results on the existence of maps in higher dimensions. To formulate our progress in this direction, for an ample groupoid $G$ we denote by $[[G]]$ its topological full group. Moreover, Matui constructs in \cite{Matui2012} an \textit{index map} $I:[[G]]\rightarrow H_1(G)$.
\begin{thmx}\label{Thm:D}
Let $G$ be a locally compact, Hausdorff, étale, ample groupoid. Then there exists a homomorphism
$$\mu_1:H_1(G)\rightarrow K_1(C_r^*(G))$$
which factors the canonical map $[[G]]\rightarrow K_1(C_r^*(G))$ via the index map $I:[[G]]\rightarrow H_1(G)$.
	
If moreover $G$ is principal, second countable, and has dynamic asymptotic dimension at most $2$, then $\mu_0$ and $\mu_1$ induce the injection $H_0(G)\to K_0(C^*_r(G))$ and isomorphism $H_1(G)\to K_1(C^*_r(G))$ from Corollary \ref{intro iso}.
\end{thmx}
The construction of $\mu_1$ is straightforward if the index map $I$ is surjective, and under further structural assumptions on $G$, Matui was already able to prove this. Our construction of the map $\mu_1$ is completely general: we in fact give two independent constructions, an elementary one based on ideas of Putnam \cite{Putnam1997}, and a more sophisticated one based on the work of Proietti and Yamashita \cite{Proietti:2021wz}.

It was shown in \cite{Guentner:2014aa} that the dynamic asymptotic dimension yields an upper bound for the nuclear dimension of reduced groupoid $C^*$-algebras. In particular, if $G$ is a second countable, principal, minimal ample groupoid with finite dynamic asymptotic dimension, then $C_r^*(G)$ is classifiable.\footnote{This is due to Kirchberg and Phillips in the purely infinite case \cite{Kirchberg-ICM}, \cite{Phillips-documenta}, and due to many hands in the finite case, including Elliott, Gong, Lin, and Niu \cite{GongLinNiu-1}, \cite{GongLinNiu-2}, \cite{ElliottGongLinNiu}, and Tikuisis, White, and Winter \cite{TikuisisWhiteWinter} (see \cite{CGSTW} for an alternative proof of classification in the finite case).} Our result allows us to completely determine the classifying invariant (usually called the Elliott invariant and denoted $\mathrm{Ell}(\cdot)$) in the 1-dimensional case:
\begin{corx}\label{Cor:E}
	Let $G$ be a locally compact, Hausdorff, étale, $\sigma$-compact, principal, ample groupoid with compact base space and with dynamic asymptotic dimension at most $1$. Then 
	$$\mathrm{Ell}(C_r^*(G))=(H_0(G),H_0(G)^+,[1_{G^{(0)}}],H_1(G),M(G),p).$$
\end{corx}

In light of these results one is tempted to formulate a stronger version of the HK-conjecture in low dimensions, by asking that for ample principal groupoids with $H_n(G)=0$ for $n\geq 2$ the canonical maps $\mu_i$ are isomorphisms. To see that this cannot be the case, we construct a counterexample using groupoids with topological property (T) introduced in \cite{DellAiera2018}. The example is based on the construction of counterexamples to the Baum-Connes conjecture by Higson, Lafforgue and Skandalis \cite{Higson2002} and work of Alekseev and Finn-Sell \cite{Alekseev2018}.

\begin{thmx}\label{Thm:F}
	There exists a locally compact, Hausdorff, étale, second countable, principal, ample groupoid $G$ with $H_n(G)=0$ for all $n\geq 2$
	such that $\mu_0:H_0(G)\rightarrow K_0(C_r^*(G))$ is not surjective.
\end{thmx}

As an application of our results, we study a geometric class of examples. Given a metric space $X$ with bounded geometry, Skandalis, Tu and Yu construct an ample groupoid $G(X)$ which encodes many coarse geometric properties of the underlying space $X$ \cite{Skandalis2002}. The following result adds to this list of connections. It might be known to experts but does not seem to appear in the literature so far (except for degree zero, which has been treated in \cite{Ara2020}).
\begin{thmx}\label{Thm:G}
	Let $X$ be a bounded geometry metric space and $G(X)$ be the associated coarse groupoid. Then there is a canonical isomorphism
	$$H_*(G(X))\cong H_*^{\mathrm{uf}}(X)$$
	between the groupoid homology of the coarse groupoid and the uniformly finite homology of $X$ in the sense of Block and Weinberger \cite{Block1992}.
\end{thmx}
As the dynamic asymptotic dimension of the coarse groupoid equals the asymptotic dimension of the underlying metric space (in the sense of Gromov), a combination of Theorems \ref{Thm:A} (adapted to the non second countable case) and \ref{Thm:D} yields the following purely geometric corollary.
\begin{corx}\label{Cor:H}
	Let $X$ be a bounded geometry metric space. 
	\begin{enumerate}
		\item If $\mathrm{asdim}(X)\leq 2$ and $H_2^{\mathrm{uf}}(X)$ is free, or finitely generated, then
		$$K_0(C_u^*(X))\cong H_0^{\mathrm{uf}}(X)\oplus H_2^{\mathrm{uf}}(X),\quad K_1(C_u^*(X))\cong H_1^{\mathrm{uf}}(X).$$
		\item If $\mathrm{asdim}(X)\leq 3$, $H_3^\mathrm{uf}(X)$ is free, or finitely generated, and $X$ is non-amenable, then
		$$K_0(C_u^*(X))\cong H_2^{\mathrm{uf}}(X),\quad K_1(C_u^*(X))\cong H_1^{\mathrm{uf}}(X)\oplus H_3^{\mathrm{uf}}(X).$$
	\end{enumerate}
\end{corx}
The first part of this corollary applies for example if $X$ is the fundamental group of a closed, orientable surface, and the second applies for example if $X$ is the fundamental group of a closed, orientable, hyperbolic $3$-manifold.  The reader might compare this to \cite[Theorem B, Theorem F, and Corollary H]{Engel:2019ux}: combining these implies related results after taking a ``completed tensor product with $\C$'' (in an appropriate sense) for certain spaces.

\subsection*{Outline of the paper}

In Section \ref{Gmod sec} we give a new picture of Crainic-Moerdijk homology by defining a $G$-equivariant homology theory for an appropriate notion of semi-simplicial $G$-spaces. This is done in section \ref{Gsphom}, by describing the homology groups as the left derived functor of the coinvariants in the sense of classical homological algebra. The necessary background is given in section \ref{Gmod}. It turns out that the groups $H_*(G)$ will be naturally isomorphic to the equivariant homology of a semi-simplicial $G$-space $EG_*$.  We deduce Theorem \ref{Thm:B} from this material.

The central piece of the vanishing result in Theorem \ref{Thm:A} is tackled in Section \ref{colour sec}. There we define a \textit{colouring} of $G$, which will induce an appropriate cover of $G$. The \textit{nerve} of this cover in the sense of section \ref{nerve} is a semi-simplicial $G$-space, and that defines the homology of the colouring. The central idea is to define an anti-\v{C}ech sequence of $G$ as a sequence of colourings with induced covers that are bigger and bigger, in the spirit of anti-\v{C}ech covers in coarse geometry, introduced by Roe (see \cite[chapter 5]{Roe2003}). We show that the inductive limit of an anti-\v{C}ech cover is well defined and that it converges to the Crainic-Moerdijk homology groups for principal $\sigma$-compact ample groupoids in theorem \ref{main comparison}.

Section \ref{hk sec} is dedicated to the two constructions of the map 
\[\mu_1:H_1(G) \rightarrow K_1(C^*_r(G)), \]
from Theorem \ref{Thm:D}, and a proof that these give the same result.  We also deduce Corollary \ref{intro iso} (in a stronger form) and Corollary \ref{Cor:E}.

Finally, in section \ref{Section:Examples&Applications} we interpret our results in various different settings, including coarse geometry (in particular, we deduce Theorem \ref{Thm:G} and Corollary \ref{Cor:H}) and Smale spaces, and present the negative results regarding a stronger form of the HK conjecture announced in Theorem \ref{Thm:F}.

\section{Models for groupoid homology\label{Gmod sec}}

\subsection{The category of $G$-modules and the coinvariant functor\label{Gmod}}

Let us first fix our notations. A groupoid is a (small) category in which all arrows are invertible. We will write $G$ and $G^{0}$ for the set of arrows and objects respectively; we usually also write $G$ for the groupoid and assume the other data is given.  We will call $G^0$ the base space. The range and source maps are denoted by $r,s: G\rightarrow G^{0}$, and the corresponding fibres by $G^x:=r^{-1}(x)$ and $G_x:=s^{-1}(x)$.  A pair $(g,h)\in G\times G$ is composable if $s(g)=r(h)$, in which case the product is written $gh$.  The identity at $x\in G^0$ is written $e_x\in G$.  For subsets $A,B$ of $G$ we write $AB=\{gh\mid g\in A, h\in B,r(h)=s(g)\}$, and we write $gA$ for $\{g\}A$.

We work exclusively with topological groupoids: $G$ and $G^{0}$ carry locally compact Hausdorff topologies that are compatible, and all structure maps are continuous.  A bisection is a subset $B$ of $G$ such that the restrictions $r|_B$ and $s|_B$ are homeomorphisms onto their images; a groupoid $G$ is \'etale if the open bisections from a basis for its topology, and is ample if the compact open bisections form a basis  for the topology. 

We will be focusing throughout on (\'etale) ample groupoids, or equivalently \'etale groupoids with a totally disconnected base space. Examples include discrete groups and totally disconnected spaces, but also action groupoids of discrete groups acting on Cantor sets by homeomorphisms, coarse groupoids, and more examples associated to $k$-graphs and related combinatorial objects.

If $G$ is an ample groupoid, let $\Z[G]$ denote the set of compactly supported continuous functions with integer values,
\[\Z[G]:= C_c(G,\Z),\]
with the ring structure given by pointwise addition and convolution. This is a ring (in general without identity) with local units, i.e.~for any $f_1,\dots, f_n \in \Z[G]$ there is an idempotent $e\in \Z[G]$ such that $ef_j = f_j e = f_j$ for all $j=1,\dots, n$. The element $e$ can always be picked in $\Z[G^0] \subseteq \Z[G]$ as the characteristic function on a compact open subset of $G^0$.
Define the augmentation map $\varepsilon : \Z [G] \rightarrow \Z[G^0]$ by
\[\varepsilon(f)(x) =\sum_{g\in G_x} f(g).\]

We will denote the category of (left, non-degenerate\footnote{If $R$ is a ring with local units, a left $R$-module $M$ is called \emph{non-degenerate} whenever $RM = M$. All of our modules will be non-degenerate by assumption.}) $\Z[G]$-modules with $\Z[G]$-linear maps as morphisms by $\Gmod$. We will often use the term $G$-module as shorthand for $\Z[G]$-module.

\begin{definition}
	Let $M$ be a $\Z[G]$-module. Define $M_0$ to be the submodule generated by the elements of the form
	\[fm-\varepsilon(f)m,\quad f\in\Z[G],\ m\in M.\]
	The group of \textit{coinvariants} of $M$ is the abelian group 
	\[M_G = M/M_0.\]
	This naturally defines the coinvariant functor 
	\[ \mathrm{Coinv} : \Gmod \rightarrow \mathbf{Ab}.\]
\end{definition}

As it may make the above more conceptual, and as we will need it later, we give a different description of the coinvaraints functor.  Define a right action of $\Z[G]$ on $\Z[G^0]$ by $a\cdot f := \varepsilon (af)$ for $a\in \Z[G^0]$ and $f\in \Z[G]$, i.e.
$$
(a\cdot f)(x) =\sum_{g\in G_x} a(r(g))f(g) \quad \forall f\in \Z[G], a\in \Z [G^0], x\in G^0.
$$
\begin{lemma}\label{coinv tp}
	For any (left, non-degenerate) $G$-module $M$, there is a canonical isomorphism $ M_G\cong \Z[G^0]\underset{\Z[G]}{\otimes} M$.
\end{lemma}

\begin{proof}
	We leave it to the reader to check that the map
	$$
	\Z[G^0]\underset{\Z[G]}{\otimes} M\to M_G,\quad a\otimes m \mapsto am + M_0
	$$
	is an isomorphism of abelian groups. The inverse is given as follows: for $m + M_0 \in M_G$ pick $a\in \Z[G^0] \subseteq \Z[G]$ such that $am = m$. The inverse map is given by $m+ M_0 \mapsto a \otimes m$.
\end{proof}

We are now going to present a natural source of $G$-modules: spaces with a topological action of $G$.  

Recall first that if $p: Y\rightarrow X$ and $q: Z\rightarrow X$ are continuous maps, the fibered product is the space 
\begin{equation}\label{fp}
Y_{p\!\!\!}\times_{\!q} Z:=\{(y,z) : p(y)=q(z) \}.
\end{equation}
equipped with the subspace topology it inherits from $Y\times Z$.

Recall that a left $G$-space is the data of a topological space $X$ together with:
\begin{itemize}
	\item[$\bullet$] a continuous map $p: X \rightarrow G^0$, called the \textit{anchor map}, and
	\item[$\bullet$] a continuous map $\alpha : G_{s\!\!\!}\times_{\!p}X \rightarrow X$, satisfying $\alpha(g,\alpha(h,x)) = \alpha(gh, x)$ for every $(g,h,x)\in G_{s\!\!\!}\times_{\!r}G_{s\!\!\!}\times_{\!p}X $ and $\alpha(e_{p(x)}, x) = x$ for every $x\in X$.  
\end{itemize}
A left $G$-space is \'etale if the anchor map is \'etale, i.e.\ a local homeomorphism.  We will typically suppress $\alpha$ from the notation, for example writing `$gx$' instead of `$\alpha(g,x)$'.  A left $G$-space is free if $gx=x$ forces $g=e_{p(x)}$, and is proper if the map $\alpha\times pr_2:G_{s\!\!\!}\times_{\!p}X \rightarrow X\times X$ is proper.  The quotient of a $G$-space $X$ is written $G\setminus X$, and is the quotient of $X$ by the equivalence relation $x\sim gx$, equipped with the quotient topology.  If the action of $G$ on $X$ is free and proper, then $G\setminus X$ is locally compact and Hausdorff, and the quotient map $q:X\to G\setminus X$ is \'etale.

\begin{definition}\label{topg def}
	Let $\mathbf{Top}_G$ be the category of left $G$-spaces which are locally compact, Hausdorff and \'etale, with morphisms being the $G$-equivariant \'etale maps.  
\end{definition}

We denote by $\Z[X]$ the abelian group of continuous compactly supported integer valued functions on $X$. It is a non-degenerate $G$-module with respect to the action defined by
\[(fa)(x) =\sum_{g\in G^{p(x)}} f(g)a(g^{-1}x) \quad \forall f\in \Z[G], a\in \Z [X], x\in X. \]

If $f: X\rightarrow Y $ is $G$-equivariant and \'etale, it induces a $\Z[G]$-linear map $f_* : \Z[X]\rightarrow \Z[Y]$ by the formula
\[f_*(a)(y) := \sum_{x\in f^{-1}(y)} a(x) \quad \forall y\in Y.\]
The sum is finite since $a\in \Z[X]$ is compactly supported, and $f_*(a)$ is continuous as $f$ is \'etale.

\begin{lemma}
	The assignments $X\mapsto \Z[X]$ and $f\mapsto f_*$ define a functor, called the functor of global sections, 
	\[\Z[?] : \mathbf{Top}_G \rightarrow \Gmod.\]
\end{lemma}

The composition of the global section functor and the coinvariant functor is denoted by $\Z[?]_G$.

\begin{lemma}\label{coinv}
	If $X$ is a free and proper space in $\mathbf{Top}_G$ then there is an isomorphism of abelian groups
	\[\Z[X]_G\cong \Z[G\backslash X].\]
\end{lemma}

\begin{proof} Denote by $[x]\in G\backslash X$ the class of $x\in X$, and define a $\Z[G]$-linear map $\tilde \varepsilon: \Z[X] \rightarrow \Z[G\backslash X]$ by 
	\[\tilde \varepsilon (f) ([x]) = \sum_{g\in G_{p(x)}} f(g\cdot x)\quad \forall x\in X.\]
	This obviously factors through $\Z[X]_G$, giving a map which we still denote by 
	\[\tilde \varepsilon: \Z[X]_G \rightarrow \Z[G\backslash X].\]
	Let us build an inverse to $\tilde \varepsilon $. Since the action of $G$ on $X$ is proper, the quotient is locally compact Hausdorff, and hence a partition of unity argument shows that the family of functions $f\in \Z[G\backslash X]$ such that there exists a compact open subset $V\subseteq X$ with $supp(f)=q(V)$ and such that $q_{\mid V}:V\rightarrow q(V)$ is a homeomorphism, generates $\Z[G\backslash X]$ as an abelian group.
	Now given such a function $f$, we define $f_V:=f\circ q_{\mid V}\in\Z[X]$ and our first goal is to show that the class of $f_V$ in $\Z[X]_G$ does not depend on the choice of $V$. So let $V'\subseteq X$ be another compact open set with $supp(f)=q(V')$ and such that $q_{\mid V'}:V'\rightarrow q(V')$ is a homeomorphism. Then $q(V)=q(V')$ and hence every element $x\in V$ can be written as $x=gy$ for some $y\in V'$, $g\in G$. Using that $G$ is \'{e}tale and compactness of $V$ we can decompose $V=\bigcup_i V_i$ and $V'=\bigcup_i V_i'$ such that there exist bisections $S_i\subseteq G$ implementing a homeomorphism $\alpha_i:V_i\rightarrow V_i'$.
	Hence doing another partition of unity argument, we may assume that the sets $V$ and $V'$ themselves are related in this way, i.e. there exists a bisection $S\subseteq G$ which induces a homeomorphism $\alpha_S:V\rightarrow V'$ by $\alpha_S(x)=hx$ where $h\in S$ is the unique element in $S\cap G^{p(x)}$. But then
	\[f_V(x) =f(q_V(x))=f(q_{V'}(\alpha_S(x)))=f_{V'}(\alpha_S(x)) \quad \forall x\in V.\]
	This relation can be rewritten as $f_V = \chi_{S^{-1}}\cdot f_{V'}$, and hence we have $[f_V]=[f_{V'}]$ in $\Z[X]_G$ as desired.
	Let $\psi:\Z[G\backslash X]\rightarrow \Z[X]_G$ be the map given by $\psi(f)=[f_V]$. If $V\subseteq X$ is open such that $q_{\mid V}:V\rightarrow q(V)$ is a homeomorphism, then for a given $x\in X$ there is at most one $g\in G_{p(x)}$ such that $gx\in V$ since the action is free. It follows that for any function $f\in \Z[G\backslash X]$ supported in such a $V$ we have
	$$\tilde\varepsilon[f_V]([x])=\sum\limits_{g\in G_{p(x)}} f(q_{V}(gx))=f([x])$$
	and hence $\tilde{\varepsilon}\circ \psi=\id$.
	
	Conversely, if $f\in \Z[X]$ is a function supported in a set of the form $S\cdot V$ for a bisection $S\subseteq G$ and $V\subseteq X$ is open such that $q_{\mid V}$ is a homeomorphism onto its image, then one easily checks that $\tilde{\varepsilon}[f]$ is supported in $q(V)$. Hence $[\psi(\tilde{\varepsilon}[f])]=[\tilde{\varepsilon}[f]\circ q_{\mid _V}]$. Using freeness again, one checks that the latter class equals $[\chi_S\cdot f]=[\varepsilon(\chi_{S^{-1}})\cdot f]=[f].$ Since functions $f$ as above generate $\Z[X]$ we are done.
\end{proof}

\subsection{Semi-simplicial $G$-spaces and homology\label{Gsphom}}

For $n$ a nonnegative integer, denote by $[n]$ the interval $\{0,...,n\}$. Recall (see for example \cite[Chapter 8]{Weibel1994}) the definition of  the semi-simplicial category $\Delta_s$: its objects are the nonnegative integers, and $\mathrm{Hom}_{\Delta_s}(m,n)$ consists of the (injective) increasing maps $f: [m]\rightarrow [n]$. A semi-simplicial object in a category $\mathcal C$ is a contravariant functor from $\Delta_s$ to $\mathcal C$.  The collection of all semi-simplicial objects in a category is itself a category, with morphisms given by natural transformations.

Let $\varepsilon^n_i: [n-1]\rightarrow [n]$ be the only increasing map whose image misses $i$. We will omit the superscript $n$ if it does not cause confusion. Any increasing map $f: [m]\rightarrow [n]$ has a unique factorization $f=\varepsilon_{i_1}\varepsilon_{i_2}...\varepsilon_{i_k}$ with $0\leq i_k \leq ...\leq i_1 \leq n$ (see Lemma 8.1.2 in \cite{Weibel1994}). Thus, any semi-simplicial object is the data, for all $n$, of an object $C_n$ of $\mathcal C$, together with arrows $\varepsilon^n_i: C_{n}\rightarrow C_{n-1}$ in $\mathcal C$, called face maps, satisfying the (semi-)simplicial identities $\varepsilon^{n-1}_j \varepsilon^n_i = \varepsilon^{n-1}_i \varepsilon^n_{j-1}$ if $i<j$.  Similarly, if $C_*$ and $D_*$ are semi-simplicial objects in $\mathcal{C}$, a morphism between them is a collection of morphisms $f_n:C_n\to D_n$ in $\mathcal{C}$ that are compatible with the face maps.

A semi-simplicial object in the category of locally compact Hausdorff topological spaces will be called a \textit{semi-simplicial topological space}.

\begin{definition}
	A \textit{semi-simplicial $G$-space} is a semi-simplicial object in the category $\mathbf{Top}_G$.
\end{definition}

As an example, define $EG_*$ to be the semi-simplicial $G$-space
\[EG_n = G_{r\!\!\!}\times_{\!r} G _{r\!\!\!}\times_{\!r} ... _{r\!\!\!}\times_{\!r}G \quad \text{(n+1 times)}\]
with
\begin{itemize} 
	\item[$\bullet$] anchor map $p:EG_n\rightarrow G^0$ given by the common range of the tuple, $p(\gamma_0,...,\gamma_n)= r(\gamma_0)=...=r(\gamma_n)$ for $\gamma\in EG_n$,
	\item[$\bullet$] left action given by left multiplication by $G$ on all factors,
	\item[$\bullet$] if $f : [m]\rightarrow [n]$, then $EG(f) :G_n\rightarrow G_m $ is defined by 
	\[(\gamma_0,...,\gamma_n)\mapsto (\gamma_{f(0)}, ..., \gamma_{f(m)})\]
	(note that the face maps are given by 
	\[\partial^{n+1}_i: (\gamma_0,...,\gamma_n)\mapsto (\gamma_{0}, ...,\hat \gamma_i, ..., \gamma_n)\]
	where the hat means that the entry is omitted).
\end{itemize}
One checks that the moment and the face maps are $G$-equivariant and \'etale.\\

On the other hand, define the \textit{classifying space} of $G$, denoted by $BG_*$, to be the semi-simplicial topological space defined by 
\[BG_n = \{(g_1,...,g_n)\in G^n \ | \ r(g_i)=s(g_{i-1}) \text{ for all }i \}\]
with face maps
\[\epsilon^n_i : (g_1,...,g_n)\mapsto 
\left\{\begin{array}{cl}
(g_2,...,g_n) & \text{ if } i=0, \\
(g_1,...,g_{i-1}, g_ig_{i+1}, ..., g_n) & \text{ if } 1\leq i <n, \\
(g_1,...,g_{n-1} ) & \text{ if }i =n. \\
\end{array}\right.\]

The $G$-action on $EG$ is free and proper and hence $G\backslash EG_*$ is a semi-simplicial topological space. The maps $(g_1,...,g_n)\mapsto [r(g_1), g_1 ,g_1g_2,\ldots, g_1\cdots g_n ]$ and $[\gamma_0,\ldots,\gamma_n]\mapsto (\gamma_0^{-1}\gamma_1,\ldots, \gamma^{-1}_{n-1}\gamma_n)$ define maps of semi-simplicial topological spaces between $BG_*$ and $G\backslash EG_*$ that are mutually inverse.

A semi-simplicial $G$-space $(X_*, \{\epsilon^*_i\}_i)$ naturally induces a chain complex of $\Z[G]$-modules $(\Z[X_*], \tilde\partial_*)$ by composing by the $\Z[?]$ functor, where the boundary maps are
\[ \tilde\partial_n = \sum_{i=0}^n (-1)^i (\partial^n_i)_*\]
and thus a chain complex of abelian groups $(\Z[X_*]_G,\tilde\partial_*)$ by composing by the coinvariant functor.

\begin{definition}\label{hom sgs}
	Let $(X_*,\{\partial^*_i\}_i)$ be a semi-simplicial $G$-space. We define the equivariant homology  group $H^G_*(X)$ to be the homology of the chain complex of abelian groups $(\Z[X_*]_G, \tilde\partial_*)$. 
\end{definition}

The relationship with the homology of Matui and Crainic-Moerdijk is now a consequence of Lemma \ref{coinv}. Namely, Matui introduced a chain complex of abelian groups to compute Crainic-Moerdijk homology groups in the special case of an ample groupoid. This chain complex is none other than $(\Z[BG_*],\tilde\epsilon_* )$. As $EG_*$ is a free and proper semi-simplicial $G$-space with $BG_* \cong G\backslash EG_*$, this homology is isomorphic to $H_*^G(EG)\cong H_*(BG)$.  In other words, we have proved the following.

\begin{proposition}\label{hom id}
	Let $G$ be an ample groupoid.  The complex introduced by Matui to compute $H_*(G)$ identifies canonically with the homology of the complex $(\Z[EG_*]_G, \tilde\partial_*)$. \qed
\end{proposition}

\subsection{Projective resolutions and Tor\label{tor sec}}

The aim of this section is to identify $H_*(G)$ with one of the standard objects studied in homological algebra.  These results will not be used in the rest of the paper; we include them as they can be derived without too much difficulty from our other methods, and seem interesting.

\begin{rem}\label{r:freevsproj}
	If $R$ is a non-unital ring, then free $R$-modules are not necessarily projective. However, for any idempotent $e\in R$ it is easily seen that $\mathrm{Hom}_R(Re, M) \cong eM$ naturally for any $R$-module $M$. If $N\twoheadrightarrow M$ is an epimorphism then elements in $eM$ lift to $eN$, so it follows that $Re$ is projective. Consequently, if $R$ is a ring with enough idempotents, i.e.~it contains a family $(e_i)_{i\in I}$ of mutually orthogonal idempotents such that $R \cong \bigoplus_{i\in I} R e_i \cong \bigoplus_{i\in I} e_i R$, then free $R$-modules are projective. More generally, a (non-degenerate) $R$-module is projective exactly when it is a direct summand of a free $R$-module.
\end{rem}

We will restrict to ample groupoids $G$ with $\sigma$-compact base space $G^0$. The main reason is that it implies that $G^0$ (resp.~$G$) can be written as a disjoint union of compact open sets (resp.~compact open bisections).\footnote{In general, this fails for locally compact, Hausdorff, totally disconnected spaces, such as Counterexample 65 from \cite{Steen78} ($\mathbb R$ equipped with a rational sequence topology). One can also show that if $X$ is this particular example, then $\Z[X]$ is not projective as a $\Z[X]$-module, so free $\Z[X]$-modules are not projective.} If $G^0 = \bigsqcup_{i\in I} U_i$ with $U_i$ compact and open, then $\Z[G] \cong \bigoplus_i \Z[G]\chi_{U_i} \cong \bigoplus_i \chi_{U_i}\Z[G]$. By the above remark it follows that free $G$-modules are projective, and similarly for $\Z[G^0]$.

To state the main result of this subsection, note that $G^0$ admits left and right actions of $G$ (with the anchor map being the identity) defined by 
$$
gx:=r(g) \quad \text{and}\quad xg:=s(g).
$$
These actions make $\Z[G^0]$ into both a left and a right $G$-module.  Thus the Tor groups $\text{Tor}_*^{\Z[G]}(\Z[G^0],\Z[G^0])$ (see for example \cite[Definition 2.6.4]{Weibel1994}) of homological algebra make sense.  

\begin{theorem}\label{tor the}
	Let $G$ be an ample groupoid with $\sigma$-compact base space.  There is a canonical isomorphism $H_*(G)\cong \mathrm{Tor}_*^{\Z[G]}(\Z[G^0],\Z[G^0])$.
\end{theorem}

The rest of this section will be spent proving this theorem, which will proceed by a sequence of lemmas.  To give the idea of the proof, recall (see for example \cite[Definition 2.6.4 and Theorem 2.7.2]{Weibel1994}) that $\text{Tor}_*^{\Z[G]}(\Z[G^0],\Z[G^0])$ can be defined by starting with an exact sequence 
$$
\xymatrix{\cdots \ar[r]^-{\partial} & P_2 \ar[r]^-{\partial} & P_1 \ar[r]^-{\partial} & P_0 \ar[r]^-{\partial} & \Z[G^0] \ar[r] & 0}
$$
of $\Z[G]$-modules, where each $P_n$ is projective. The group $\text{Tor}_n^{\Z[G]}(\Z[G^0],\Z[G^0])$ is then by definition the $n^\text{th}$ homology group of the complex
$$
\xymatrix{\cdots \ar[r]^-{\text{id}\otimes \partial} & \Z[G^0]\underset{\Z[G]}{\otimes}P_2 \ar[r]^-{\text{id}\otimes \partial} & \Z[G^0]\underset{\Z[G]}{\otimes}P_1 \ar[r]^-{\text{id}\otimes \partial} & \Z[G^0]\underset{\Z[G]}{\otimes}P_0 }
$$
of abelian groups.  We will prove Theorem \ref{tor the} by showing that each $\Z[G]$-module $\Z[EG_n]$ is projective, and that we have an exact sequence 
$$
\xymatrix{ \cdots \ar[r]^-\partial & \Z[EG_1] \ar[r]^-\partial  & \Z[EG_0] \ar[r]^-\partial & \Z[G^0] \ar[r] & 0 }
$$
where the boundary maps are the alternating sums of the face maps. As 
$$\Z[G^0]\underset{\Z[G]}\otimes M=M_G$$ 
for any left non-degenerate $G$-module $M$ by Lemma \ref{coinv tp}, Proposition \ref{hom id} completes the proof.

We now embark on the details of the proof.  Recall that $G^0$ is equipped with a left $G$-action defined by stipulating that the anchor map is the identity and defining the moment map by $gx:=r(g)$, and that $\Z[G^0]$ is a left $G$-module with the induced structure.  Define $\partial:\Z[EG_0]\to \Z[G^0]$ to be the map induced on functions by the \'{e}tale map $r:G\to G^0$.  

\begin{lemma}\label{ex proj}
	The sequence 
	$$
	\xymatrix{ \cdots \ar[r]^-\partial & \Z[EG_1] \ar[r]^-\partial  & \Z[EG_0] \ar[r]^-\partial & \Z[G^0] \ar[r] & 0 }.
	$$
	is exact.
\end{lemma}

\begin{proof}
	For $n\geq 0$, define 
	$$
	h:EG_n\to EG_{n+1} ,\quad (g_0,\ldots,g_n)\mapsto (r(g_0),g_0,\ldots,g_n).
	$$
	Then one computes on the spatial level that $\partial^ih=h\partial^{i-1}$ for $1\leq i\leq n$, and $\partial^0h$ is the identity map.  Hence by functoriality 
	\begin{align*}
	\partial h_*+h_*\partial & =\sum_{i=0}^n (-1)^i (\partial^ih)_*+\sum_{i=0}^{n-1} (-1)^i(h\partial^i)_* \\ & =(\partial^0h)_*+\sum_{i=1}^n  (-1)^i \big((\partial^i h)_* - (h\partial^{i-1})_*\big),
	\end{align*}
	and this is the identity.  Define also $h:G^0\to EG_0$ by $h(x)=x$.  Then the map $r h:G^0\to G^0$ is the identity, and so $\partial h_*$ is the identity map on $\Z[G^0]$.  
	
	To summarise, $h_*$ is a chain homotopy (as in for example \cite[Definition 1.4.4]{Weibel1994}) between the identity map and the zero map. This implies that the complex has trivial homology (see for example \cite[Lemma 1.4.5]{Weibel1994}), or, equivalently, is exact.
\end{proof} 

For the next lemma, let us define a right action of $\Z[G^0]$ on $\Z[G]$ via 
$$(f\phi)(g):=f(g)\phi(s(g)) \quad \forall f\in \Z[G], \phi\in \Z [G^0] $$ 
and similarly a left action of $\Z[G^0]$ on $\Z[EG_n]$ via $(\phi a)(g_0,...,g_n):=\phi(r(g_0))a(g_0,...,g_n)$.  Note that this right action of $\Z[G^0]$ on $\Z[G]$ commutes with the canonical left action of $\Z[G]$ on itself: indeed, the action of $\Z[G^0]$ is just the action by right-multiplication of the submodule $\Z[G^0]$ of $\Z[G]$, so this commutativity statement is associativity of multiplication. For notational convenience, we define $EG_{-1}:=G^0$.

\begin{lemma}\label{tp lem}
	With notation as above, for any $n\geq -1$ there is a canonical isomorphism of $\Z[G]$-modules.
	$$
	\Z[G]\underset{\Z[G^0]}{\otimes}\Z[EG_n]\cong \Z[EG_{n+1}].
	$$
\end{lemma}

\begin{proof}
	For $f\in \Z[G]$ and $a\in \Z[EG_n]$, we define 
	$$
	(f,a):EG_{n+1}\to \Z,\quad (g_0,...,g_{n+1})\mapsto f(g_0)a(g_0^{-1}g_1,...,g_0^{-1}g_{n+1}).
	$$
	Note that $(f,b)$ is in $\Z[EG_{n+1}]$ by the support conditions defining $\Z[G]$ and $\Z[EG_n]$.  Note also that if $\phi\in \Z[G^{0}]$, then 
	$$
	(f\phi,a):(g_0,...,g_{n+1})\mapsto f(g_0)\phi(s(g_0))a(g_0^{-1}g_1,...,g_0^{-1}g_n)
	$$
	and 
	$$
	(f,\phi a):(g_0,...,g_{n+1})\mapsto f(g_0)\phi(r(g_0^{-1}g_1))a(g_0^{-1}g_1,...,g_0^{-1}g_n)  ;
	$$
	these are the same however, as for any $(g_0,...,g_{n+1})\in EG_{n+1}$, $r(g_0^{-1}g_1)=s(g_0)$.  Hence by the universal property of the balanced tensor product, we have a well-defined map
	$$
	\Z[G]\underset{\Z[G^{0}]}{\otimes}\Z[EG_n]\to \Z[EG_{n+1}].
	$$
	We claim that this map is an isomorphism.  
	
	For injectivity, say we have an element $\sum_{i=1}^n f_i\otimes a_i$ that goes to zero.  Splitting up the sum further, we may assume that each $f_i$ is the characteristic function $\chi_{B_i}$ of a compact open bisection $B_i$ such that $B_i\cap B_j=\varnothing$ for $i\neq j$.  For each $i$, let $\phi_i$ be the characteristic function of $s(B_i)$.  As $\chi_{B_i}=\chi_{B_i}\phi_i$, we have that $\chi_{B_i}\otimes a_i=\chi_{B_i}\otimes \phi_ia_i$, we may further assume that each $a_i$ is supported in $\{(g_0,...,g_{n})\in EG_n \mid r(g_0)\in s(B_i)\}$.  Now, we are assuming that $\sum_{i=1}^n (\chi_{B_i},a_i)=0$.  As $B_i\cap B_j=\varnothing$ for $i\neq j$, the functions $(\chi_{B_i},a_i)$ have disjoint supports, and therefore we have that $(\chi_{B_i},a_i)=0$ for each $i$.  Assume for contradiction that $a_i\neq 0$, so there exists $(h_0,...,h_n)\in EG_n$ with $a(h_0,...,h_n)\neq 0$; our assumptions force $r(h_0)\in s(B_i)$.  Let $g\in B_i$ be such that $s(g)=r(h_0)$.  Then $(g,gh_0,...,gh_n)\in EG_{n+1}$ and $(\chi_{B_i},a)$ evaluates to $a(h_0,...,h_n)\neq 0$ at this point, giving the contradiction and completing the proof of injectivity.
	
	For surjectivity, as any element of $\Z[EG_{n+1}]$ is a finite $\Z$-linear combination of characteristic functions of subsets of the form $B_0\times \cdots \times B_{n+1}\cap EG_{n+1}$ with each $B_i$ a compact open bisection in $G$, it suffices to show that any such characteristic function is in $B$.  Set $f\in \Z[G]$ to be the characteristic function of $B_0$, and $a\in \Z[EG_n]$ to be the characteristic function of $B_0^{-1}B_1\times \cdots B_0^{-1}B_n\cap EG_n$.  We leave it to the reader to check that $f\otimes a$ maps to the function we want.
\end{proof}

For the next lemma, we consider $\Z[G]$ as a left $\Z[G^0]$-module via the left-multiplication action induced by the inclusion $\Z[G^0]\subseteq \Z[G]$.

\begin{lemma}\label{zg proj}
	If $G^0$ is $\sigma$-compact, then, considered as a (left) module over $\Z[G^0]$, $\Z[G]$ is projective.
\end{lemma}

In many interesting cases $\Z[G]$ is actually free over $\Z[G^0]$: for example, this happens for transformation groupoids associated to actions of discrete groups.  However, freeness does not seem to be true in general.

\begin{proof}
	For a compact open bisection $U$, let $\chi_U\in \Z[G]$ denote the characteristic function of that bisection.   As $G^0$ is $\sigma$-compact we may choose a covering $G=\bigsqcup_{i\in I}U_i$ of $G$ by disjoint compact open bisections, and for each $i$, let $\Z[G^0]\chi_{U_i}$ denote the $\Z[G^0]$-submodule of $\Z[G]$ generated by $\chi_{U_i}$.  Then as $\Z[G^0]$-modules, we have 
	$$
	\Z[G]\cong \bigoplus_{i\in I} \Z[G^0]\chi_{U_i}.
	$$
	It thus suffices to prove that each $\Z[G^0]\chi_{U_i}$ is projective.  For this, one checks that $\Z[G^0]\chi_{U_i}$ is isomorphic as a $\Z[G^0]$-module to $\Z[G^0]\chi_{r(U_i)}$, so projective by Remark \ref{r:freevsproj}, and we are done.
\end{proof}

\begin{corollary}\label{bn proj}
	If $G^0$ is $\sigma$-compact, the $G$-module $\Z[EG_n]$ is projective for $n\geq 0$.
\end{corollary}

\begin{proof}
	We proceed by induction on $n$.  For $n=0$, $\Z[EG_0]$ identifies with $\Z[G]$ as a left $\Z[G]$-module, so free, and thus projective by Remark \ref{r:freevsproj}.  Now assume we have the result for $\Z[EG_n]$.  By Remark \ref{r:freevsproj}, $\Z[EG_n]$ is a direct summand in a free $\Z[G]$-module, say $\bigoplus_J \Z[G]$. Hence, as $\Z[G^0]$-modules, $\Z[EG_n]$ is a direct summand in $\bigoplus_J \Z[G]$ which is projective by Lemma \ref{zg proj}. Therefore, $\Z[EG_n]$ is projective as a $\Z[G^0]$-module.
	Let then $N$ be a $\Z[G^0]$-module such that $\Z[EG_n]\oplus N$ is isomorphic as a $\Z[G^0]$-module to $\bigoplus_{i\in I} \Z[G^{0}]$ for some index set $I$.  It follows that as $\Z[G]$-modules 
	\begin{align*}
	\Big(\Z[G]\underset{\Z[G^0]}{\otimes}\Z[EG_n]\Big)\oplus \Big(\Z[G]\underset{\Z[G^0]}{\otimes}N\Big) & \cong \Z[G]\underset{\Z[G^0]}{\otimes}(M\oplus N)  \\ & \cong \Z[G]\underset{\Z[G^0]}{\otimes}\Big(\bigoplus_{i\in I} \Z[G^{0}]\Big)\\ & \cong \bigoplus_{i\in I} \Z[G].
	\end{align*}
	Hence $\Z[G]\underset{\Z[G^0]}{\otimes}\Z[EG_n]$ is isomorphic to a direct summand of a free $\Z[G]$-module, so projective as a $\Z[G]$-module.  Using Lemma \ref{tp lem}, we are done.
\end{proof}

\begin{proof}[Proof of Theorem \ref{tor the}]
	Lemma \ref{ex proj} and Corollary \ref{bn proj} together imply that 
	$$
	\xymatrix{ \cdots \ar[r]^-\partial & \Z[EG_1] \ar[r]^-\partial  & \Z[EG_0] \ar[r]^-\partial & \Z[G^0] \ar[r] & 0 }.
	$$
	is a resolution of $\Z[G^0]$ by projective modules.  The groups $\text{Tor}_*^{\Z[G]}(\Z[G^0],\Z[G^0])$ are therefore by definition the homology groups of the complex
	$$
	\xymatrix{ \cdots \ar[r]^-{\text{id}\otimes \partial} & \Z[G^0]\underset{\Z[G]}{\otimes}\Z[EG_1] \ar[r]^-{\text{id}\otimes \partial}   & \Z[G^0] \underset{\Z[G]}{\otimes}\Z[EG_0] \ar[r] & 0  }.
	$$
	However, using Lemma \ref{coinv tp}, this is the same as the complex 
	$$
	\xymatrix{ \cdots \ar[r]^-\partial & \Z[EG_1]_G \ar[r]^-\partial  & \Z[EG_0]_G \ar[r] & 0 },
	$$
	and we have already seen that the homology of this is the same as the homology $H_*(G)$.
\end{proof}

\section{Colourings and homology\label{colour sec}}

The goal of this section is to use what we shall call a \emph{colouring} of a groupoid $G$ to produce an associated nerve space, and prove that the nerve is a semi-simplicial $G$-space.  This lets us define the homology of a colouring in terms of the homology of the associated semi-simplicial $G$-space. Throughout most of the section we assume that $G$ is an ample groupoid with compact base space $G^0$: the most important exception is the main result---Theorem \ref{main vanishing}---at the end, where we drop the compact base space assumption.

\subsection{Colourings and nerves\label{nerve}}

In this subsection we introduce colourings of a groupoid and the associated nerve spaces and homology.

\begin{definition}\label{colouring}
	A \emph{colouring} for $G$ is a finite ordered collection $\mathcal{C}=(G_0,...,G_d)$ of compact open subgroupoids of $G$ such that the collection $\{G_0^{0},...,G_d^{0}\}$ of unit spaces covers $G^{0}$.  
	
	Elements of the set $\{0,...,d\}$ are called the \emph{colours} of the colouring, and the colour of $G_i$ is $i$.
\end{definition}

We will associate a cover of $G$ to a colouring.  Let $\mathcal{P}(G)$ denote the collection of subsets of $G$.

\begin{definition}\label{cover}
	Let $G_0,...,G_d$ be a colouring of $G$.  The \emph{cover} associated to the colouring is 
	$$
	\mathcal{U}:=\{(gG_i^{s(g)},i)\in \mathcal{P}(G)\times\{0,...,d\}\mid g\in G, G_i^{s(g)}\neq \varnothing\}.
	$$
\end{definition}

Typically, we will just write $U$ for an element of $\mathcal{U}$, and treat $U$ as a (non-empty) subset of $G$. In particular, when considering intersections $U_0\cap U_1$ for $U_0,U_1\in \mathcal{U}$ we just mean the intersection of the corresponding subsets of $G$ (ignoring the colour!). The precise definition however calls for pairs as we want each element of $\mathcal{U}$ to have a well-defined colour in $\{0,...,d\}$.  Note that an element $U\in \mathcal{U}$ could be equal to $(gG_i^{s(g)},i)$ and $(hG_i^{s(h)},i)$ for $g\neq h$ in $G$, i.e.\ representations of elements of $\mathcal{U}$ of the form $gG_i^{s(g)}$ need not be unique.  Note also that the fact that $G_0^{0},...,G_d^{0}$ covers $G^{0}$ implies that $\mathcal{U}$ is a cover of $G$, but typically not by open sets: indeed, an element $gG_i^{s(g)}$ of $\mathcal{U}$ is contained in the single range fibre $G^{r(g)}$.  

\begin{definition}\label{nerve}
	Let $G_0,...,G_d$ be a colouring of $G$ with associated cover $\U$.  For $n\geq 0$, set 
	$$
	\Nc_n:=\Big\{(U_0,...,U_n)\in \mathcal{U}^{n+1} ~\Big| ~\bigcap_{i=0}^n U_i\neq \varnothing\Big\}.
	$$
	The sequence $(\Nc_n)_{n=0}^\infty$ is denoted $\Nc_*$, and called the \emph{nerve} of the colouring.  
\end{definition}
As noted before, the intersection in the definition above is to be interpreted as the intersection of the corresponding subsets of $G$ (ignoring the colour). In particular we allow distinct $U_i$ appearing in a tuple as above to have different colours.

Our next goal is to give $\Nc_*$ the structure of a semi-simplicial $G$-space.

\begin{definition}\label{nerve act}
	For each $n$, the anchor map $r:\Nc_n\to G^{0}$ takes $(U_0,...,U_n)$ to the unique $x\in G^0$ such that $U_0$ is a subset of\footnote{and therefore all the $U_j$ are subsets of} $G^x$.  Let $G_{s\!\!\!}\times_{\!r}\Nc_n$ be the fibered product as in line \eqref{fp} above, and define an action by 
	$$
	G_{s\!\!\!}\times_{\!r}\Nc_n\to \Nc_n,\quad (g,(U_0,...,U_n))\mapsto (gU_0,...,gU_n).
	$$
\end{definition}

Direct checks show this is a well-defined groupoid action of $G$ on the set $\Nc_n$.  Our next goal is to introduce a topology on $\Nc_n$, and prove that the action is continuous.

\begin{definition}\label{nerve top}
	Let $i\in \{0,...,d\}$ and let $V$ be an open subset of $G$ such that $G_i^{s(g)}$ is non-empty for all $g\in V$.  Define
	$$
	U_{V,i}:=\{(gG_i^{s(g)},i)\mid g\in V\},
	$$
	and equip $\Nc_0$ with the topology generated by these sets. For each $n\geq 0$, equip $\Nc_0^{n+1}$ with the product topology, and give $\Nc_n\subseteq \Nc_0^{n+1}$ the subspace topology.
\end{definition}

\begin{lemma}\label{nerve top 0}
	For each $n$, the topology on $\Nc_n$ is locally compact, Hausdorff, and totally disconnected.  Moreover, the action defined in Definition \ref{nerve act} above is continuous.
\end{lemma}

\begin{proof}
	We first look at the case $n=0$. Given a compact open bisection $V\subseteq G$ with $s(V)\subseteq G_i^0$ one readily verifies that the mapping $g\mapsto (gG_i^{s(g)},i)$ defines a homeomorphism $V\rightarrow U_{V,i}$. In particular, the sets $U_{V,i}$ for such compact open bisections $V$ are themselves compact open Hausdorff subsets of $\Nc_0$.  Consequently, $\Nc_0$ is locally compact and locally Hausdorff.
	To check that it is Hausdorff, it suffices to check that if $(U_j)_j$ is a net in $\Nc_0$ that converges to both $U$ and $V$, then $U=V$.  Equivalently, say we have a net $(g_jG_{i_j}^{s(g_j)},i_j)_j$ in $\Nc_0$ such that there are $h_j$ in $G$ with $g_jG_{i_j}^{s(g_j)}=h_jG_{i_j}^{s(h_j)}$ for all $j$, and so that $g_j\to g$ and $h_j\to h$. First note that since the net converges, $i_j$ will be eventually constant so we may as well assume that $i_j=i$ for all $j$; we need to check that $gG_i^{s(g)}=hG_i^{s(h)}$. 
	
	Say then that $gk$ is in $gG_i^{s(g)}$ with $k\in G_i^{s(g)}$.  By symmetry, it suffices to check that $gk$ is in $hG_i^{s(h)}$.  As $G_i$ is open in $G$, it is \'{e}tale, whence for all suitably large $j$, we can find $k_j\in G_i^{s(g_j)}$ such that $k_j\to k$.  As $g_jG_i^{s(g_j)}=h_jG_i^{s(h_j)}$ for all $j$, we can also find $l_j\in G_i^{s(h_j)}$ for all $j$ so that $g_jk_j=h_jl_j$.  Using that $G_i$ is compact, we may pass to a subnet and so assume that $(l_j)$ converges to some $l\in G_i$, which is necessarily in $G_i^{s(h)}$.  Hence $gk=\lim g_jk_j=\lim h_jl_j=hl$, and so $gk$ is in $hG_i^{s(h)}$ as required.
	
	The first paragraph of the proof implies that $\Nc_0$ admits a basis of compact open subsets and hence it is totally disconnected.
	
	Continuity of the action follows on observing that if $(g_j,h_jG_i^{s(h)})$ is a convergent net in $G_{s\!\!\!}\times_{\!r}\Nc_0$, then $(g_jh_jG_i^{s(g_jh)})$ is a convergent net in $\Nc_0$ by continuity of the multiplication in $G$. 
	
	We now look at the case of general $n$.  The facts that $\Nc_n$ is Hausdorff and totally disconnected, as well as the continuity of the $G$-action, all follow directly from the corresponding properties for $\Nc_0$.  
	
	We claim that $\Nc_n$ is closed in $\Nc_0^{n+1}$; as closed subsets of locally compact spaces are locally compact, this will suffice to complete the proof.  To check closedness, for each $j\in \{0,...,n\}$, say $(g_j^k)_{k\in K}$ is a net such that $g_j^k\to g_j$ as $k\to\infty$, and such that $(g_0^kG_i^{s(g_0^k)},...,g_n^kG_i^{s(g_n^k)})$ is in $\Nc_n$ for all $k$.  We need to show that $(g_0G_i^{s(g_0)},...,g_nG_i^{s(g_n)})$ is also in $\Nc_n$.  Indeed, as $(g_0^kG_i^{s(g_0^k)},...,g_n^kG_i^{s(g_n^k)})$ is in $\Nc_n$, there exist $h_0^k,...,h_n^k\in G_i$ such that 
	$$
	g_0^kh_0^k=g_1^kh_1^k=\cdots=g_n^kh_n^k.
	$$
	As $G_i$ is compact, we may assume that each net $(h_j^k)_{k\in K}$ converges to some $h_j$ in $G_i$.  Hence 
	$$
	g_0h_0=g_1h_1=\cdots=g_nh_n
	$$
	is a point in $g_0G_i^{s(g_0)}\cap \cdots \cap g_nG_i^{s(g_n)}$, which is thus non-empty.  Hence $$(g_0G_i^{s(g_0)},...,g_nG_i^{s(g_n)})$$ is in $\Nc_n$ as required.
\end{proof}

We thus have shown that each of the spaces $\Nc_n$ is in the category $\mathbf{Top}_G$ of Definition \ref{topg def}.  To show that $\Nc_*$ is a semi-simplicial $G$-space, it remains to build the face maps and show that they are equivariant and \'{e}tale.

\begin{definition}\label{face map}
	For each $n\geq 1$ and each $j\in \{0,...,n\}$, define the \emph{$j^\text{th}$ face map} to be the function
	$$
	\partial^{n}_j :\Nc_n\to \Nc_{n-1},\quad (U_0,...,U_n)\mapsto (U_0,...,\widehat{U_j},...,U_n),
	$$
	where the hat ``\,$\widehat{\cdot}$\,'' means to omit the corresponding element.
\end{definition}

\begin{lemma}\label{face lh}
	For each $j$, $\partial_j$ as in Definition \ref{face map} is equivariant and \'{e}tale, and moreover 
	\[\partial^{n-1}_j \partial^n_i = \partial^{n-1}_i \partial^n_{j-1} \text{ if } i<j. \]
\end{lemma}

\begin{proof}
	Let $(U_0,...,U_n)$ be a point in $\Nc_n$, and write $U_k=g_kG_{i_k}^{s(g_k)}$ for each $k\in \{0,...,n\}$.  For each $k$, let $B_k$ be an open bisection containing $g_k$.  Then the set 
	$$
	W:=\Bigg\{(h_0G_{i_0}^{s(h_0)},...,h_nG_{i_n}^{s(h_n)})~\Big|~ h_k\in B_k \text{ for each } k \text{ and } \bigcap_{k=0}^n h_kG_{i_k}^{s(h_k)}\neq \varnothing\Bigg\}
	$$ 
	is an open neighbourhood of $(U_0,...,U_n)$ in $\Nc_n$.  We claim that $\partial^j$ restricts to a homeomorphism on $W$.  Indeed, let $W_j$ be the image of $W$ under $\partial^j$.  Then an inverse is defined by sending a point $(V_0,...,V_{n-1})$ in $W_j$ to the point $(V_0,...,hG_{i_j}^{s(h)},...,V_{n-1})$ in $W$, where $hG_{i_j}^{s(h)}$ occurs in the $j^{\text{th}}$ entry, and where $h$ is the unique point in $B_j$ so that $r(h)=p(V_0,...,V_{n-1})$.  
	
	The $G$-equivariance and the claimed relations between the face maps are straightforward.
\end{proof}

\begin{definition}\label{bounded and lebesgue}
	Let $\mathcal{C}$ be a colouring of $G$.  The \emph{homology of the colouring}, denoted $H_*(\mathcal{C})$, is the homology $H_*^G(\Nc_*)$ of the semi-simplicial $G$-space $\Nc_*$ as in Definition \ref{hom sgs}.
\end{definition}

The homology groups $H_*(\mathcal{C})$ depend strongly on the colouring.  For example, $\mathcal{C}$ could just consist of a partition of $G^{0}$ by compact open subsets, in which case one can check that the groups $H_n(\mathcal{C})$ are zero for $n>0$.  We will, however, eventually show that an appropriate limit of the homologies $H_*(\mathcal{C})$ as the colourings vary recovers the Cranic-Moerdijk-Matui homology $H_*(G)$ for principal and $\sigma$-compact $G$.

\subsection{Homology vanishing}

Our goal in this subsection is to show that if $$\mathcal{C}=\{G_0,...,G_d\}$$ is a colouring of $G$, then $H_n(\mathcal{C})=0$ for $n>d$.  We will actually establish something a little more precise than this, as it will be useful later.  The computations in this section are inspired by classical results in sheaf cohomology: see for example \cite[Section 3.8]{Godement:1958aa}.  More specifically, the precise formulas we use are adapted from \cite[\href{https://stacks.math.columbia.edu/tag/01FG}{Tag 01FG}]{stacks-project}.

Throughout this subsection, we fix an ample groupoid $G$ with compact unit space $G^0$, a colouring $\mathcal{C}=\{G_0,...,G_d\}$ as in Definition \ref{colouring}, and associated nerve space $\Nc_*$ as in Definition \ref{nerve}.

\begin{lemma}\label{overlap}
Let $\U$ be the cover associated to the colouring $\mathcal{C}$ as in Definition \ref{cover}.  Then any two elements of $\U$ that are the same colour and intersect non-trivially are the same.  In particular, if $(U_0,...,U_n)$ is a point of some $\Nc_n$, then any two elements of the same colour are actually the same. 
\end{lemma}

\begin{proof}
	Say $gG_i^{s(g)}\cap hG_i^{s(h)}\neq \varnothing$; we need to show that $gG_i^{s(g)}=  hG_i^{s(h)}$.  Indeed, there are $k_g\in G_i^{s(g)}$ and $k_h\in G_i^{s(h)}$ with $gk_g=hk_h$.  It follows that $h^{-1}g=k_hk_g^{-1}$, so $h^{-1}g$ is in $G_i$, as $G_i$ is a subgroupoid.  Hence whenever $gk$ is in $gG_{i}^{s(g)}$, we have that $gk=h(h^{-1}gk)$ is also in $hG_i^{s(h)}$, and so $gG_i^{s(g)}\subseteq  hG_i^{s(h)}$.  Hence by symmetry, $gG_i^{s(g)}=  hG_i^{s(h)}$.
\end{proof}

\begin{definition}\label{colour function}
	For each $n$, define the \emph{colour map}
	$$
	c:\mathcal{N}_n\to \{0,...,d\}^{n+1}, \quad U \mapsto (\text{colour of }U_0,...,\text{colour of }U_n).
	$$
\end{definition}
We leave it to the reader to check that $c$ is continuous and invariant under the action of $G$.  

Throughout, we identify the symmetric group $S_{n+1}$ with the permutations of $\{0,...,n\}$.  Fix now $\sigma\in S_{n+1}$ and for each $a\in \{0,...,n\}$ define a new permutation $\sigma_a$ as follows.  Let $\sigma_0$ be the identity.  For each $a\in \{1,...,n\}$, let $\sigma_a$ be the unique permutation of $\{0,...,n\}$ such that:
\begin{enumerate}[(i)]
	\item the restriction of $\sigma_a$ to $\{0,...,a-1\}$ agrees with the restriction of $\sigma$ to this set;
	\item the restriction of $\sigma_a$ to $\{a,...,n\}$ is the unique order-preserving bijection 
	$$
	\{a,....,n\}\to \{0,...,n\}\setminus \{\sigma(0),...,\sigma(a-1)\}.
	$$  
\end{enumerate}

For a point $x=(i_0,...,i_n)\in \{0,...,d\}^{n+1}$, let $\sigma^x\in S_{n+1}$ be the unique permutation determined by the conditions below:
\begin{enumerate}[(i)]
	\item $i_{\sigma^x(0)}\leq i_{\sigma^x(1)}\leq \cdots \leq i_{\sigma^x(n)}$;
	\item $\sigma^x$ is order-preserving when restricted to each subset $S$ of $\{0,...,n\}$ such that the elements $\{i_j\mid j\in S\}$ all have the same colour (i.e.\ so that the set $\{i_j\mid j\in S\}$ consists of a single element of $\{0,...,d\}$).
\end{enumerate}

\begin{definition}\label{ordered bits}
	For each $n$, let $c:\Nc_n\to \{0,...,d\}^{n+1}$ be the colour map of Definition \ref{colour function} and define 
	$$
	\Nc_n^>:=c^{-1}\big(\{(i_0,...,i_n)\mid i_0<\cdots <i_n\}\big).
	$$
	Similarly, define 
	$$
	\Nc_n^\geq:=c^{-1}\big(\{(i_0,...,i_n)\mid i_0\leq \cdots \leq i_n\}\big).
	$$
\end{definition}
We note that each $\Nc_*^>$ and $\Nc_*^\geq$ is a semi-simplicial $G$-space with the restricted structures from $\Nc_*$: indeed, each $\Nc_n^>$ and $\Nc_n^\geq$ is a closed, open and $G$-invariant subset of $\Nc_n$ as $c$ is continuous and $G$-invariant, and the face maps of Definition \ref{face lh} clearly restrict to maps $\Nc_n^>\to \Nc_n^>$ and $\Nc_n^\geq \to \Nc_n^\geq$.

For each $a\in \{0,...,n\}$ define now $h_a:\mathcal{N}_n\to \mathcal{N}_{n+1}$ by the formula 
$$
h_a(U):=(U_{\sigma^{c(U)}_a(0)},...,U_{\sigma^{c(U)}_a(a-1)},U_{\sigma^{c(U)}(a)},U_{\sigma^{c(U)}_a(a)}, U_{\sigma^{c(U)}_a(a+1)},...,U_{\sigma^{c(U)}_a(n)})
$$
(in words, we use $\sigma^{c(U)}_a$ to rearrange the order of the components of $U$, but also insert $U_{\sigma^{c(U)}(a)}$ into the $a^\text{th}$ position).  

It is not too difficult to see that each $h_a$ is an equivariant \'{e}tale map, and so induces a pushforward map $(h_a)_*:\Z[\mathcal{N}_n]\to \Z[\mathcal{N}_{n+1}]$ of $\Z[G]$-modules.  We define 
$$
h:\Z[\mathcal{N}_n]\to\Z[\mathcal{N}_{n+1}]
$$
by stipulating that for each $x\in \{0,...,d\}^{n+1}$, its restriction to each subset $\Z[c^{-1}(x)]$ equals 
$$
\sum_{a=0}^n (-1)^a\text{sign}(\sigma^x_a)(h_a)_*.
$$
On the other hand, define 
$$
p:\Z[\mathcal{N}_n]\to\Z[\mathcal{N}_n^\geq ]
$$
by stipulating that for each $x\in \{0,...,d\}^{n+1}$, its restriction to each subset $\Z[c^{-1}(x)]$ equals $\text{sign}(\sigma^x)\sigma^x_*$.  Finally, let 
$$
i:\Z[\mathcal{N}_*^\geq]\to \Z[\mathcal{N}_*]
$$
be the canonical inclusion.

\begin{lemma}\label{h homotopy}
	Let $\partial$ be the boundary map on $\Z[\mathcal{N}_*]$.   Then 
	$$
	\partial h +h\partial = \mathrm{identity} - i\circ p.
	$$
\end{lemma}

\begin{proof}
	We look at the restriction of $\partial h$ to $\Z[c^{-1}(x)]$ for some $x$; it suffices to prove the given identity for such restrictions.  For notational simplicity, let $\sigma:=\sigma^x$.   We have then for this restriction that 
	$$
	\partial h= \sum_{i=0}^{n+1} \sum_{a=0}^n (-1)^{i+a} \mathrm{sign}(\sigma_a) \partial_*^i (h_a)_*
	$$  
	We split the terms into three types.
	\begin{enumerate}[(i)]
		\item Terms of the form $\partial^i_*(h_a)_*$ where $i<a$.
		\item Terms of the form $\partial^{i+1}_*(h_a)_*$ where $i\geq a$ and $\sigma_a(i)\neq \sigma(a)$.
		\item Terms of the form $\partial^{a}_*(h_a)_*$, or of the form $\partial^{i+1}_*(h_a)_*$ where $\sigma_a(i)=\sigma(a)$.
	\end{enumerate}
	We look at each type in turn.
	\begin{enumerate}[(i)]
		\item We leave it to the reader to compute that as maps on the spatial level, $\partial^i h_a = h_{a-1}\partial^{\sigma_a(i)}$.  Moreover, $(\partial^i h_a )_*$ occurs in the sum defining $\partial h$ with sign $(-1)^i(-1)^a\text{sign}(\sigma_a)$, and $(h_{a-1}\partial^{\sigma_a(i)})_*$ occurs in the sum defining $h\partial$ with the sign $(-1)^{a-1}(-1)^{\sigma_a(i)}\text{sign}(\sigma'_{a-1})$, where $\sigma'_{a-1}$ is the permutation defined as the composition 
		$$
		\xymatrix{ \{0,...,n-1\} \ar[r]^-f & \{0,...,\widehat{i},...,n\} \ar[r]^-{\sigma_a} &  \{0,...,\widehat{\sigma_a(i)},...,n\} \ar[d]^-g \\ & & \{0,...,n-1\}}
		$$
		with $f$ and $g$ the unique order preserving bijections.  One can compute that $\text{sign}(\sigma_a)=(-1)^{\sigma_a(i)-i}\text{sign}(\sigma'_{a-1})$, essentially as $\sigma_a$ can be built from the same transpostitions as used to construct $\sigma_a'$ (conjugated by $f$ and $g$) together with a cycle of length $|\sigma_a(i)-i|+1$, which has sign $(-1)^{\sigma_a(i)-i}$.  In conclusion, the term 
		$$
		(-1)^i(-1)^a\text{sign}(\sigma_a)(\partial^i h_a)_*
		$$  
		appearing in the sum defining $\partial h$ is matched by the term 
		\begin{align*}
		(-1)^{a-1}&(-1)^{\sigma_a(i)} \text{sign}(\sigma'_{a-1})(h_{a-1}\partial^{\sigma_a(i)})_* \\ & =(-1)^{a-1}(-1)^i \text{sign}(\sigma_a) (h_{a-1}\partial^{\sigma_a(i)})_*
		\end{align*}
		appearing in the sum defining $h\partial$; as these precisely match other than having opposing signs, they cancel.
		\item We compute that as maps on the spatial level $\partial^{i+1}h_a=h_a\partial^{\sigma_a(i)}$.  Moreover, $(\partial^{i+1} h_a )_*$ occurs in the sum defining $\partial h$ with sign $(-1)^{i+1}(-1)^a\text{sign}(\sigma_a)$, and $(h_a\partial^{\sigma_a(i)})_*$ occurs in the sum defining $h\partial$ with the sign $(-1)^{a}(-1)^{\sigma_a(i)}\text{sign}(\sigma'_{a})$, where $\sigma'_{a}$ is the permutation defined as the composition 
		$$
		\xymatrix{ \{0,...,n-1\} \ar[r]^-f & \{0,...,\widehat{i},...,n\} \ar[r]^-{\sigma_a} &  \{0,...,\widehat{\sigma_a(i)},...,n\} \ar[d]^g \\&&\{0,...,n-1\}}
		$$
		with $f$ and $g$ the unique order preserving bijections.  Much as in case (i), we have that $\text{sign}(\sigma_a)=(-1)^{\sigma_a(i)-i}\text{sign}(\sigma'_{a-1})$, and can conclude that the term 
		$$
		(-1)^{i+1}(-1)^a\text{sign}(\sigma_a)(\partial^{i+1} h_a )_*
		$$
		appearing in the sum defining $\partial h$ is matched by the term 
		$$
		(-1)^{a}(-1)^{\sigma_a(i)}\text{sign}(\sigma'_{a})(h_a\partial^{\sigma_a(i)})_*=(-1)^a(-1)^i\text{sign}(\sigma_a)h_a\partial^{\sigma_a(i)})_*
		$$
		appearing in the sum defining $h\partial$; as these precisely match other than having opposite signs, they cancel.
	\end{enumerate}
	At this point, one can check that we have canceled \emph{all} the terms appearing in the sum defining $h\partial$ using terms from $\partial h$ of types (i) and (ii).  It remains to consider terms of type (iii), and show that the sum of these equals 
	$$
	\text{identity} - i\circ p
	$$
	as claimed.  For each $a$, let $i_a$ be the index such that $\sigma_a(i_a)=\sigma(a)$.  The totality of terms of type (iii) looks like 
	\begin{align*}
	&(\partial^0 h_0)_* \\ & +\sum_{a=0}^{n-1} \big((-1)^{i_a+1}(-1)^a \text{sign}(\sigma_a)(\partial^{i_a+1}h_a)_*+(-1)^{a+1}(-1)^{a+1}\text{sign}(\sigma_{a+1}) (\partial^{a+1}h_{a+1})_*\big)  \\ & \quad \quad + (-1)^n(-1)^n\text{sign}(\sigma_n)(\partial^{i_n} h_n)_*.
	\end{align*}
	The first term is the identity, and the last term is $i\circ p$ (note that $\sigma_n=\sigma$). Hence it suffices to show that each term 
	$$
	(-1)^{i_a+1}(-1)^a \text{sign}(\sigma_a)(\partial^{i_a+1}h_a)_*+(-1)^{a+1}(-1)^{a+1}\text{sign}(\sigma_{a+1}) (\partial^{a+1}h_{a+1})_*
	$$
	in the sum in the middle is zero.  Now, one computes on the spatial level that $\partial^{a+1}h_{a+1}=\partial^{i_a+1}h_a$ (this uses Lemma \ref{overlap} to conclude that two coordinates with the same colour are actually the same), and so it suffices to prove that 
	$$
	(-1)^{i_a+1}(-1)^a \text{sign}(\sigma_a)+(-1)^{a+1}(-1)^{a+1}\text{sign}(\sigma_{a+1})=0,
	$$ 
	or having simplified slightly, that $(-1)^{i_a+a+1}\text{sign}(\sigma_a)+\text{sign}(\sigma_{a+1})=0$.  Indeed, one checks that $\sigma_{a+1}$ differs from $\sigma_a$ by a cycle moving the element in the $i_a^\text{th}$ position to the $a^\text{th}$ (and keeping all other elements in the same order), and such a cycle has sign $(-1)^{a-i_a}$.  Hence $\text{sign}(\sigma_a)=\text{sign}(\sigma_{a+1})(-1)^{i_a-a}$ and we are done. \qedhere
\end{proof}

\begin{corollary}\label{semi-ord}
	The natural inclusion $i:\Z[\mathcal{N}_*^\geq]\to \Z[\mathcal{N}_*]$ is a chain homotopy equivalence.
\end{corollary}

\begin{proof}
	Lemma \ref{h homotopy} implies in particular that the map $i\circ p$ is a chain map; as $i$ is an injective chain map, this implies that $p:\Z[\mathcal{N}_n]\to\Z[\mathcal{N}_n^\geq]$ is a chain map too.  Lemma \ref{h homotopy} implies that $i\circ p$ is chain-homotopic to the identity, while $p\circ i$ just is the identity.  Hence $p$ provides an inverse to $i$ on the level of chain homotopies.
\end{proof}

Our next goal is to show that the natural inclusion $j:\Z[\mathcal{N}_*^>]\to \Z[\mathcal{N}_*^\geq]$ is again a chain homotopy equivalence.    For each $a\in \{0,...,n-1\}$, let us write $D_a^n$ for the subset of $\{0,...,d\}^{n+1}$ consisting of those tuples $(i_0,...,i_n)$ such that $i_0<i_1<\cdots i_a=i_{a+1}\leq i_{a+2}\leq \cdots \leq i_n$.  For $x\in D_a^n$, define $k^x:\Z[c^{-1}(x)]\to \Z[\mathcal{N}_{n+1}]$ by 
$$
k^x:(U_0,...,U_n)\mapsto (U_0,U_1,...,U_a,U_a,U_{a+1},U_{a+2},...,U_n).
$$
Now define $k:\Z[\mathcal{N}_n^\geq ]\to \Z[\mathcal{N}_{n+1}^\geq]$ by stipulating that for each $x\in \{0,...,d\}^{n+1}$ the restriction of $k$ to $\Z[c^{-1}(x)]$ is given by $(-1)^ak_*^x$ if $x$ is in $D_a^n$ for some $a\in \{0,...,n-1\}$, and by zero otherwise.  On the other hand, let 
$$
q:\Z[\mathcal{N}^\geq_n]\to\Z[\mathcal{N}_n^>] 
$$
be the natural projection that acts as the identity on each $\Z[c^{-1}(x)]$ with $c^{-1}(x)\subseteq \Nc_*^>$, and as $0$ otherwise.

\begin{lemma}\label{k homotopy}
	Let $\partial$ be the boundary map on $\Z[\mathcal{N}_*^\geq]$ and let $j:\Z[\mathcal{N}_*^>]\to \Z[\mathcal{N}_*^\geq]$ be the canonical inclusion.  Then 
	$$
	\partial k +k\partial = \mathrm{identity} - j\circ q.
	$$
\end{lemma}

\begin{proof}
	It suffices to check the formula for each restriction to a submodule of the form $\Z[c^{-1}(x)]$.  
	
	Say first that $x=(i_0,...,i_n)$ satisfies $i_0<\cdots <i_n$.  Then $j\circ q$ acts as the identity on $\Z[c^{-1}(x)]$, whence the right hand side is zero.  Note that $k$ restricts to zero on $\Z[c^{-1}(x)]$, whence $\partial k$ is zero.  On the other hand, the image of $\Z[c^{-1}(x)]$ is contained in a direct sum of subgroups of the form $\Z[c^{-1}(y)]$ where $y=(j_0,...,j_{n-1})\in \{0,...,d\}^n$  satisfies $j_0<\cdots <j_{n-1}$.  Hence $k\partial=0$ too, and we are done with this case.
	
	Say then that $x=(i_0,...,i_n)$ does not satisfy $i_0<\cdots <i_n$.  Then $x$ is in $D_a^n$ for some $a$.  We then compute using the assumption that $x$ is in $D_a^n$ that $\partial^ik^x$ is the identity map for $i\in \{a,a+1,a+2\}$, and therefore that 
	\begin{align}\label{partial k}
	\partial k &=\sum_{i=0}^{n+1} (-1)^i \partial^i_*k \nonumber \\
	& = \sum_{i=0}^{a-1}(-1)^a(-1)^i \partial^i_* k^x_*+(-1)^a\big((-1)^a+(-1)^{a+1}+(-1)^{a+2})\text{identity} \nonumber\\ & \quad \quad \quad \quad \quad \quad \quad \quad \quad \quad \quad \quad \quad +\sum_{i=a+3}^{n+1}(-1)^a(-1)^{i}\partial^i_* k^x_* \nonumber \\ 
	& = \sum_{i=0}^{a-1}(-1)^a(-1)^i \partial^i_* k^x_*+\text{identity}+\sum_{i=a+3}^{n+1}(-1)^a(-1)^{i}\partial^i_* k^x_*.
	\end{align}
	Looking instead at $k\partial$, note first that as $x$ is in $D_a$, we have that $\partial^a=\partial^{a+1}$ when restricted to $c^{-1}(x)$, and so 
	\begin{equation}\label{k partial}
	k\partial=k\Big(\sum_{i=0}^n (-1)^i \partial^i_*\Big)=\sum_{i=0}^{a-1} (-1)^ik\partial^i_*+\sum_{i=a+2}^n (-1)^ik\partial^i_*.
	\end{equation}
	Now, for $y\in \{0,...,d\}^{n+1}$, let $\partial^i y$ denote $y$ with the $i^\text{th}$ component removed.  Then for $0\leq i\leq a-1$, $\partial^ix\in D_{a-1}^{n-1}$, and so $k\partial_i^*=(-1)^{a-1}k^y_*\partial^i_*$.  It is not difficult to prove that for such $y$, $k^y_*\partial^i_*=\partial_i^*k^x_*$.  On the other hand, for $a+2\leq i\leq n$, $\partial^i_*x$ is in $D^{n-1}_a$, and so $k\partial^i_*=(-1)^ak^y_*\partial^i_*$; one computes moreover that for such $i$, $k^y_*\partial^i_*=\partial^{i+1}_*k^x_*$.  Putting this discussion together with the formula in line \eqref{k partial} gives that
	\begin{align*}
	k\partial & =\sum_{i=0}^{a-1} (-1)^i(-1)^{a-1}\partial^i_*k^x_*+\sum_{i=a+2}^n (-1)^i(-1)^a\partial^{i+1}_*k^x_* \\ 
	&=\sum_{i=0}^{a-1} (-1)^i(-1)^{a-1}\partial^i_*k^x_*+\sum_{i=a+3}^{n+1} (-1)^{i-1}(-1)^a\partial^{i}_*k^x_*.
	\end{align*}
	Comparing this with the formula in line \eqref{partial k}, it follows that $k\partial +\partial k$ restricts to the identity on this summand $\Z[c^{-1}(x)]$.  On the other hand, the same is true for the right hand side `$\text{identity} - j\circ q$', so we are done.
\end{proof}

\begin{corollary}\label{final cor}
	The inclusion $j:\Z[\mathcal{N}_*^>]\to \Z[\mathcal{N}_*^\geq]$ is a chain homotopy equivalence.
\end{corollary}

\begin{proof}
	Lemma \ref{k homotopy} implies that the natural projection $q$ is an inverse on the level of chain homotopies.
\end{proof}

\begin{theorem}\label{vanish}
	Let $\mathcal{N}_*$ be a nerve complex built from a colouring $\mathcal{C}=(G_0,...,G_{d})$.  Then the canonical inclusion $\mathcal{N}_*^>\to \mathcal{N}_*$ induces an isomorphism $H_*(\mathcal{C})\cong H(\mathcal{N}^>_*)$.  In particular, $H_n(\mathcal{C})=0$ for $n>d$.
\end{theorem}

\begin{proof}
	Corollaries \ref{semi-ord} and \ref{final cor} imply that the inclusion $\Z[\mathcal{N}^>_*]\to \Z[\mathcal{N}_*]$ is a chain homotopy equivalence.  Functoriality of taking coinvariants then implies that the naturally induced map $\Z[\mathcal{N}^>_*]_G\to \Z[\mathcal{N}_*]_G$ is a chain homotopy equivalence, so in particular induces an isomorphism on homology.  The remaining statement follows as $\mathcal{N}_n^>$ is empty for $n>d$.
\end{proof}

\subsection{Maps between nerves and $G$}

In this subsection, we build maps between our nerve spaces $\Nc_0$ and $G$ that induce maps between the higher nerve spaces $\mathcal{N}_n$ and the spaces $EG_n$ we used to define groupoid homology.  This will allow us to compare the homology of colourings of $G$ to the homology of $G$.

\begin{definition}\label{leb bound}
	Let $G_0,...,G_d$ be a colouring of $G$ as in Definition \ref{colouring}, and let $K$ be a subset of $G$.  We say that the colouring $G_0,...,G_d$ is:
	\begin{enumerate}[(i)]
		\item \emph{$K$-bounded} if every $G_i$ is contained in $K$;
		\item \emph{$K$-Lebesgue} if for every $x\in G^0$, there exists $G_i$ such that $G^x\cap K$ is contained in $G_i$.
	\end{enumerate}
\end{definition}

For $n\geq 0$, recall that $EG_{n}=\{(g_0,...,g_{n})\in G^{n+1}\mid r(g_0)=...=r(g_n)\}$, equipped with the subspace topology that it inherits from $G^{n+1}$.  For $n\geq 1$ and any subset $K$ of $G$, let $EG_n^K$ denote the subspace of $EG_n$ consisting of those tuples $(g_0,...,g_n)$ such that $g_i^{-1}g_j\in K$ for all $i,j$.  Let $EG_{0}^K$ be just $EG_{0}=G$, whatever $K$ is.

\begin{lemma}\label{points to sets}
	Say $K$ is a compact open subset of $G$, and that $G_0,...,G_d$ is a $K$-Lebesgue colouring of $G$ with associated nerve $\Nc_*$.  Then there exists an equivariant \'etale map $\Phi_0:G\to \Nc_0$ such that:
	\begin{enumerate}[(i)]
		\item $\Phi_0(g)\supseteq gK$ for all $g\in G$;
		\item for all $n$, the map 
		$$
		\Phi_n:EG_{n}^K\to \Nc_n,\quad (g_0,...,g_n)\mapsto (\Phi_0(g_0),...,\Phi_0(g_n))
		$$
		is a well-defined, equivariant local homeomorphism.
	\end{enumerate}
\end{lemma}

\begin{proof}
	For each $i\in \{0,...,d\}$, define $V_i:=\{x\in G^{0}\mid G^x\cap K\subseteq G_i\}$.  We claim that each $V_i$ is open.  Indeed, let $x\in V_i$.  As the set $G^x\cap K$ is compact and discrete, it is finite.  Write $g_1,...,g_n$ for the elements of this set, which are all in $G_i$ by assumption that $x$ is in $V_i$.  For each $j\in \{1,...,n\}$, let $W_j\subseteq G_i$ be a compact open bisection containing $g_j$.  We may assume the $W_j$ are disjoint by shrinking them if necessary.  Define $W:=\bigcap_{j=1}^n r(W_j)$, which is a compact open neighbourhood of $x$.  We may write the compact open set $r^{-1}(W)\cap K$ as a finite disjoint union of compact open bisections of the form $W_1\cap r^{-1}(W),...,W_n\cap r^{-1}(W),B_1,...,B_m$.  Note that no $B_j$ can intersect $G^x\cap K$, whence none of the sets $r(B_j)$ can contain $x$.  Define 
$$
V:=W\setminus \Big(\bigcup_{j=1}^m B_j\Big).
$$
This is an open set containing $x$.  Moreover, $r^{-1}(V)$ is contained in $W_1\cup\cdots \cup W_n$, and therefore in $G_i$.  Hence $V$ is an open neighborhood of $x$ contained in $V_i$, so $V_i$ is open as claimed.
	
Note now that $V_0,...,V_d$ covers $G^{0}$ by the assumption that the underlying colouring is $K$-Lebesgue.  As $G^0$ has a basis of compact open sets and is compact, there is a finite cover, say $\mathcal{U}$ of $G^0$ consisting of disjoint compact open sets, and such that each $U\in \mathcal{U}$ is contained in some $V_i$.  Define $E_i$ to be the union of those $U\in \mathcal{U}$ such that $i$ is the smallest element of $\{0,...,d\}$ with $U$ contained in $V_i$.  Then the sets $E_0,...,E_d$ are a partition of $G^{0}$ by compact open subsets, and each $E_i$ is contained in $V_i$.   
	
	Now, for each $g\in G$, let $i(g)\in \{0,...,d\}$ be the unique $i$ such that $s(g)$ is in $E_i$; as the partition $G^{0}=\bigsqcup_{i=0}^n E_d$ is into clopen sets, the map $i:G\to \{0,...d\}$ this defines is continuous.  Define 
	$$
	\Phi_0:G\to \Nc_0,\quad g\mapsto gG_{i(g)}^{s(g)}.
	$$
	We claim this has the right properties.  First, note that for $g\in G$, $G^{s(g)}\cap K\subseteq G_{i(g)}$ by definition of $i(g)$ and the cover $V_0,...,V_d$.  Hence 
	$$
	gK=g(G^{s(g)}\cap K)\subseteq gG_{i(g)}=gG_{i(g)}^{s(g)} =\Phi_0(g).
	$$
	We now show that $\Phi_0$ is \'etale.  Continuity of the restriction of $\Phi_0$ to each set $s^{-1}(E_i)$ follows from the definition of the topology of $\Nc_0$, and continuity of $\Phi_0$ on all of $G$ follows from this as the sets $s^{-1}(E_0),...,s^{-1}(E_d)$ are a closed partition of $G$.  Let now $g\in G$, and let $B$ be a clopen bisection containing $g$ such that the map $i:G\to \{0,...,d\}$ is constant on $B$ (such exists as $i$ is continuous).  Let $C=\Phi_0(B)$.  Then the map $C\to B$ defined by sending $U$ to the unique element of $B\cap r(U)$ is a well-defined continuous inverse to the restriction $\Phi_0|_B$, completing the proof that $\Phi_0$ is \'etale.
	
	Equivariance of $\Phi_0$ follows as if $s(g)=r(h)$, then $s(gh)=s(h)$ and $i(gh)=i(h)$, whence
	$$
	\Phi_0(gh)=ghG_{i(gh)}^{s(gh)}=ghG_{i(h)}^{s(h)}=g\Phi_0(h).  
	$$
	
	To see that $\Phi_n$ is well-defined, note that if $(g_0,...,g_n)\in EG_{n}^K$, then $g_i^{-1}g_j$ is in $K$ for all $i,j$.  Hence in particular $g_0$ is in $g_iK\subseteq \Phi_0(g_i)$ for each $i$, and so $g_0$ is in $\Phi_0(g_0)\cap \cdots \cap \Phi_0(g_n)$, and so this set is non-empty.  Hence $(\Phi_0(g_0),...,\Phi_0(g_n))$ is a well-defined element of $\Nc_n$.  Equivariance of $\Phi_n$ and the fact that it is \'etale are straightforward from the corresponding properties for $\Phi_0$, so we are done.
\end{proof}

\begin{lemma}\label{sets to points}
	Assume that $G$ is principal, and $K$ is a compact open subset of $G$ that contains $G^{0}$.  Let $G_0,...,G_d$ be a $K$-bounded colouring with associated nerve $\Nc_*$.  Then there exists an equivariant \'etale map $\Psi_0:\Nc_0\to G$ with the following properties:
	\begin{enumerate}[(i)]
		\item  $\Psi_0(U)\in U$ for all $U\in \Nc_0$;
		\item for all $n$, the map 
		$$
		\Psi_n:\Nc_n\to EG_{n}^{KK^{-1}},\quad (U_0,...,U_n)\mapsto (\Psi_0(U_0),...,\Psi_0(U_n))
		$$
		is a well-defined, equivariant local homeomorphism.
	\end{enumerate}
\end{lemma}

To prove this, we need an ancillary lemma, which is based on the following structural result from \cite[Lemma 3.4]{Giordano:2003aa}.

\begin{lemma}\label{com pri structure}
Let $H$ be a compact, ample, principal groupoid.  Then there are $m\in \N$ and
\begin{enumerate}[(i)]
\item disjoint clopen subgroupoids $H_1,...,H_m$ of $H$,
\item clopen subsets $X_1,...,X_m$ of $H^0$ (equipped with the induced, i.e.\ trivial, groupoid structure), and
\item finite pair groupoids $P_1,...,P_m$,
\end{enumerate}
such that $H$ identifies with the disjoint union $H=\bigsqcup_{i=k}^m H_k$ as a topological groupoid, and such that each $H_k$ is isomorphic as a topological groupoid to $X_k\times P_k$. \qed
\end{lemma}

\begin{corollary}\label{split lem}
Let $H$ be a compact, ample, principal groupoid, and let $H^{0}/H$ be the quotient space of $H^{0}$ by the equivalence relation induced by $H$: precisely $x\sim y$ if there is $h\in H$ with $s(h)=x$ and $r(h)=y$.  Then $H^{0}/H$ is Hausdorff and there are \'etale maps $\sigma:H^0/H\to H^0$ and $\tau:H^{0}\to H$ such that:
	\begin{enumerate}[(i)]
		\item $\sigma$ splits the quotient map $\pi:H^{0}\to H^{0}/H$ (so in particular, $\pi$ is \'etale);
		\item $r\circ \tau=\text{identity}$ and and $s\circ \tau=\sigma\circ\pi$.  
	\end{enumerate} 
\end{corollary}

\begin{proof}
Assume first that $H=X\times P$, where $X$ is a compact trivial groupoid and $P$ is the pair groupoid on some finite set $\{0,...,n\}$.  Then $H^0/H$ identifies homeomorphically with $X$ (so in particular is Hausdorff) via the map 
$$
H^0\to X, \quad (x,(i,i))\mapsto x.
$$
Making this identification, we may define $\sigma(x)=(x,(0,0))$ and $\tau(x,(i,i))=(x,(i,0))$.  These maps have the right properties when $H=X\times P$.  

In the general case, Lemma \ref{com pri structure} gives a decomposition of $H$ into groupoids of the form $X\times P$ as above, and we may build $\sigma$ and $\tau$ on each separately using the method above.
\end{proof}

\begin{proof}[Proof of Lemma \ref{sets to points}]
	Let $\pi_i$, $\sigma_i$ and $\tau_i$ be as in Corollary \ref{split lem} for $H=G_i$.  Define 
	$$
	\Psi_0:\Nc_0\to G,\quad gG_i^{s(g)}\mapsto g\tau_i(s(g)).
	$$
	We first check that this is well-defined.  Indeed, if $hG_i^{s(h)}=gG_i^{s(g)}$, then $r(h)=r(g)$ and $h^{-1}g\in G_i$.  Hence 
	$$
	\pi_i(s(g))=\pi_i(s(h^{-1}g))=\pi_i(r(h^{-1}g))=\pi_i(s(h)),
	$$
	and so $\sigma_i(\pi_i(s(g)))$ and $\sigma_i(\pi_i(s(h)))$ are the same.  As $\tau_i(x)$ has source $\sigma_i(\pi_i(x))$ for all $x\in G_i^{(0)}$, this implies that both $\tau_i(s(g))$ and $\tau_i(s(h))$ have the same source.  As moreover $g$ and $h$ have the same range, the elements $g\tau_i(s(g))$ and $h\tau_i(s(h))$ of $G$ have the same source and range and are therefore the same as $G$ is principal.   Having seen that $\Psi_0$ is well-defined, equivariance of $\Psi_0$ is straightforward.  The fact that $\Psi_0(U)\in U$ for all $U\in \Nc_0$ follows as if we write $U=gG_i^{s(g)}$, then $\Psi_0(U)=g\tau_i(s(g))$, and $g\tau_i(s(g))$ is in $G_i^{s(g)}$, as $\tau_i(s(g))$ is in $G_i^{s(g)}$.
	
	To see that $\Psi_0$ is \'etale, let $gG_i^{s(g)}$ be an element of $\Nc_0$, and  and let $B$ be a clopen bisection of $g$ in $G$ such that the set $\{hG_i^{s(h)}\mid h\in B\}$ is a clopen neighbourhood of $gG_i^{s(g)}$ in $\Nc_0$; in particular, this implies that $s(h)\in G_i^{0}$ for all $h\in B$.  Using that both $s$ and $\tau_i$ are \'etale, we have that $s(\tau_i(B))$ is open, and therefore that $\Psi_0(B)$ is open.  We claim that the map 
	$$
	\kappa:\Psi_0(B)\to \Nc_0,\quad h\mapsto hG_i^{s(h)}
	$$
	is a local inverse to $\Psi_0$; as it is continuous, this will suffice to complete the proof.  Indeed, for any $h\in B$, 
	$$
	\kappa(\Psi_0(hG_i^{s(h)}))=h\tau_i(s(h))G_i^{s(\tau_i(s(h)))}=hG_i^{s(h)}.
	$$
	On the other hand, for $h\in \Psi_0(B)$, as $h$ is in the image of $\Psi_0(B)$, we have that $s(h)$ is in the image of $\sigma_i$, and therefore that $\sigma_i(\pi_i(s(h)))=s(h)$, and so $\tau_i(s(h))=s(h)$.  Hence 
	$$
	\Psi_0(\kappa(h))=h\tau_i(s(h))=h
	$$
	and we are done with showing that $\Psi_0$ is \'etale.

	To see that $\Psi_n$ is well-defined, we need to check that if $(U_0,...,U_n)$ is in $\Nc_n$, then $(\Psi_0(U_0),...,\Psi_0(U_n))$ is in $EG_{n}^{KK^{-1}}$.  Write $g_j:=\Psi_0(U_j)$ for notational simplicity, so $g_j$ is in $U_j$ by the properties of $\Psi_0$.  Let $h$ be an element of $U_0\cap \cdots \cap U_n$.  Then for all $j$, the fact that the colouring is $K$-bounded implies that $g^{-1}_jh$ is in $K$ for each $j$.  Hence for any $i,j$, $g_i^{-1}g_j=g^{-1}hh^{-1}g_j\in KK^{-1}$, completing the proof that $\Psi_n$ is well-defined.  The facts that $\Psi_n$ is \'etale and equivariant follow from the corresponding properties for $\Psi_0$, so we are done.
\end{proof}

\subsection{Anti-\v{C}ech homology}

In this subsection, we show that the Crainic-Moerdijk-Matui homology groups $H_*(G)$ can be realised by a direct limit of homology groups of appropriate colourings. 

The key definition is as follows.

\begin{definition}\label{anticech}
	An \emph{anti-\v{C}ech sequence} for $G$ consists of the following data:
	\begin{enumerate}[(i)]
		\item a sequence $\mathcal{C}_m:=(G_{0}^{(m)},...,G_{d_m}^{(m)})_{m=0}^\infty$ of colourings of $G$ with associated sequence of nerves $\Nc_*^{(m)}$;
		\item for each $m$ a morphism $\iota^{(m)}:\Nc^{(m-1)}_*\to \Nc^{(m)}_*$ of semi-simplicial $G$-spaces such that for all $U\in \Nc^{(m-1)}_0$ we have that $\iota^{(m)}(U)\supseteq U$, and moreover so that for any compact open subset $K$ of $G$, there exists $m_K$ such that for all $m\geq m_K$ and all $U\in \Nc^{(m-1)}_0$, we have that $\iota^{(m)}(U)\supseteq UK$.
	\end{enumerate}
\end{definition}

\begin{definition}\label{ac hom}
	Let $\mathcal{A}=(\mathcal{C}_m)_{m=1}^\infty$ be an anti-\v{C}ech sequence for $G$, with associated sequence of morphisms $\iota^{(m)}:\Nc^{(m-1)}_*\to \Nc^{(m)}_*$.  We define the \emph{homology} of $\mathcal{A}$, denoted $H_*(\mathcal{A})$, to be the corresponding direct limit of the sequence of maps $(\iota^{(m)}_*:H_*(\mathcal{C}_{m-1})\to H_*(\mathcal{C}_m))_{m=1}^\infty$.
\end{definition}

Anti-\v{C}ech sequences always exist under the assumptions that $G$ is principal and $\sigma$-compact. This follows from the next three lemmas.

\begin{lemma}\label{sets to sets}
Assume that $G$ is principal, and that $K$ is a compact open subset of $G$ that contains $G^{0}$ and that satisfies $K=K^{-1}$.    Let $G_0,...,G_d$ and $H_0,...,H_e$ be colourings of $G$ with associated nerves $\Nc_*$ and $\M_*$ respectively.  Assume moreover that the colouring $G_0,...,G_d$ is $K$-bounded, and that the colouring $H_0,...,H_e$ is $K^3$-Lebesgue.
	
	Then there exists an equivariant \'etale map $\iota_0:\Nc_0\to \M_0$ with the following properties:
	\begin{enumerate}[(i)]
		\item  $\iota_0(U)\supseteq UK$ for all $U\in \Nc_0$;
		\item for all $n$, the map 
		$$
		\iota_n:\Nc_n\to \M_n,\quad (U_0,...,U_n)\mapsto (\iota_0(U_0),...,\iota_0(U_n))
		$$
		is, well-defined, equivariant, and \'etale.
	\end{enumerate}
\end{lemma}

\begin{proof}
	Using Lemma \ref{sets to points}, there is an equivariant \'etale map $\Psi_0:\Nc_0\to G$ such that $\Psi_0(U)\in U$ for all $U\in \Nc_0$, and such that for all $n$, the prescription
	$$
	\Psi_n:\Nc_n\to EG_{n+1}^{K^3},\quad (U_0,...,U_n)\mapsto (\Psi_0(U_0),...,\Psi_0(U_n))
	$$
	gives gives a well-defined equivariant \'etale map (Lemma \ref{sets to points} has $KK^{-1}$ in place of $K^3$, but note that our assumptions imply that $K^3$ contains $KK^{-1}$).  Using Lemma \ref{points to sets}, there is an equivariant local homeomorphism $\Phi_0:G\to \M_0$ such that $gK^3\subseteq \Phi_0(g)$ for all $g\in G$, and so that the prescription 
	$$
	\Phi_n:EG_{n}^{K^3}\to \Nc_n,\quad (g_0,...,g_n)\mapsto (\Phi_0(g_0),...,\Phi_0(g_n))
	$$
	is a well-defined equivariant local homeomorphism.  
	
	Define then $\iota_0:=\Phi_0\circ \Psi_0$, and note that $\iota_n:=\Phi_n\circ \Psi_n$ for all $n$.  Each $\iota_n$ is then a well-defined equivariant \'etale map.  Moreover, fix $U\in \Nc_0$ and write $U=gG_i^{s(g)}$.  As $\Psi_0(U)\in U$, we may write $\Psi_0(U)=gh$ for some $h\in G_i^{s(g)}$.  Then $h^{-1}$ is in $G_i$, and so in $K$ as each $G_i$ is a subset of $K$.  Hence for an arbitrary element $gk$ of $gG_i^{s(g)}$ with $k\in G_i\subseteq K^{-1}$, we have that $gk=ghh^{-1}k\in \Psi_0(U)K^2$.   As $gk$ was an arbitrary element of $U$, this gives that $U\subseteq \Psi_0(U)K^2$, and so $UK\subseteq \Psi_0(U)K^3$.  Using the properties of  $\Psi_0$ and $\Phi_0$, we thus get that
	$$
	UK\subseteq \Psi_0(U)K^3\subseteq \Phi_0(\Psi_0(U))=\iota_0(U)
	$$ 
	and are done.
\end{proof}

\begin{lemma}\label{comb lem}
Assume that $G$ is principal and that $K$ is a compact open subset of $G$ that contains $G^{0}$.  Then for any $x\in G^{0}$ there is a compact open subset $B_0$ of $G^{0}$ containing $x$ and open bisections $B_1,...,B_n$ such that:
	\begin{enumerate}[(i)]
		\item $r^{-1}(B_0)\cap K=\bigsqcup_{i=0}^n B_i$;
		\item for each $i$, $r|_{B_i}:B_i\to B_0$ is a homeomorphism;
		\item for each $i\neq j$, $s(B_i)\cap s(B_j)=\varnothing$.
	\end{enumerate}
\end{lemma}

\begin{proof}
	Say the elements of $r^{-1}(x)\cap K$ are $g_0=x,g_1,...,g_n$.  As $G$ is principal, we have $s(g_i)\neq s(g_j)$ for all $i\neq j$.  Hence for each $g_i$, we may choose a clopen bisection $D_i$ containing $g_i$ such that $s(D_i)\cap s(D_j)=\varnothing$ for $i\neq j$, so that $r|_{D_i}$ is a homeomorphism, and so that $D_0$ is contained in $G^{0}$.  
	
	Set $C_0:=\bigcap_{i=0}^n r(B_i)$, which is a clopen set containing $x$, and for each $i$, set $C_i:=B_i\cap r^{-1}(C_0)$, which is a clopen set containing $g_i$.  The set $r^{-1}(C_0)\cap K$ is compact.  We may thus write it as 
	$$
	\bigsqcup_{i=0}^n C_i\sqcup \bigsqcup_{j=0}^m E_j,
	$$
	where each $E_j$ is a clopen bisection such that $r|_{E_j}$ is a homeomorphism, and so that each $E_j$ does not intersect $r^{-1}(x_0)$.  Set 
	$$
	B_0:=C_0\setminus \bigcup_{j=0}^m r(E_j),
	$$
	and set $B_i:=r^{-1}(B_0)\cap C_i$.  These sets have the right properties. 
\end{proof}

\begin{lemma}\label{cover lem}
	Let $G$ be principal, and assume that $K$ is a compact open subset of $G$ that contains $G^{0}$.  Then there exists a $K$-Lebesgue colouring $G_0,...,G_d$ for $G$.
\end{lemma}

\begin{proof}
	Fix $x\in X$, and let $B_0,...,B_n$ be a collection of sets with the properties in Lemma \ref{comb lem}.  For each $i$, let $\rho_i:B_0\to B_i$ be the inverse of $r|_{B_i}$.  Let $P=\{0,...,n\}^2$ be the pair groupoid on the set $\{0,...,n\}$, and define 
	$$
	f:B_0\times P\to G,\quad (x,(i,j))\mapsto \rho_i(x)\rho_j(x)^{-1}.
	$$
	It is not difficult to check that $f$ is a homeomorphism onto its image, which is a compact open subgroupoid of $G$.  Write $G_x$ for the image of $f$.  Moreover, by construction we have that for every $y\in G_x^{0}$, the set $r^{-1}(y)\cap K$ is contained in $G_x$.  
	
	The collection $\{G_x^{0}\mid x\in G^{0}\}$ is an open cover of $G^{0}$, and thus has a finite subcover.  Let $G_0,...,G_d$ be the collection of compact open subgroupoids of $G$ whose base spaces appear in this subcover.  This collection has the right properties.
\end{proof}

\begin{corollary}\label{ac exist}
	For any $\sigma$-compact principal $G$ with compact base space, an anti-\v{C}ech sequence exists.
\end{corollary}

\begin{proof}
	As $G$ is $\sigma$-compact, there is a sequence $L_0\subseteq L_1\subseteq \cdots $ of compact open subsets of $G$ such that each $L_n$ equals $L_n^{-1}$ and contains $G^{0}$, and such that any compact subset of $G$ is eventually contained in all of the $L_n$.  Set $K_0=L_0$.  Lemma \ref{cover lem} implies that there is a $K_0^3$-Lebesgue collection of compact open subgroupoids of $G$, say $G_0,...,G_d$, with associated nerve space $\Nc_*^{(0)}$.  As this collection is finite, there exists some compact open subset $M_0$ of $G$ that contains all of $G_0,...,G_d$, and that satisfies $M_0=M_0^{-1}$.  Set $K_1:=K_0\cup L_1\cup M_0$.   Lemma \ref{cover lem} gives a new colouring that is $K_1^3$-Lebesgue with associated nerve $\Nc_*^{(1)}$, and Lemma \ref{sets to sets} gives a morphism of semi-simplicial étale $G$-spaces $\iota^{(1)}:\Nc_*^{(0)}\to \Nc_*^{(1)}$ with the properties there.  Now let $M_1$ be a compact open subset of $G$ such that $M_1=M_1^{-1}$, and that contains all the groupoids from this new colouring.  Set $K_2:=K_1\cup L_2\cup M_1$ and use Lemma \ref{cover lem} to build a $K_2^3$-Lebesgue covering, and Lemma \ref{sets to sets} to build a map $\iota^{(2)}$ from $\Nc_*^{(1)}$ to the associated nerve $\Nc_*^{(2)}$ with the properties in that lemma.  Iterating this process builds an anti-\v{C}ech sequence as desired.
\end{proof}

Our main goal in this subsection is to prove the following theorem.

\begin{theorem}\label{main comparison}
	Let $G$ be principal and $\sigma$-compact, with compact base space.  Let $\mathcal{A}$ be an anti-\v{C}ech sequence for $G$.  Then the homology groups $H_*(\mathcal{A})$ and $H_*(G)$ are isomorphic.
\end{theorem}

The proof will proceed by some lemmas.  First, we need a definition.

\begin{definition}\label{close}
	Let $C_*$ be a semi-simplicial $G$-space and $D_*$ be either the nerve of a colouring $\Nc_*$, or one of the $EG_*^L$ for some compact open subset $L$ of $G$. Let $\alpha,\beta:C_*\to D_*$ be two morphisms of semi-simplicial $G$-spaces.  Then $\alpha,\beta$ are \emph{close} if there exists a compact open subset $K$ of $G$, either:
	\begin{enumerate}[(i)]
		\item $D_*$ is a nerve $\Nc_*$, and for all $x\in C_0$ there exists $g\in G$ such that $\alpha(x)$ and $\beta(x)$ are both subsets of $gK$;
		\item $D_*$ is of for the form $EG_*^L$, and for all $x\in C_0$, $\alpha(x)^{-1}\beta(x)$ is in $K$.
	\end{enumerate} 
\end{definition}

\begin{lemma}\label{close mor}
Let $(\Nc_*^{(m)}, \iota^{(m)})$ be an anti-\v{C}ech sequence, and let $\alpha,\beta:C_*\to \Nc_*^{(m)}$ be close morphisms for some $m$.  Let 
	$$
	\iota^{(\infty)}:H_*(\Nc^{(m)})\to  \varinjlim H_*(\Nc^{(l)})
	$$ 
	be the natural map to the direct limit, i.e.\ to the homology of the anti-\v{C}ech sequence.  Then the compositions 
	$$
	\iota^{(\infty)}\circ \alpha_*:H_*(C)\to \varinjlim H_*(\Nc^{(l)}) \quad \text{and}\quad \iota^{(\infty)}\circ \beta_*:H_*(C)\to \varinjlim H_*(\Nc^{(l)})
	$$
	are the same.
\end{lemma}

\begin{proof}
	Let $K$ be as in the definition of closeness for $\alpha$ and $\beta$, and assume also that $K$ is so large that the colouring underlying $\Nc^{(m)}$ is $K$-bounded.  Let $l\geq m$ be large enough so that if $\iota:\Nc^{(m)}\to \Nc^{(l)}$ is the composition of the morphisms in the definition of the anti-\v{C}ech sequence, then for all $U\in \Nc^{(m)}$, we have that $\iota(U)\supseteq UK^2$ (such an $l$ exists by definition of an anti-\v{C}ech sequence). It will suffice to show that $\iota\circ\alpha$ and $\iota\circ \beta$ induce the same map $H_*(C)\to H_*(\Nc^{(l)})$.
	
	Let now $x$ be a point in $C_n$ for some $n$.  For each $j\in \{0,...,n\}$ let $\pi_j:C_n\to C_0$ be the map corresponding under the semi-simplicial structure to the map $\{0\}\to \{0,...,n\}$ that sends $0$ to $j$ (see Section \ref{Gsphom} for notation).  Define $x_j:=\pi_j(x)$.  We claim that the intersection 
	\begin{equation}\label{big int}
	\bigcap_{j=0}^n \iota(\alpha(x_j))\cap \bigcap_{j=0}^n \iota(\beta(x_j))
	\end{equation}
	is non-empty.  Indeed, $(\alpha(x_0),...,\alpha(x_n))$ is a point of $\Nc^{(l)}_n$, whence there is some $g_\alpha$ in the intersection $\bigcap_{j=0}^n \alpha(x_j)$, and similarly for $g_\beta$ with $\beta$ replacing $\alpha$.  As $\alpha$ and $\beta$ are close with respect to $K$, we have that $\alpha(x_0)$ and $\beta(x_0)$ are both subsets of $gK$ for some $g\in G$, whence there are $k_\alpha$ and $k_\beta$ in $K$ such that $g_\alpha=gk_\alpha$ and $g_\beta=gk_\beta$.  Hence $g_\alpha=g_\beta k_\beta k_\alpha^{-1}$, so in particular $g_\alpha$ is in $g_\beta K^2$.  Now, by choice of $\iota$, $\iota(\beta(x_j))\supseteq \beta(x_j)K^2$ for all $j$, whence $g_\alpha$ is in $\iota(\beta(x_j))$ for all $j$.  Moreover, $g_\alpha\in \alpha(x_j)\subseteq \iota(\alpha(x_j))$ for all $j$, so $g_\alpha$ is a point in the claimed intersection.
	
	For each $n$ and each $i\in \{0,...,n\}$, we define a map  
	$$
	h^i:C_n\to \Nc^{(l)}_{n+1}
	$$
	by the formula 
	$$
	x\mapsto (\iota\circ \alpha(\pi_0(x)),...,\iota\circ\alpha(\pi_i(x)),\iota\circ \beta(\pi_i(x)),...,\iota\circ \beta(\pi_n(x))),
	$$
	which is well-defined by the claim.  It is moreover an equivariant local homeomorphism as $\iota$, $\alpha$ and $\beta$ all have these properties.  Hence $h^i$ induces a map $h^i_*:\Z[C_n]\to \Z[\Nc^{(l)}_{n+1}]$ in the usual way.  We define 
	$$
	h:=\sum_{i=0}^n (-1)^i h^i_*:\Z[C_n]\to \Z[\Nc_{n+1}^{(l)}].
	$$
	Direct checks show that $h$ (and the map induced on coinvariants by $h$) is a chain homotopy between the maps induced by $\iota\circ \alpha$ and $\iota\circ \beta$.  Hence $\iota\circ \alpha$ and $\iota\circ \beta$ indeed induce the same map on homology as claimed. 
\end{proof}

For the next lemma, let $K_0\subseteq K_1\subseteq K_2...$ be a sequence of compact open subsets of $G$, all of which contain $G^{0}$, and whose union is $G$.  We then get a sequence $(EG_*^{K_m})_{m=0}^\infty$ of spaces.  Note moreover that the corresponding limit ${\displaystyle \lim_\to H_*(G^{K_n})}$ canonically identifies with $H_*(G)$: indeed, this follows directly from the observation that for each $n$, $EG_{n}$ is the increasing union of the $EG^{K_m}_{n}$, and the fact that taking homology groups commutes with direct limits.

\begin{lemma}\label{close mor 2}
Let $(G_*^{K_m})_{m=0}^\infty$ be a sequence of spaces associated to a nested sequence of compact open subsets of $G$ as above.  Let $\alpha,\beta:C_*\to G_*^{(K_m)}$ be close morphisms for some $m$.  Let 
	$$
	\kappa:H_*(G^{K_m})\to \varinjlim H_*(G^{K_m})
	$$
	be the natural map to the direct limit, i.e.\ to the homology $H_*(G)$ of $G$.  Then the compositions 
	$$
	\kappa\circ \alpha_*:H_*(C)\to H_*(G)\quad  \text{and}\quad \kappa\circ \beta_*:H_*(C)\to H_*(G)
	$$
	are the same.
\end{lemma}

\begin{proof}
	The proof is very similar to that of Lemma \ref{close mor}.  We leave the details to the reader.
\end{proof}

\begin{proof}[Proof of Theorem \ref{main comparison}]
	Let $\mathcal{A}$ be the given anti-\v{C}ech sequence with associated nerves and morphisms $\iota^{(m)}:\Nc^{(m-1)}_*\to \Nc^{(m)}$.  
	
	Let $m_1=1$.  Then the colouring underlying $\Nc^{(1)}_*$ is $K$-bounded for some compact open subset $K$ of $G$, which we may assume contains $G^{0}$, and that satisfies $K=K^{-1}$.  Set $K_1:=K^2$.  Then Lemma \ref{sets to points} gives a morphism $\Psi^{(1)}:\Nc^{(m_1)}_*\to G^{K_1}_*$.  On the other hand, by definition of an anti-\v{C}ech sequence there is $m_2>m_1$ such that $\Nc^{(m_2)}$ is $K_1$-Lebesgue, whence Lemma \ref{points to sets} gives a morphism $\Phi^{(1)}:G^{K_1}_*\to \Nc^{(m_2)}$.  Continuing, the colouring underlying $\Nc^{(m_2)}$ is $K$-bounded for some compact open subset $K$ of $G$ that we may assume contains $K_1$ and satisfies $K=K^{-1}$.  Set $K_2:=K^2$, so Lemma \ref{sets to points} gives a morphism $\Phi^{(2)}:\Nc^{(m_2)}_*\to G^{K_2}_*$.
	
	Continuing in this way, we get sequences $1=m_1<m_2<m_3<\cdots$ of natural numbers and $K_1\subseteq K_2\subseteq \cdots $ of compact open subsets of $G$ together with morphisms
	$$
	\xymatrix{ \Nc^{(m_1)}_* \ar[dr]^-{\Psi^{(1)}} & & \Nc^{(m_2)}_* \ar[dr]^-{\Psi^{(2)}} & & \Nc^{(m_3)}_* \ar[dr]^-{\Psi^{(3)}} & \cdots \\
		& EG^{K_1}_* \ar[ur]_-{\Phi^{(1)}} & & EG^{K_2}_* \ar[ur]_-{\Phi^{(2)}} & & & \cdots}.
	$$
	We may fill in horizontal arrows in the diagram: on the top row, these should be appropriate compositions of the morphisms $\iota^{(m)}$ coming from the definition of an anti-\v{C}ech sequence, while on the bottom row they should be induced by the canonical inclusions $EG_{n}^{K_k}\to EG_{n}^{K_{k+1}}$ coming from the fact that $K_k\subseteq K_{k+1}$ for all $k$.  We thus get a (non-commutative!) diagram
	\begin{equation}\label{ladder}
	\xymatrix{ \Nc^{(m_1)}_* \ar[dr]^-{\Psi^{(1)}} \ar[rr] & & \Nc^{(m_2)}_* \ar[dr]^-{\Psi^{(2)}} \ar[rr] & & \Nc^{(m_3)}_* \ar[dr]^-{\Psi^{(3)}} \ar[rr] & & \\
		& EG^{K_1}_* \ar[ur]_-{\Phi^{(1)}} \ar[rr] & & EG^{K_2}_* \ar[ur]_-{\Phi^{(2)}} \ar[rr] & & & \cdots}.
	\end{equation}
	Notice that the limit of the horizontal maps in the top row is $H_*(\mathcal{A})$.  Moreover, the definition of an anti-\v{C}ech sequence and the construction of the sequence $(K_k)$ forces $G=\bigcup K_k$, so the limit of the horizontal maps on the bottom row is $H_*(G)$.  
	
	Now, consider the compositions
	\begin{equation}\label{to limit}
	\xymatrix{ H_*(EG^{K_k}) \ar[rr]^-{\Phi^{(k)}_*} & & H_*(\Nc^{(m_k)}) \ar[r] & H_*(\mathcal{A})}.  
	\end{equation}
	where the second arrow is the canonical one that exists by definition of the direct limit.  Any two morphisms into any $\Nc^{(m_k)}$ are close using that all the colourings are bounded.  It follows therefore from Lemma \ref{close mor} that for any $k$, the diagram
	$$
	\xymatrix{ H_*(\mathcal{A}) \ar@{=}[r] & H_*(\mathcal{A}) \\
		H_*(EG^{K_k}) \ar[r] \ar[u] & H_*(EG^{K_{k+1}})\ar[u] }
	$$
	commutes; here the vertical maps are the ones in line \eqref{to limit}, and the bottom horizontal line is induced by the canonical inclusion $EG^{K_k}_*\to EG^{K_{k+1}}_*$.  Taking the limit in $k$ of the maps in line \eqref{to limit}, we thus get a well-defined homomorphism 
	$$
	\Phi:H_*(G)\to H_*(\mathcal{A}).
	$$
	Precisely analogously, using Lemma \ref{close mor 2} in place of Lemma \ref{close mor}, we get a homomorphism 
	$$
	\Psi:H_*(\mathcal{A})\to H_*(G).
	$$
	We claim that $\Phi$ and $\Psi$ are mutually inverse, which will complete the proof.  Indeed, for any $k$, the triangles  
	$$
	\xymatrix{  \Nc^{(m_k)}_* \ar[rr] \ar[dr]_-{\Psi^{(k)}} & &  \Nc^{(m_{k+1})}_*  \\
		& EG_*^{K_{k}} \ar[ur]_-{\Phi^{(k)}} & }
	$$
	and 
	$$
	\xymatrix{  & \Nc^{(m_{k+1})}_*  \ar[dr]^-{\Psi^{(k+1)}} &  \\
		EG_*^{K_{k}} \ar[ur]^-{\Phi^{(k)}} \ar[rr] & &  EG_*^{K_{k+1}} }
	$$
	appearing in line \eqref{ladder} commute up to closeness, whence the claim follows directly from Lemmas \ref{close mor} and \ref{close mor 2}, so we are done.
\end{proof}

\subsection{Dynamic asymptotic dimension}\label{asdim}

In this subsection, we show that the homology of an ample $\sigma$-compact principal groupoid vanishes above its dynamic asymptotic dimension, and also that the top-dimensional homology group is torsion free. 

The following definition is \cite[Definition 5.1]{Guentner:2014aa}.

\begin{definition}\label{dad}
Let $d\in \N$.  A (locally compact, Hausdorff, \'{e}tale) groupoid has \emph{dynamic asymptotic dimension at most $d$} if for any relatively compact open subset $K$ of $G$ there are open subsets $U_0,...,U_d$ of $G^{0}$ that cover $r(K)\cup s(K)$ and such that for each $i$, the set $\{g\in K\mid s(g),r(g)\in U_i\}$ is contained in a relatively compact subgroupoid of $G$.
\end{definition}

We record the some basic facts about products of subsets of a groupoid.  See for example \cite[Lemma 5.2]{Guentner:2014aa} for a proof.

\begin{lemma}\label{gp prod}
Let $G$ be an étale groupoid, and $H$ and $K$ are subsets of $G$.  Then if $H$ and $K$ are open (respectively, compact, or relatively compact), the product $HK$ is open (respectively, compact, or relatively compact).  

Moreover, if $K$ is open then the subgroupoid generated by $K$ is also open. \qed
\end{lemma}

Here is our key use of dynamic asymptotic dimension.

\begin{lemma}\label{dad lem}
Let $G$ be an ample groupoid with compact base space which has dynamic asymptotic dimension at most $d$.  Then for any compact open subset $K\subseteq G$ there exists a $K$-Lebesgue colouring of $G$ with at most $d+1$ elements.
\end{lemma}

\begin{proof}
	Say $G$ has dynamic asymptotic dimension at most $d$, and let $K\subseteq G$ be compact and open.  As making $K$ larger only makes the problem more difficult, we may assume that $K=K^{-1}$, and that $K$ contains $G^{0}$.  Note that $K^3$ is still compact and open by Lemma \ref{gp prod}.  The definition of dynamic asymptotic dimension at most $d$ therefore gives an open cover $U_0,...,U_d$ of $G^{0}$ such that for each $i$, the subgroupoid generated by $\{g\in K^3\mid r(g),s(g)\in U_i\}$ is relatively compact.  As $G^{0}$ is compact and as its topology has a basis of compact open sets, there is a cover of $G^{0}$ consisting of compact open sets $V_0,...,V_d$ such that each $V_i$ is contained in $U_i$.  For each $i$, let $W_i=s(r^{-1}(V_i)\cap K)$.  Note that $r^{-1}(V_i)\cap K$ is compact and open, whence $W_i$ is also compact and open as $s$ is an open continuous map.  Moreover, $W_i$ contains $V_i$ as $K$ contains $G^{0}$.  Let now $G_i$ be the subgroupoid of $G$ generated by $\{g\in K\mid r(g),s(g)\in W_i\}$.  We claim that $G_0,...,G_d$ is a $K$-Lebesgue colouring of $G$. 
	
	Indeed, each $G_i$ is open by Lemma \ref{gp prod}, as it is generated by an open set.  For compactness, we first claim that $G_i$ has compact closure.  Indeed, say $k_1...k_n$ is an element of $G_i$, where each $k_j$ is in $\{g\in K\mid s(g),r(g)\in W_i\}$.  Then by definition of $W_i$, for each $j\in \{1,...,n\}$ there is an element $h_j$ of $K$ such that $h_{j-1}k_jh_{j}^{-1}$ has range and source in $U_i$.  Hence for each $j\in \{1,...,n\}$, $h_{j-1}k_jh_{j}^{-1}$ is in $\{g\in K^3\mid r(g),s(g)\in U_i\}$.  Let $H_i$ be the subgroupoid of $G$ generated by $\{g\in K^3\mid r(g),s(g)\in U_i\}$, so $H_i$ has compact closure by choice of the cover $U_0,...,U_d$.  Moreover, 
	$$
	k_1\cdots k_n = h_0^{-1}\Big(\prod_{j=1}^n h_{j-1}k_jh_{j}^{-1}\Big)h_n\in KH_iK.
	$$
We now have that $G_i$ is contained in $KH_iK$.  However, $KH_iK$ is relatively compact by Lemma \ref{gp prod}, so we see that $G_i$ is also relatively compact.  
	
	We next claim that there is $N\in \N$ such that any element of $G_i$ can be written as a product of at most $N$ elements of $\{g\in K\mid s(g),r(g)\in W_i\}$.  Indeed, from \cite[Lemma 8.10]{Guentner:2014aa} there exists $N\in \N$ such that each range fibre $G_i^x$ has at most $N$ elements.  Now, let $g=k_1\cdots k_n$ be an element of $G$, where each $k_j$ is in $K$.  Consider the `path' $k_1,k_1k_2,\dots,k_1\cdots k_n$.  If this path contains any repetitions, we may shorten it by just omitting all the elements in between.  Hence the length $n$ of this path can be assumed to be at most the number of elements of $G^{r(g)}$, which is $N$ as claimed.  
	
	To complete the proof that $G_i$ is compact, it now suffices to show that it is closed.  For this, let $(g_j)$ be a net of elements of $G_i$.  Using the previous claim, we may write $g_j=k_1^j...k_N^j$, where each $k_j^i$ is in $\{g\in K\mid r(g),s(g)\in W_i\}$ (possibly identity elements).  Note that the latter set is compact, whence up to passing to subnets, we may assume that each net $(k^j_i)$ converges to some $k_i\in \{g\in K\mid r(g),s(g)\in W_i\}$.  Hence $g=k_1...k_N$ is in $G_i$, so $G_i$ is indeed compact.
	
	Finally, we check that the colouring is $K$-Lebesgue.  Note that for any $x\in G^0$, $x$ is contained in $V_i$ for some $i$, whence $s(r^{-1}(x)\cap K)$ is contained in $W_i$ by definition of $W_i$.  Hence $G_i$ contains $r^{-1}(x)\cap K$, giving the $K$-Lebesgue condition.
\end{proof}

\begin{theorem}\label{main vanishing}
	Let $G$ be a $\sigma$-compact ample groupoid. If $G$ is principal and has dynamic asymptotic dimension at most $d$, then $H_n(G)=0$ for $n>d$, and $H_d(G)$ is torsion-free.
\end{theorem}

\begin{proof}
	We first prove the result under the additional assumption that $G^0$ is compact. In that case Lemma \ref{dad lem} lets us build an anti-\v{C}ech sequence $\mathcal{A}$ for $G$ consisting of colourings $(\mathcal{C}^{(n)})$ each of which only has $d+1$ colours.  As ${\displaystyle H_*(\mathcal{A})=\lim_{n\to\infty}H_*(\mathcal{C}^{(n)})}$ and as Theorem \ref{vanish} implies that $H_n(\mathcal{C}^{(m)})=0$ for all $m$ and all $n>d$, we see that $H_n(\mathcal{A})=0$ for all $n>d$.  Theorem \ref{main comparison} now gives the vanishing result.
	
	For the claim that $H_d(G)$ is torsion-free result, keep notation as above, and let us write $\Nc_*^{(n)}$ for the nerve space associated to $\mathcal{C}^{(n)}$.  Then we have
	\begin{align*}
	H_d(\mathcal{A}) & =\lim_{n\to\infty}H_d(\mathcal{C}^{(n)})= \lim_{n\to\infty}H_d(\Z[\Nc_*^{(n)}]_G)=\lim_{n\to\infty}H_d(\Z[(\Nc_*^{(n)})^>]_G) \\ & =H_d(\lim_{n\to\infty}\Z[(\Nc_*^{(n)})^>]_G),
	\end{align*}
	where the first equality is just by definition of $H_d(\mathcal{A})$, the second is by definition of $H_d(\mathcal{C}^{(n)})$, the third is Theorem \ref{vanish}, and the fourth follows as taking homology commutes with direct limits.  Now, using Lemma \ref{coinv}, $\Z[(\Nc_d^{(n)})^>]_G$ naturally identifies with $\Z[G\setminus (\Nc_d^{(n)})^>]$.   Each $\Z[G\setminus (\Nc_d^{(n)})^>]$ is clearly torsion-free\footnote{In fact, it is even free: it can be regarded as a commutative ring (with pointwise multiplication) that is generated by idempotents, and hence its underlying group is free abelian by \cite[Theorem 1.1]{Bergman:1972aa}.}. Hence the limit ${\displaystyle \lim_{n\to\infty}\Z[G\setminus (\Nc_d^{(n)})^>]}$ is torsion-free.  On the other hand, ${\displaystyle \lim_{n\to\infty}\Z[(\Nc_m^{(n)})^>]=0}$ for $m>d$.  Hence $H_d(\mathcal{A})$ identifies with a subgroup of ${\displaystyle \lim_{n\to\infty}\Z[G\setminus (\Nc_d^{(n)})^>]}$, and so is itself torsion-free.
	
	Let us now assume that $G^0$ is only locally compact. In this case use $\sigma$-compactness to write $G^0$ as an increasing union $G^0=\bigcup_{m\in\mathbb{N}} K_m$ of compact open subsets. Let $G_m:=\{g\in G\mid s(g),r(g)\in K_m\}$ denote the restriction of $G$ to $K_m$. Note that each $G_m$ is a $\sigma$-compact principal ample groupoid in its own right and it is straightforward to check from the definition (see \cite[Definition 5.1]{Guentner:2014aa}) that the dynamic asymptotic dimension of each $G_m$ is dominated by the dynamic asymptotic dimension of the ambient groupoid $G$. Then $G=\bigcup_{m\in\mathbb{N}} G_m$ and the inclusion maps $G_m\hookrightarrow G$ induce isomorphisms $H_n(G)\cong \lim\limits_{m\to\infty} H_n(G_m)$ for each $n$ (see for example \cite[Proposition~4.7]{Farsi2019}). Applying the compact unit space case from above for each $G_m$ and passing to the limit, we obtain both the vanishing and the torsion-freeness results for $G$.
\end{proof}

\section{The one-dimensional comparison map and the HK-conjecture}\label{hk sec}
In this section our first goal is to construct a canonical comparison map 
$$
\mu_1:H_1(G)\to K_1(C^*_r(G))
$$
for an arbitrary ample groupoid $G$, extending earlier constructions of Matui under additional restrictions on the structure of $G$ (see \cite[Corollary~7.15]{Matui2012} and \cite[Theorem~5.2]{Matui2015}).  In fact we will provide two different constructions.  The first approach is based on relative homology and $K$-theory groups, and is quite explicit and elementary; this is carried out in Subsection \ref{one d}.  The second (suggested by a referee) has its origins in the  triangulated category perspective on groupoid equivariant Kasparov theory as recently exploited by Proietti and Yamashita \cite[Corollary 4.4]{Proietti:2021wz}; it is carried out in Subsection \ref{pyss sec}.  The two approaches turn out to be equivalent, but this requires some fairly lengthy computations; these are carried out in Subsection \ref{id comp sec}.  

In the remaining Subsection \ref{asdim}, we apply these results and our vanishing results from Theorem \ref{main vanishing} to deduce a general theorem on validity of the HK conjecture for low-dimensional groupoids.  The result is actually more precise than this: it shows not only that the HK conjecture is true, but identifies the isomorphisms with the comparison maps we have constructed.

\subsection{Comparison maps from relative $K$-theory}\label{one d}

In this subsection, we give an elementary construction of a map $\mu_1:H_1(G)\to K_1(C^*_r(G))$ that works for any ample groupoid.  This is based on two ingredients: a suitable description of the $K$-theory of the mapping cone of the inclusion $\iota:C_0(G^0)\hookrightarrow C_r^*(G)$, and a relative version of groupoid homology with respect to an open subgroupoid.

Let us start with our description of the $K_0$-group of a mapping cone, which is inspired by results of Putnam in \cite{Putnam1997}.  Recently, Haslehurst has independently developed a similar description to ours \cite{Haslehurst}.  Neither Putnam's nor Haslehurst's constructions are quite the same as ours, but the latter in particular has substantial overlap.  We will thus omit some details below that can be filled in using techniques from Haslehurst's paper \cite[Section 3]{Haslehurst}; full details can also be found in the first version of the current paper on the arXiv.

For a $C^*$-algebra (or Hilbert module) $B$, define $IB:=C([0,1],B)$, and recall that the \emph{mapping cone} of a $*$-homomorphism $\phi:A\to B$ is
\begin{equation}\label{c phi}
C(\phi):=\{(a_0,b)\in A\oplus IB\mid \phi(a_0)=b(0),~b(1)=0\}.
\end{equation}

\begin{definition}
Given a $\ast$-homomorphism of $C^\ast$-algebras $\phi \colon A \to B$, we define the \emph{relative $K$-theory groups} with respect to $\phi$ by
\begin{equation}
K_\ast (A, B; \phi) := K_\ast(C(\phi)).
\end{equation} 
\end{definition}

The first coordinate projection $e^0:A\oplus IB\to A$ restricts to a surjection $C(\phi) \to A$ with kernel $\Sigma B := C_0((0,1), B)$. This induces a six-term exact sequence in $K$-theory
\begin{equation}\label{eq:relsixterm}
\xymatrix{
	K_1(B) \ar[r] & K_0(A,B;\phi) \ar[r] & K_0(A) \ar[d]^{\phi_0} \\
	K_1(A) \ar[u]^{\phi_1} & K_1(A,B;\phi) \ar[l] & \ar[l] K_0(B).
}
\end{equation}

Our first aim is to give a different picture of the relative $K_0$-group.  
In what follows, when $\phi \colon A \to B$ is a $\ast$-homomorphism, we abuse notation by letting $\phi$ also denote the induced homomorphism $M_n(A) \to M_n(B)$ for $n\in \mathbb N$.

\begin{definition}
Let $\phi \colon A \to B$ be a $\ast$-homomorphism between $C^\ast$-algebras.  For $n\in \mathbb N$, let $V_n(A, B; \phi)$ denote the set of triples $(p, v, q)$ where $p, q\in M_n(A)$ are projections and $v\in M_n(B)$ satisfies $vv^\ast = \phi(p)$, and $v^\ast v = \phi(q)$. For each $n$, regard $V_n(A,B;\phi)$ as a subset of $V_{n+1}(A,B;\phi)$ via the usual top-left corner inclusion, and define
$$
V_\infty(A, B; \phi) := \bigcup_{n\in \mathbb N}V_n(A, B;\phi),
$$
equipped with the inductive limit topology.  Let $\sim_{\mathrm h}$ denote homotopy in $V_\infty(A, B; \phi)$,\footnote{A homotopy in $V_\infty(A, B; \phi)$ is automatically contained in $V_n(A, B; \phi)$ for some $n\in \mathbb N$, so one can also take this as the definition of homotopy.} and let $\approx$ be the equivalence relation on $V_\infty(A, B; \phi)$ defined as follows: $(p,v,q) \approx (p', v', q')$ exactly when there are projections $r, r' \in M_n(A)$ such that
\begin{equation}
(p, v , q) \oplus (r, \phi(r) , r) \sim_{\mathrm{h}} (p',v',q') \oplus (r', \phi(r') , r').
\end{equation}
In the special case where $\phi$ is an inclusion of a $C^\ast$-subalgebra, we simply write $V_\infty(A,B)$ instead of $V_\infty(A,B; \phi)$. Moreover, if $(p,v,q)\in V_\infty(A,B)$ then $p=vv^\ast$ and $q=v^\ast v$. Hence we will simply denote the elements of $V_\infty(A,B)$ by $v$.

We equip $V_\infty(A, B; \phi)/\!\!\approx$ with the block sum operation defined by
$$
[p_1,v_1,q_1]\oplus [p_2,v_2,q_2]:=\Bigg[\begin{pmatrix} p_1 & 0 \\ 0 & p_2 \end{pmatrix}~,~\begin{pmatrix} v_1 & 0 \\ 0 & v_2 \end{pmatrix}~,~\begin{pmatrix} q_1 & 0 \\ 0 & q_2 \end{pmatrix}\Bigg];
$$
it is straightforward to check that this makes $V_\infty(A, B; \phi)/\!\!\approx$ into a well-defined monoid with identity $[0,0,0]$.
\end{definition}

Our goal is to construct an isomorphism $\eta: K_*(A,B;\phi)\to V_\infty(A, B; \phi)/\!\!\approx$ when $\phi$ is non-degenerate, and $A$ has an approximate unit of projections; we will apply our construction to the canonical $*$-homomorphism $\iota:C_0(G^0)\to C^*_r(G)$ for an ample groupoid $G$, so these assumptions are satisfied.  For this, we need an ancillary construction.

Let $\phi \colon A \to B$ be a $\ast$-homomorphism and define the \emph{double} of $\phi$ to be 
\begin{equation}\label{d phi}
D(\phi):=\{(a_0,b,a_1)\in A\oplus IB\oplus A\mid \phi(a_0)=b(0)\text{ and }\phi(a_1)=b(1)\}.
\end{equation}
With $C(\phi)$ as in line \eqref{c phi}, we obtain a short exact sequence
\begin{equation}\label{c d ses}
\xymatrix{ 0 \ar[r] & C(\phi) \ar[r]^{\iota} & D(\phi) \ar[r]^-{e^1} & A \ar[r] & 0 }
\end{equation}
with morphisms defined by $\iota:(a_0,b)\mapsto (a_0,b,0)$ and $e^1:(a_0,b,a_1)\mapsto a_1$.  Let $\mathsf c \colon B \to IB$ denote the constant $\ast$-homomorphism.  Then the short exact sequence in line \eqref{c d ses} is split by the map $s:a\mapsto (a,\mathsf{c}(\phi(a)),a)$.   Hence the composition
\begin{equation}\label{eq:relisogen}
K_\ast(A, B;\phi) \xrightarrow{\iota_\ast} K_\ast(D(\phi)) \twoheadrightarrow K_\ast(D(\phi))/s_\ast(K_\ast(A))
\end{equation}
is an isomorphism. 

We let $U_n(A)$ denote the unitary group of $M_n(A)$ for a unital $C^\ast$-algebra $A$, and let $V(A)$ denote the the semigroup of equivalence classes of projections in $\bigcup_{n\in \mathbb N} M_n(A)$, so that $K_0(A)$ is the Grothendieck group of $V(A)$ whenever $A$ is unital, or more generally, when $A$ has an approximate unit consisting of projections.

Assume now that $A$ is a $C^*$-algebra with an approximate unit consisting of projections, and let $\phi:A\to B$ be a non-degenerate $*$-homomorphism; note then that $B$ and so $D(\phi)$ also have approximate units consisting of projections, so in particular $K_0(D(\phi))$ is the Grothendieck group of  $V(D(\phi))$.  Let $(p_0,\mathsf p, p_1)\in M_n(D(\phi))$ be a projection representing a $K$-theory class.  Using for example \cite[Corollary 4.1.8]{Higson:2000bs}, there is a continuous path of unitaries $(u_t)_{t\in [0,1]}$ in $M_n(B)$ such that 
\[
\mathsf p_t = u_t^\ast \mathsf p_0 u_t = u_t^\ast \phi(p_0) u_t, \quad \textrm{for } t\in [0,1]\quad \textrm{and} \quad u_0 = 1.
\]
We define $\eta:V(D(\phi))\to V_\infty(A,B;\phi)$ by the formula
\begin{equation}\label{etadef}
\eta(p_0,\mathsf p, p_1):=(p_0, \phi(p_0) u_1, p_1).
\end{equation}

Here then is our promised picture of relative $K$-theory.  Due to similarities to Haslehurst's methods \cite{Haslehurst}, we do not give a proof here: the interested reader can find a complete proof in the original arXiv version of this paper.

\begin{theorem}\label{p:relKthunital}
Let $\phi:A\to B$ be a non-degenerate $*$-homomorphism, where $A$ has an approximate unit consisting of projections.  

With notation as in line \eqref{eq:relisogen}, the map $\eta$ from line \eqref{etadef} above descends to a well-defined monoid isomorphism
$$
\eta:\frac{K_0(D(\phi))}{s_0(K_0(A))}\to \frac{V_\infty(A, B; \phi)}{\approx}.
$$
In particular, $V_\infty(A, B; \phi)/\!\!\approx$ is an abelian group that is isomorphic to $K_0(A,B;\phi)$.  

Moreover, the six term exact sequence from line \eqref{eq:relsixterm} identifies with an exact sequence
$$
K_1(A) \xrightarrow{\phi_1} K_1(B) \to \frac{V_\infty(A, B; \phi)}{\approx} \to K_0(A) \xrightarrow{\phi_0} K_0(B)
$$
where the map $K_1(B) \to \frac{V_\infty(A, B; \phi)}{\approx}$ is given by $[u]_1 \mapsto [p\otimes 1_{M_n} ,u,p \otimes 1_{M_n}]$ for $u\in U_n(\phi(p)B\phi(p))$ (where $p\in A$ is a projection); and $\frac{V_\infty(A, B; \phi)}{\approx} \to K_0(A)$ is given by $[p,v,q] \mapsto [p]_0- [q]_0$.  \qed
\end{theorem}

\begin{rem}\label{vab rem}
We note the following basic properties of cycles $(p_1,v_1,q_1)$ and $(p_2,v_2,q_2)$ in $V_\infty(A,B;\phi)$; these can be justified using the same sort of rotation homotopies that establish the analogous properties for the classical $K_0$ and $K_1$ groups.
\begin{enumerate}[(i)]
\item \label{op} If $p_1p_2=q_1q_2=0$, then in $V_\infty(A, B; \phi)/\!\!\approx$ we have that 
$$
[p_1,v_1,q_1]+[p_2,v_2,q_2]=[p_1+q_1,v_1+v_2,p_2+q_2].
$$
\item \label{sp} If $p_2=q_1$, then 
$$
[p_1,v_1,q_1]+[p_2,v_2,q_2]=[p_1,v_1v_2,p_2].
$$
\end{enumerate}
\end{rem}

Let us now turn to the second ingredient in the construction of the comparison map $\mu_1:H_1(G)\rightarrow K_1(C_r^*(G))$: relative groupoid homology. To define it, we return to Matui's picture of groupoid homology as it allows for an elementary description.
For an open subgroupoid $H\subseteq G$ the canonical inclusion induces for each $n\geq 0$ a short exact sequence of abelian groups
\begin{equation}\label{SES relative homology}
0\rightarrow \Z[H^{(n)}]\rightarrow \Z[G^{(n)}] \rightarrow \Z[G^{(n)}]/\Z[H^{(n)}]\rightarrow 0.
\end{equation}

Let $\Z_n[G,H]$ denote the quotient group.  One easily checks that $\Z[H^{(n)}]$ is invariant under the boundary maps $\partial_n$ for $\Z[G^{(n)}]$, and hence we obtain induced maps $\partial_n'$ turning $(\Z[G,H],\partial'_n)$ into a chain complex such that the sequence (\ref{SES relative homology}) is an exact sequence of chain complexes.
\begin{defi} 
Let $G$ be an ample groupoid.
The \textit{relative homology} of $G$ with respect to an open subgroupoid $H$ is defined as the homology of the chain complex $(\Z[G,H],\partial'_n)$, i.e.\
$$H_n(G,H):=\ker (\partial_n')/\mathrm{im}(\partial_{n+1}').$$
\end{defi}

From the short exact sequence (\ref{SES relative homology}) we obtain a long exact sequence of homology groups
$$
\cdots\rightarrow H_1(H) \rightarrow  H_1(G)\rightarrow H_1(G,H)\rightarrow H_0(H)\rightarrow H_0(G)\rightarrow H_0(G,H)
$$

Important for us is the special case where $H=G^0$: there one easily checks that $H_0(G,G^0)=0$, and $H_1(G,G^0)\cong \ZZ[G]/\mathrm{im}(\partial_2)$.  We are now ready to put everything together and construct the map $\mu_1:H_1(G)\rightarrow K_1(C_r^*(G))$.

\begin{lemma}\label{Lemma:ConstructionH1}
Let $G$ be an ample groupoid. Then there exists a well-defined homomorphism 
$$
\tilde{\rho}: \ZZ[G] \rightarrow V_\infty(C_0(G^0),C_r^*(G))/\!\!\approx
$$
determined by $\tilde{\rho}(1_W)=[1_W^*]$ and $\tilde{\rho}(-1_W)=[1_W]$ for every compact open bisection $W\subseteq G$.
	
Moreover, $im(\partial_2)\subseteq \ker(\tilde{\rho})$ and hence $\tilde{\rho}$ factors through a well-defined homomorphism
$$
\rho:H_1(G,G^0)\cong \ZZ[G] /im(\partial_2)\rightarrow V_\infty(C(G^0),C_r^*(G))/\!\!\approx.
$$	
\end{lemma}

\begin{proof}
As $G$ is an ample groupoid, every function $f\in \ZZ[G]$ can be written (non-uniquely) as a linear combination $f=\sum \lambda_i 1_{W_i}$ where $\lambda_1,\ldots,\lambda_n\in \ZZ$ and $W_1,\ldots, W_n$ are pairwise disjoint compact open bisections of $G$.
We have to show that the resulting class of the partial isometry 
$$
\bigoplus_{\lbrace i\mid \lambda_i\geq 0\rbrace} \lambda_i 1_{W_i}\oplus \bigoplus_{\lbrace i\mid \lambda_i < 0\rbrace} \abs{\lambda_i}1_{W_i^{-1}}
$$
 in $V_\infty(C_0(G^0),C_r^*(G))/\!\!\approx$ only depends on $f$ and not on the particular choice of the bisections $W_i$.  Let us first note the following: If $W$ is a compact open bisection of $G$ such that $W=U\sqcup V$ for compact open subsets $U,V\subseteq W$, then $1_W=1_U+1_V\approx 1_V\oplus 1_U$ by Remark \ref{vab rem}, part \eqref{op}.

With this in mind we can establish the lemma:
Let $f\in \ZZ[G]$. We may assume that $f\geq 0$. Suppose we have $f=\sum \lambda_i 1_{U_i}=\sum \mu_j 1_{V_j}$ for $\lambda_i,\mu_j\in \NN$ and two families $(U_i)_i$ and $(V_j)_j$  of pairwise disjoint compact open bisections of $G$. Since $G$ is ample we may choose a common refinement of these two families by compact open bisections $(W_k)_k$. Let $\eta_k$ be equal to $\lambda_i$ if $W_k\subseteq U_i$ and equal to $\mu_j$ if $W_k\subseteq V_j$.
Then we have 
$$
\bigoplus_k \eta_k 1_{W_k}\approx\bigoplus_j\bigoplus_{\lbrace k\mid W_k\subseteq V_j\rbrace}\mu_j 1_{W_k}\approx\bigoplus_j \mu_j\bigoplus_{\lbrace k\mid W_k\subseteq V_j\rbrace} 1_{W_k}\approx \bigoplus_j \mu_j 1_{V_j}.
$$
Similarly, we obtain $\bigoplus_k \eta_k 1_{W_k}\approx \bigoplus_i \lambda_i1_{U_i}$, from which we conclude that $\tilde{\rho}$ is indeed well-defined. It is a group homomorphism by construction.
	
For the second part let $U$ and $V$ be compact open bisections of $G$ such that $s(U)=r(V)$. Then we have $\tilde{\rho}(\partial_2(1_{(U\times V)\cap G^{(2)}}))=\tilde{\rho}(1_U-1_{UV}+1_V)=0$ since $1_U\oplus 1_V\approx 1_U1_V=1_{UV}$ by Remark \ref{vab rem}, part \eqref{sp}.
\end{proof}

\begin{theorem}\label{mu1 exist} 
Let $G$ be an ample groupoid. Then there exists a canonical homomorphism $$\mu_1:H_1(G)\rightarrow K_1(C_r^*(G)).$$
\end{theorem}
\begin{proof}
Let $\iota:C_0(G^0)\rightarrow C_r^*(G)$ denote the canonical inclusion, which is non-degenerate. As $G$ is ample, $C_0(G^0)$ contains an approximate identity of projections. By Theorem \ref{p:relKthunital} and Lemma \ref{Lemma:ConstructionH1} we obtain a commutative diagram with exact rows
\begin{equation}\label{Equation:CommutativeDiagram}
\begin{tikzcd}[column sep=small]
0 \arrow[r] & H_1(G) \arrow[r] \arrow[d, "\mu_1", dashed] & H_1(G,G^0) \arrow[r] \arrow[d, "\rho"] & H_0(G^0) \arrow[r] \arrow[d, "\cong"] & H_0(G) \arrow[d, "\mu_0"] \\
0 \arrow[r] & K_1(C_r^*(G)) \arrow[r]                     & V_\infty(C_0(G^0), C^\ast_r(G))/\!\!\approx \arrow[r]                                         & K_0(C_0(G^0)) \arrow[r, "\iota_0"]                            & K_0(C_r^*(G))            
\end{tikzcd},
\end{equation}
where the top row is part of the long exact sequence for the pair $(G,G^0)$.  By exactness, there exists a unique homomorphism $\mu_1:H_1(G)\rightarrow K_1(C_r^*(G))$ filling in the dashed arrow such that the diagram commutes.
\end{proof}

\subsection{Comparison maps from the ABC spectral sequence}\label{pyss sec}

In this subsection we recall the spectral sequence constructed by Proietti and Yamashita in \cite[Corollary 4.4]{Proietti:2021wz} (which is in turn based on the ABC spectral sequence of Meyer \cite{Meyer2008}), and show that this gives rise to canonical comparison maps $\mu_k:H_k(G)\to K_k(C^*_r(G))$ for $k\in \{0,1\}$.  This approach to the construction of comparison maps was suggested by the referee.

Throughout this section, $G$ is a second countable ample groupoid.  We will need to work extensively with $G$-$C^*$-algebras and the $G$-equivariant Kasparov category $KK^G$: see \cite{Le-Gall:1999aa} for background.  We will also need to treat $KK^G$ as a triangulated category as described in \cite[Section A.4]{Proietti:2021wz}.  For background on the material we will need about triangulated categories and homological ideals and functors, see \cite{Meyer:2010vt}.  

Let $\mathcal{I}$ be the homological ideal (see \cite[Definition 2.20 and Remark 2.21]{Meyer:2010vt}) in $KK^G$ defined as the kernel of the restriction functor $F:=\text{Res}_{G^0}^G:KK^G\to KK^{G^0}$.  If we define also $E:=\text{Ind}_{G^0}^G:KK^{G^0}\to KK^G$ (see \cite[Section 2.1]{BP22} for a detailed treatment of this), then $(E,F)$ form an adjoint pair as established in \cite[Section 6]{B20} (see also \cite[Theorem 2.3]{BP22}). Define $L:=E\circ F:KK^G\to KK^G$, and for an object $B$ of $KK^G$, let $\epsilon_B\in KK^G(L(B),B)$ be the counit of adjunction, which is computed explicitly in \cite[Theorem 2.3]{BP22}.

Now, as a consequence of \cite[Proposition 3.1]{Proietti:2021wz} we see there is an (even) $\mathcal{I}$-projective resolution in the sense of \cite[Definition 2.14]{Proietti:2021wz} of the object $A=C_0(G^0)$ in $KK^G$
\begin{equation}\label{pr}
\xymatrix{ A & \ar[l]^-{\delta_0} P_0 & \ar[l]^-{\delta_1} P_1 &  P_2 \ar[l]^-{\delta_2} &\ar[l]^-{\delta_3} \cdots},
\end{equation}
where for each $p\in \N$, $P_p:=L^{p+1}(A)$ and 
\begin{equation}\label{boundary}
\delta_p:=\sum_{i=0}^p (-1)^iL^i(\epsilon_{L^{n-i}(A)}).
\end{equation}
Moreover, one computes from the explicit description of $\text{Ind}_{G^0}^G$ in \cite[Section 2.1]{BP22} that $L^{p+1}(A)$ identifies canonically with $C_0(G^{(p+1)})$, where we recall that 
$$
G^{(p)}:=\{(g_1,...,g_p)\in G^p\mid s(g_i)=r(g_{i+1}) \text{ for all }i\in \{1,...,n-1\}\}
$$
and we write $G^{(0)}$ and $G^0$ interchangeably.  Here and throughout, we equip $G^{(p)}$ with the left $G$-action induced by 
\begin{equation}\label{gp act}
g:(h_1,h_2,...,h_p)\mapsto (gh_1,h_2,...,h_p)
\end{equation}
and $C_0(G^{(p)})$ with the corresponding $G$-action.  We will have need of the following definition and lemma multiple times.

\begin{definition}\label{l2yp}
Let $X$ and $Y$ be locally compact spaces and let $\mathrm p:Y\to X$ be an \'{e}tale map.  Equip $C_c(Y)$ with the $C_0(X)$-valued inner product defined\footnote{The function $\langle \xi,\eta\rangle$ is well-defined as the sum is finite; it is continuous as $p$ is \'{e}tale; and it is compactly supported as it is supported in $p(\text{supp}(\xi)\cap \text{supp}(\eta))$.} by
$$
\langle \xi,\eta\rangle (x):=\sum_{y\in \mathrm p^{-1}(\{x\})}\overline{\xi(y)}\eta(y)
$$
The corresponding completion is a Hilbert $C_0(X)$-module that we will denote $L^2(Y,p)$.
\end{definition}

Let now $A$ be a $G$-$C^*$-algebra, and let $r^*A$ be the pullback along $r:G\to G^0$ (see for example \cite[2.7(d)]{Khoshkam:2004uj}).  The reduced crossed product of $A$ by $G$, denoted $A\rtimes_r G$, is defined to be a certain completion of $C_c(G)\cdot r^*A$, where the latter is equipped with a natural $*$-algebra structure: we refer the reader to \cite[Section 3]{Khoshkam:2004uj} for more details.   For an element $f\in C_c(G)\cdot r^*A$ and $h\in G$, we write $f(h)$ for the corresponding element of $(r^*A)_h=A_{r(h)}$.  In the special case $A=C_0(G^{(p)})$ with the left $G$-action from line \eqref{gp act}, we see that $f(h)$ identifies with an element of $C_0(G^{r(h)}\!_s\times_r G^{(p-1)})$.   The next lemma is well-known and will be used by us several times below: it follows from direct checks based on the formulas we give that we leave to the reader.    

\begin{lemma}\label{mor1}
For each $p\geq 2$, let $\mathrm{pr}:G^{(p)}\to G^{(p-1)}$ be the map $(g_1,...,g_p)\mapsto (g_2,...,g_p)$, and let $\mathrm{pr}:G^{(1)}\to G^{(0)}=G^{0}$ be the map $g\mapsto s(g)$.  
For each $p\geq 1$, $f\in C_c(G)\cdot r^*C_0(G^{(p)})$, $\xi\in L^2(G^{(p)},\mathrm{pr})$, and $g\in G$ define 
$$
(\kappa_p(f)\xi)(g_1,...,g_p):=\sum_{h\in G^{r(g_1)}}f(h)(g_1,...,g_p)\xi(h^{-1}g_1,g_2,...,g_p).
$$
This extends to an injective representation
$$
\kappa_p:C_0(G^{(p)})\rtimes_r G\to \mathcal{L}(L^2(G^{(p)},\mathrm{pr}))
$$
whose image is exactly $\K(L^2(G^{(p)},\mathrm{pr}))$.  

In particular, the Kasparov class 
$$
[\kappa_p,L^2(G^{(p)},\mathrm{pr}),0]\in KK(C_0(G^{(p)})\rtimes_r G,C_0(G^{(p-1)}))
$$
is an isomorphism, induced by a canonical Morita equivalence $C_0(G^{(p)})\rtimes_r G \stackrel{M}{\sim} C_0(G^{(p-1)})$. \qed
\end{lemma}

We will need to work with a phantom tower in the sense of \cite[Definition 3.1]{Meyer2008} associated to the $\mathcal{I}$-projective resolution in line \eqref{pr} above.  This is an augmentation of that resolution to a diagram of the form 
\begin{equation}\label{pt diag}
\xymatrix{ A  \ar[rr]|\circ^-{\iota_0^1} & & N_1 \ar[rr]|\circ^-{\iota_1^2} \ar[dl]^-{\phi_1} & & N_2\ar[dl]^-{\phi_2} \ar[rr]|\circ^-{\iota_2^3} & &  \ar[dl]^-{\phi_3}  \cdots  &     \\
& \ar[ul]^-{\delta_0} P_0 &  & \ar[ll]^-{\delta_1}  P_1 \ar[ul]^-{\psi_1} & &  \ar[ll]^-{\delta_2} P_2 \ar[ul]^-{\psi_2} && \ar[ll]^-{\delta_3}   \cdots, }
\end{equation}
satisfying certain conditions: we refer the reader to \cite{Meyer2008} for more details.  Note that as in \cite{Meyer2008} we are using circled arrows for morphisms of degree one and plain arrows for morphisms of degree zero, but the circles are not in the same place as those of \cite[Definition 3.1]{Meyer2008}.  This is because (following \cite{Proietti:2021wz}), we are working with \emph{even} $\mathcal{I}$-projective resolutions, whereas in \cite[Definition 3.1]{Meyer2008}, Meyer works with \emph{odd} resolutions, where the maps $\delta_p$ in line \eqref{pr} have degree one (compare \cite[Definition 2.15]{Proietti:2021wz}).  Analogously to the explanation below \cite[Definition 2.15]{Proietti:2021wz}, it is straightforward to switch between these two degree conventions for phantom towers: if $\Sigma:KK^G\to KK^G$ is the suspension automorphism, and we are given a phantom tower associated to an odd $\mathcal{I}$-projective resolution with objects $N_p$ and $P_p$ as in \cite[Definition 3.1]{Meyer2008}, replacing these by $\Sigma^{-p}N_p$ and $\Sigma^{-p}P_p$ yields a phantom tower associated to the corresponding even resolution as in diagram \eqref{pt diag} above.

Now, following \cite[Lemma 3.2]{Meyer2008}, the $\mathcal{I}$-projective resolution in line \eqref{pr} can be augmented in an essentially unique way to a phantom tower as in line \eqref{pt diag}.  Let $J:KK^G\to \text{Ab}^{\Z/2}$ be the functor from $KK^G$ to the category of $\Z/2$-graded abelian groups defined by $J:=K_*\circ j_G$, where $j_G:KK^G\to KK$ is descent (see for example \cite[Section 1.3]{BP22}), and $K_*:KK\to\text{Ab}^{\Z/2}$ is $K$-theory; note that on objects, $J(A):= K_*(A\rtimes_r G)$.  This is a stable homological functor (see \cite[Definitions 2.12 and 2.14]{Meyer:2010vt}), so we may use it to build the ABC spectral sequence of \cite[Sections 4 and 5]{Meyer2008} based on the phantom tower from line \eqref{pt diag}.  If $G$ is amenable and has torsion-free isotropy groups, \cite[Corollary 4.4]{Proietti:2021wz} shows that the ABC spectral sequence for $A=C_0(G^0)$ is a convergent spectral sequence 
\begin{equation}\label{con ss}
E^r_{p,q}\Rightarrow K_{p+q}(C^*_r(G)).
\end{equation}
We call this the \emph{Proietti-Yamashita (or PY) spectral sequence}.
 
As we will need to do explicit computations, let us give more details about how the general construction of the ABC spectral sequence specialises to our case, following \cite[Sections 4 and 5]{Meyer2008}.  For each $p,q\in \N$ as in \cite[page 189]{Meyer2008}, define\footnote{Note that the conventions of \cite[Sections 4 and 5]{Meyer2008} would use ``$K_{p+q}$'' where we have ``$K_q$'' in the definitions of $E^1_{pq}$ and of $D^1_{pq}$; the difference comes from our choice of using even projective resolutions - which necessitates introducing $p$-fold desuspensions - as explained above}
$$
E^1_{p,q}:=K_q(P_p\rtimes_r G) \cong \left\{\begin{array}{ll} K_0(C_0(G^{(p)})) & q \text{ even} \\ 0 &q\text{ odd} \end{array}\right.
$$ 
where the isomorphism uses the canonical Morita equivalence $P_p\rtimes_r G=C_0(G^{(p+1)})\rtimes_r G \stackrel{M}{\sim} C_0(G^{(p)})$ from Lemma \ref{mor1}.  Define also
$$
D^1_{p,q}:=K_q(N_{p+1}\rtimes_r G).
$$
Continuing following \cite[page 189]{Meyer2008}, there is an exact couple (see for example \cite[Section 5.9]{Weibel1994} for background)
$$
\xymatrix{ D \ar[rr]^-i & & D \ar[dl]^-{j} \\ & E \ar[ul]^-k & }
$$
with morphisms $i_{pq}:D^1_{p,q}\to D^1_{p+1,q-1}$ defined by $i_{pq}:=J_q(\iota_{p+1}^{p+2})$, $j_{pq}:D^1_{p,q}\to E^1_{p,q}$ defined by $j_{pq}:=J_q(\phi_p)$, and $k_{pq}:E^1_{p,q}\to D^1_{p-1,q}$ defined by $k_{pq}:=J_q(\psi_p)$.  As in \cite[Corollary 4.4]{Proietti:2021wz}, one computes that the second page $E^2$ has entries given by 
$$
E^2_{p,q}\cong \left\{\begin{array}{ll} H_p(G) & q \text{ even} \\ 0 &q\text{ odd} \end{array}\right..
$$
With notation as in \cite[page 172]{Meyer2008}, \cite[Theorem 5.1]{Meyer2008} implies that the filtration on $K_{p+q}(C^*_r(G))$ that corresponds to the convergence in line \eqref{con ss} above is given by $(J:\mathcal{I}^k(A))_{k\in \N}$; moreover, using the discussion on \cite[pages 176-177]{Meyer2008}, we see that for each $k\geq 1$
\begin{equation}\label{j:i}
J:\mathcal{I}^{k}(A)=\text{Kernel}\big(J(\iota^k_{k-1}\circ \cdots \circ \iota^1_0):K_*(A\rtimes_r G)\to K_*(N_k\rtimes_rG)\big),
\end{equation}
while $J:\mathcal{I}^0(A)=0$ by definition.

Now, following the proof of \cite[Proposition 4.1]{Meyer2008} and using that in our case the spectral sequence converges, for each $r,p,q$ with $0\leq p\leq r$, the map\footnote{As mentioned by Meyer on \cite[page 193]{Meyer2008}, this map is functorially determined by $A$, and does not depend on any of the choices of projective resolution or phantom tower involved.}
\begin{equation}\label{maps}
k^{(r+1)}_{pq}=J_q(\psi_p):E^{r+1}_{p,q}\to D^{r+1}_{p-1,q}
\end{equation}
induced by $k$ on the $(r+1)^{th}$ page has image equal to 
\begin{equation}\label{einf}
J_q(\iota^p_{p-1}\circ \cdots \circ \iota^1_0)\big(J_q:\mathcal{I}^{p+1}(A)\big)\cong  \frac{J_{q}:\mathcal{I}^{p+1}(A)}{J_{q}:\mathcal{I}^{p}(A)}\cong E^\infty_{p,q}
\end{equation}
(for $p=0$, $\iota^p_{p-1}\circ \cdots \circ \iota^1_0$ should be interpreted as the identity map on $A$).  

Now, taking $p=0$, $q=0$, and $r=2$ we get a canonical map 
\begin{equation}\label{pre mu0}
\mu_0:H_0(G)=E^{2}_{0,0}\to \frac{J_0:\mathcal{I}^1(C_0(G^0))}{J_0:\mathcal{I}^0(C_0(G^0))},
\end{equation}
which is induced exactly by $J(\delta_0)$.  Note however that by definition $J_0:\mathcal{I}^0(C_0(G^0))=0$, and by definition of $J$, $J_0:\mathcal{I}^1(C_0(G^0))$ is a subgroup of $K_0(C^*_r(G))$.  Hence we may identify $\mu_0$ with the map
$$
\mu_0:H_0(G)\to K_0(C^*_r(G))
$$
induced by $\delta_0\in KK^G(C_0(G),C_0(G^0))$.  Let us record this more explicitly in a lemma.

\begin{lemma}\label{mu0 descr}
Let $\delta_0\in KK^G(C_0(G),C_0(G^0))=KK^G(L(C_0(G^0)),C_0(G^0))$ be the co-unit of adjunction.   Let $[X_G]\in KK(C_0(G^0),C_0(G)\rtimes_r G)$ be the $KK$ class arising from the canonical Morita equivalence $C_0(G^0)\stackrel{M}{\sim}C_0(G)\rtimes_r G$ of Lemma \ref{mor1}.  Then there is a canonical homomorphism
$$
\mu_0:H_0(G)\to K_0(C^*_r(G))
$$
defined by taking the composition
$$
J_0(\delta_0)\circ K_0([X_G]):K_0(C_0(G^0)) \to K_0(C^*_r(G)),
$$
then identifying $K_0(C_0(G^0))=\Z[G^0]$, and noting that this map descends to the quotient $H_0(G)$ of $\Z[G^0]$.  \qed
\end{lemma}

On the other hand, taking $p=1$, $q=0$, and $r=2$ in lines \eqref{maps} and \eqref{einf}, we get a canonical map 
\begin{equation}\label{pre mu1}
\mu_1:H_1(G)=E^2_{1,0}\to \frac{J_1:\mathcal{I}^2(C_0(G^0))}{J_1:\mathcal{I}^1(C_0(G^0))}
\end{equation}
induced by $J(\psi_1)$ and taking image in $J_1(\iota_0^1)\big(J_1(A)\big)\subseteq J_0(N_1)$, which is isomorphic to the right hand side in line \eqref{pre mu1} above.  Note that by definition (compare line \eqref{j:i}), the group $J_1:\mathcal{I}^1(C^0(G^0))$ is the kernel of the map
$$
J_1(\iota^1_0):J_1(A)\to J_0(N_1).
$$
However, thanks to exactness of the first triangle
\begin{equation}\label{pt diag 2}
\xymatrix{ A  \ar[rr]|\circ^-{\iota_0^1} & & N_1  \ar[dl]^-{\phi_1} &     \\
& \ar[ul]^-{\delta_0} P_0 &   }
\end{equation}
from line \eqref{pt diag} and the fact that $J$ is homological, we have an exact sequence 
$$
\xymatrix{ J_1(P_0) \ar[r]^-{J(\psi_0)} & J_1(A) \ar[r]^-{J(\iota_0^1)} & J_0(N_1) }.
$$
The group $J_1(P_0)$ is by definition equal to $K_1(C_0(G)\rtimes_rG)$, which is isomorphic to $K_1(C_0(G^0))$ by the Morita equivalence from Lemma \ref{mor1}, so is zero as $G$ is ample.  Hence by exactness, $J_1(\iota^1_0)$ is injective, and so $J_1:\mathcal{I}^1(C_0(G^0))=0$.  Putting this information into line \eqref{pre mu1}, together with the fact that $J_1:\mathcal{I}^2(C^0(G^0))$ is naturally a subgroup of $K_1(C^*_r(G))$, we get a canonical map
$$
\mu_1:H_1(G)\to K_1(C^*_r(G))
$$
induced by $\psi_1\in KK^G(C_0(G^{(2)}),N_1)$.  Let us again record this more explicitly as a lemma.

\begin{lemma}\label{mu1 descr}
Let $\psi_1\in KK^G(C_0(G^{(2)}),N_1)=KK^G(P_1,N_1)$ be the morphism from the following part of a phantom tower for $A=C_0(G^0)$
\begin{equation}\label{pt diag 3}
\xymatrix{ A  \ar[rr]|\circ^-{\iota_0^1} & & N_1 \ar[dl]^-{\phi_1} \ar[rr]|\circ^-{\iota_1^2} & & N_2 \ar[dl]^-{\phi_2} &    \\
& \ar[ul]^-{\delta_0} P_0 &  & \ar[ll]^-{\delta_1}  P_1 \ar[ul]^-{\psi_1} & & P_2 \ar[ll]^-{\delta_2} \ar[ul]^-{\psi_2} }.
\end{equation}
Let $J_1(\iota_0^1)|^{\text{Im}(J_1(\iota_0^1))}$ be the corestriction of $J(\iota_0^1)$ to an isomorphism $J_1(A)\to  \text{Image}(J_1(\iota_0^1)) \subseteq J_0(N_1)$.  Let $[X_{G^{(2)}}]\in KK(C_0(G),C_0(G^{(2)})\rtimes_r G)$ be the $KK$-isomorphism arising from the canonical Morita equivalence $C_0(G)\stackrel{M}{\sim}C_0(G^{(2)})\rtimes_r G$ of Lemma \ref{mor1}.  

Then there is a canonical homomorphism
$$
\mu_1:H_1(G)\to K_1(C^*_r(G))
$$
defined by taking the composition\footnote{Note that the restriction to the kernel $\text{Ker}(J_0(\delta_1))$ of $J(\delta_1)$ is needed to ensure that the image of $J_0(\psi_1)\circ K_*([X_{G^{(2)}}])$ is contained in the image of $J_1(\iota_0^1)$; we mention this fact explicitly as it is slightly buried in the spectral sequence machinery.}
$$
(J_1(\iota_0^1)|^{\text{Im}(J_1(\iota_0^1))}\big)^{-1}\circ J_0(\psi_1)\circ K_0([X_{G^{(2)}}]):\text{Ker}(J_0(\delta_1))\to K_1(C^*_r(G)),
$$
and noting that this restriction passes to the quotient by the image of $J_0(\delta_2):K_0(C_0(G^{(2)})\rtimes_r G)\to K_0(C_0(G)\rtimes_r G)$\footnote{This last holds as commutativity of the rightmost triangle in diagram \eqref{pt diag 3} and exactness of the second triangle from the right imply that $J(\psi_1)\circ J(\delta_2)=J(\psi_1)\circ J(\phi_2)\circ J(\psi_2)=0\circ J(\psi_2)=0$.} .  \qed
\end{lemma}

\begin{rem}\label{no go higher}
In the context of the HK conjecture, it is natural to ask if there are ``higher'' maps $\mu_k:H_k(G)\to K_k(C^*_r(G))$ arising from the Proietti-Yamashita spectral sequence in a canonical way.  This does not seem clear, even for $k=2$: here lines \eqref{maps} and \eqref{einf} would give a map
$$
\mu_2:H_2(G)=E^2_{2,0}\to \frac{J_0:\mathcal{I}^3(C_0(G^0))}{J_0:\mathcal{I}^2(C_0(G^0))}.
$$
For $k=0$ and $k=1$, the denominator on the right hand side turned out to be the zero group; for $k=2$, this is no longer clear, so the natural map arising from the spectral sequence could a priori take its image in a proper quotient of $K_0(C^*_r(G))$, not in $K_0(C^*_r(G))$ itself.  For $k=3$, the situation is similar: one has    
$$
\mu_3:H_3(G)=E^3_{3,0}\to \frac{J_1:\mathcal{I}^4(C_0(G^0))}{J_1:\mathcal{I}^3(C_0(G^0))}
$$
and this map a priori takes values in a quotient of $K_1(C^*_r(G))$ (note that we need to increase to $r=3$ so the condition ``$0\leq p\leq r$'' needed to apply line \eqref{maps} is satisfied; as the differentials on the $E^2$ page are all zero for degree reasons, we still have an identification $H_3(G)=E^3_{3,0}$ however).  For $k=4$, the situation seems worse: one has 
$$
\mu_4:E^4_{4,0}\to \frac{J_0:\mathcal{I}^5(C_0(G^0))}{J_0:\mathcal{I}^4(C_0(G^0))}
$$
but $E^4_{4,0}$ could in principle be a proper subquotient of $H_4(G)$, so a priori one only gets a map from a subquotient of $H_4(G)$ to a subquotient of $K_0(C^*_r(G))$; the situation is similar to this for all $k\geq 4$.  

The recent principal counterexamples to the HK conjecture of Deeley \cite{Deeley22} suggest that the a priori obstructions to the existence of the higher comparison maps discussed above really do pertain; however, we did not yet attempt the relevant computations.
\end{rem}

\subsection{Identification of the comparison maps}\label{id comp sec}

Our goal in this subsection is to identify the map $\mu_0:H_0(G)\to K_0(C^*_r(G))$ of Lemma \ref{mu0 descr} with the comparison map constructed by Matui, and to identify the map $\mu_1:H_1(G)\to K_1(C^*_r(G))$ of Lemma \ref{mu1 descr} with the explicitly constructed comparison map of Theorem \ref{mu1 exist}.  Much of what follows is essentially routine ``book-keeping'' computations; however, as some of it is of quite an involved nature, we thought it was worthwhile to record the details.

The special cases $L^2(G,r)$ and $L^2(G,s)$ of Definition \ref{l2yp} will be particularly important for us: we introduce the shorthand $\K_r:=\K(L^2(G,r))$ and $\K_s:=\K(L^2(G,s))$ for the compact operators on these modules.  These $C^*$-algebras are equipped with the canonical $G$-actions coming from the left action of $G$ on itself for $L^2(G,r)$, and the right action of $G$ on itself for $L^2(G,s)$.  We also write $M_r:C_0(G)\to \K_r$ and $M_s:C_0(G)\to \K_s$; the former is equivariant, while the latter is only equivariant if we consider $C_0(G)$ as a $G$-$C^*$-algebra via the right action of $G$ on itself (however, we will never do this: $C_0(G)$ \emph{always} either has the left $G$-action in what follows, or the trivial action if we have passed through descent).

With notation as in Lemma \ref{mu0 descr}, $\mu_0$ is the map induced on $H_0(G)$ by 
\begin{equation}\label{mu02}
J_0(\delta_0)\circ K_0([X_G])=K_0([X_G]\otimes j_G(\delta_0)),
\end{equation}
where $\delta_0\in KK^G(C_0(G),C_0(G^0))$ equals the counit of adjunction $\epsilon_{C_0(G^0)}$ for the adjoint pair $(E,F)$ discussed at the start of Subsection \ref{pyss sec}.  Using the description of this counit in \cite[Theorem 2.3]{BP22}, one computes that 
\begin{equation}\label{delta0 def}
\delta_0=[M_r,L^2(G,r),0].
\end{equation}

For the statement of the next lemma, let $\iota:C_0(G^0)\to C^*_r(G)$ denote the canonical inclusion.  Recall also (compare \cite[Section 1.3]{BP22}) that the descent of a Hilbert $G$-$A$-module $E$ is defined to be $E\rtimes_r G:=E\otimes_A A\rtimes_r G$.
To have concrete formulas to work with, let $\mathcal{E}$ denote the upper-semicontinuous field of Hilbert modules over $G^0$ associated with $E$. Then $E\rtimes_r G$ can alternatively be constructed as the completion of the vector space of compactly supported continuous sections $\Gamma_c(G;r^*\mathcal{E})$ with respect to the $A\rtimes_r G$-valued inner product
\[
\langle\xi,\xi'\rangle(g):=\sum_{h\in G_{r(g)}}\alpha_{h^{-1}}\big(\langle\xi(h),\xi'(hg)\rangle_{A_{r(h)}}\big),
\]
for $\xi,\xi'\in\Gamma_{c}(G;r^{*}\mathcal{E})$ and $g\in G$.
Moreover, if $\pi:B\rightarrow \mathcal{L}(E)$ is a $G$-equivariant representation of $B$, then for $f\in\Gamma_{c}(G;r^{*}\mathcal{B})$, $\xi\in\Gamma_{c}(G;r^{*}\mathcal{E})$, and $g\in G$, the formula
\[
(f\cdot\xi)(g):=\sum_{h\in G^{r(g)}}\pi_{r(h)}(f(h))W_{h}\big(\xi(h^{-1}g)\big)
\]
defines a representation $\pi\rtimes_r G :B\rtimes_r G\to\mathcal{L}(E\rtimes_r G)$. 

The next lemma follows from direct (if somewhat lengthy) computations that we leave to the reader.

\begin{lemma}\label{desc mor}
For $\xi\in L^2(G,s)$, $f\in C_c(G)$, $g\in G$, and $h\in G^{r(g)}$ define
$$
(\omega(\xi\otimes f)(g))(h):=\xi(h)f(h^{-1}g).
$$
Then $\omega$ extends to a unitary isomorphism of Hilbert $C^*_r(G)$-modules 
$$
\omega:L^2(G,s)\otimes_\iota C^*_r(G)\cong L^2(G,r)\rtimes_r G.
$$
Moreover, for $\kappa_1$ as in Lemma \ref{mor1}, $\omega$ satisfies 
$$
\omega(\kappa_1(a)\otimes 1_{C^*_r(G)})\omega^*=(M_r\rtimes_r G)(a)
$$
for all $a\in C_0(G)\rtimes_r G$.  In particular, there is a $*$-homomorphism $\beta:\K_s\to \K_r\rtimes_rG$ defined by
$$
\beta:\K_s\to \K_r\rtimes_rG,\quad k\mapsto \omega(k\otimes 1_{C^*_r(G)})\omega^*.
$$
Finally, $\omega^*(\K_r\rtimes_r G)\omega=\K(L^2(G,s)\otimes_\iota C^*_r(G))$, and so the element $[L^2(G,s)\otimes_\iota C^*_r(G),\text{ad}_{\omega^*}\circ M_r\rtimes G,0]\in KK(\K_r\rtimes_r G,C^*_r(G))$ is an isomorphism in $KK$, induced by a Morita equivalence bimodule.  \qed
\end{lemma}

We are now ready to show that the description of $\mu_0$ arising from the spectral sequence agrees with the classical map $H_0(G)\to K_0(C^*_r(G))$ introduced by Matui.

\begin{proposition}
The map $\mu_0:H_0(G)\to K_0(C^*_r(G))$ of Lemma \ref{mu0 descr} is the map induced by the canonical inclusion $\iota:C_0(G^0)\to C^*_r(G)$ on the level of $K_0$-groups.
\end{proposition}

\begin{proof}
Lemma \ref{desc mor} implies that $j_G(\delta_0)$ is represented by the class
$$
[\kappa_1\otimes 1_{C^*_r(G)},L^2(G,s)\otimes_\iota C^*_r(G),0]\in KK(C_0(G)\rtimes_r G,C^*_r(G)),
$$
(where we have used the isomorphism $\K_s\cong C_0(G)\rtimes_r G$, which is a special case of Lemma \ref{mor1}).  Lemma \ref{mor1} implies that $[X_G]$ is represented by the opposite of the Morita equivalence Kasparov cycle $(\kappa_1,L^2(G,s),0)$.  Hence $[X_G]\otimes j_G(\delta_0)$ is represented by the tensor product of these, which one computes directly is the cycle $(\iota,C^*_r(G),0)$ in $KK(C_0(G^0),C^*_r(G))$.  The image of this under the $K$-theory functor is indeed the map induced by $\iota$, so we are done.
\end{proof}

We now move on to $\mu_1$.  Let us start by more explicitly describing the exact triangle
\begin{equation}\label{pt diag 4}
\xymatrix{ C_0(G^0)  \ar[rr]|\circ^-{\iota_0^1} & & N_1  \ar[dl]^-{\phi_1} &     \\
& \ar[ul]^-{\delta_0} C_0(G) &   }
\end{equation}
from line \eqref{pt diag 2} above.  Note that the class $[\text{id},L^2(G,r),0]\in KK^G(\K_r,C_0(G^0))$ is an isomorphism (as it arises from an equivariant Morita equivalence) and from the formula in line \eqref{delta0 def} for $\delta_0$ we clearly have 
$$
[M_r,\K_r,0]\otimes [\text{id},L^2(G,r),0]=\delta_0,
$$
so up to replacing $C_0(G^0)$ by $\K_r$ using the isomorphism $[\text{id},L^2(G,r),0]\in KK^G(\K_r,C_0(G^0))$, we may replace the exact triangle in line \eqref{pt diag 4} above with  
\begin{equation}\label{pt diag 5}
\xymatrix{ \K_r  \ar[rr]|\circ^-{\iota_0^1} & & N_1  \ar[dl]^-{\phi_1} &     \\
& \ar[ul]^-{M_r} C_0(G) &   }
\end{equation}
where $M_r\in KK^G(C_0(G),\K_r)$ is the class corresponding to the $*$-homomorphism $M_r:C_0(G)\to \K_r$, and we have abused notation slightly by keeping the same label for the top horizontal map.   The remaining parts of the diagram can be described explicitly in terms of the mapping cone of $M_r$; we recall this next.  

Now, by definition of the exact triangles in $KK^G$ (see \cite[Appendix A.4]{Proietti:2021wz}), in diagram \eqref{pt diag 5} we may take $N_1=C(M_r)$, and $\iota_0^1$ and $\phi_1$ are then the $KK$-classes given by the left and right maps respectively in the canonical short exact sequence
$$
\xymatrix{ 0 \ar[r] & \Sigma \K_r \ar[r]^-\iota & C(M_r) \ar[r]^-{e^0} & C_0(G) \ar[r] & 0 }.
$$
Hence we may replace the first part of the phantom tower in line \eqref{pt diag 3} above with the diagram 
$$
\xymatrix{ \K_r  \ar[rr]|\circ^{\iota} & & C(M_r) \ar[dl]^{e^0} &    \\
& \ar[ul]^-{M_r} C_0(G) &  & \ar[ll]^-{\delta_1}  C_0(G^{(2)}). \ar[ul]^-{\psi_1} }
$$
According to the proof of \cite[Lemma 3.2]{Meyer2008}, the morphism $\psi_1$ appearing above is unique, subject to the condition that the right hand triangle commutes (and the various other conditions defining a phantom tower, which are satisfied by the diagram above).  Our next task towards computing $\mu_1$ is to give an explicit description of $\psi_1$; for this, it will be helpful to see why $M_r\circ \delta_1=0$ (note that this is indeed the case, by definition of an $\mathcal{I}$-projective resolution).

We first find an explicit representation for $\delta_1\in KK^G(C_0(G^{(2)}),C_0(G))$.  Let $\mathrm{pr}_1: G_s\!\times_r G\to G$ be the projection onto the first factor, and define
$$
F_0:=L^2(G_s\!\times_r G,\mathrm{pr}_1).
$$
Similarly, let $\mathrm{pr}_2: G_r\!\times_r G\to G$ be the projection on the second factor and define
$$
F_1:=L^2(G_r\!\times_r G,\mathrm{pr}_2).
$$
Equip $F_0$ with the $G$-action defined by the left translation action of $G$ on the first factor, and $F_1$ with the $G$-action defined by the diagonal left action of $G$; note that both $F_0$ and $F_1$ are then $G$-$C_0(G)$ Hilbert modules, where the action on $C_0(G)$ is (as usual) induced from the left action of $G$ on itself.  Define a representation 
$$
\pi_0:C_0(G^{(2)})\to \mathcal{L}(F_0)
$$
by pointwise multiplication, and a representation 
$$
\pi_1:C_0(G^{(2)})\to \mathcal{L}(F_1)
$$
by 
$$
(\pi_1(f)\xi)(g,h):=f(g,g^{-1}h)\xi(g,h);
$$
both of these representations are equivariant and take values in the compact operators on the corresponding modules.  Using the explicit description of the induction functor from \cite[Section 2.1]{BP22} and of the co-unit of adjunction from \cite[Theorem 2.3]{BP22}, one computes that the elements $L(\epsilon_{C_0(G^0)})$ and $\epsilon_{L(C_0(G^0))}$ of $KK^G(C_0(G^{(2)}),C_0(G))$ satisfy 
$$
\epsilon_{L(C_0(G^0))}=[\pi_0,F_0,0]
$$
and 
$$
L(\epsilon_{C_0(G^0)})=[\pi_1,F_1,0];
$$
from the formula in line \eqref{boundary}, we thus have that 
\begin{equation}\label{delta1 diff}
\delta_1=[\pi_0,F_0,0]-[\pi_1,F_1,0].
\end{equation}
On the other hand, from the fact that line \eqref{pr} is a projective resolution, we have that $\delta_1\otimes \delta_0=0$; as we are identifying $\delta_0$ with $[M_r,\K_r,0]\in KK^G(C_0(G),\K_r)$, we therefore see that 
$$
[\pi_0\otimes 1_{\K_r},F_0\otimes_{M_r} \K_r,0]=[\pi_1\otimes 1_{\K_r},F_1\otimes_{M_r}\K_r,0]
$$
in $KK^G(C_0(G^{(2)}),\K_r)$.  In order to construct $\psi_1$ and also its image under the descent morphism, we need to make this identity explicit; this is our next task.

We may identify $L^2(G,s)\otimes_\iota C^*_r(G)$ with the completion of $C_c(G_s\!\times_rG)$ for the inner product valued in the dense subset $C_c(G)$ of $C^*_r(G)$ defined by 
$$
\langle \xi,\eta\rangle_s(g):= \sum_{h\in G^{r(g)}}\sum_{k\in G_{s(h)}}\overline{\xi(k,h^{-1})}\eta(k,h^{-1}g).
$$ 
Similarly, $L^2(G,r)\otimes_\iota C^*_r(G)$ identifies with the completion of $C_c(G_r\!\times_rG)$ for the inner product defined by 
$$
\langle \xi,\eta\rangle_r(g):= \sum_{h\in G^{r(g)}}\sum_{k\in G^{s(h)}}\overline{\xi(k,h^{-1})}\eta(k,h^{-1}g).
$$
Given these descriptions, direct checks show that the map 
$
G_r\!\times_rG\to G_s\!\times_rG$, $(g,h)\mapsto (g,g^{-1}h)
$
induces a unitary $C^*_r(G)$-module isomorphism
$$
w:L^2(G,s)\otimes_\iota C^*_r(G)\to L^2(G,r)\otimes_\iota C^*_r(G).
$$

\begin{lemma}\label{unitary}
There is a unitary, equivariant, isomorphism 
$$
u:F_0\otimes_{M_r} \K_r\to F_1\otimes_{M_r}\K_r
$$
such that $u(\pi_0\otimes 1_{\K_r})u^*=\pi_1\otimes 1_{\K_r}$.

Moreover if $L^2(G^{(2)},\mathrm{pr})$ and $\kappa_2$ are as in Lemma \ref{mor1}, then there is a commutative diagram
$$\xymatrixcolsep{1pc}
\xymatrix{ \big(F_0\otimes_{M_r} \K_r\big)\rtimes_r G \otimes_{id} L^2(G,r)\rtimes_r G \ar[rr]^-{j_G(u)\otimes 1} \ar[d]^{\Theta_0} & &  \big(F_1\otimes_{M_r}\K_r\big)\rtimes_r G \otimes_{id} L^2(G,r)\rtimes_r G \ar[d]^{\Theta_1} \\
L^2(G^{(2)},\mathrm{pr})\otimes_{M_s} L^2(G,s)\otimes_\iota C^*_r(G) \ar[rr]^-{1\otimes w} & & L^2(G^{(2)},\mathrm{pr})\otimes_{M_r} L^2(G,r)\otimes_\iota C^*_r(G) }
$$
where all the maps are unitary isomorphisms.  The maps $\Theta_0$ and $\Theta_1$ satisfy 
$$
\Theta_0(\pi_0\rtimes_r G)\Theta_0^*=\kappa_2\otimes 1_{\K_s\otimes_\iota C^*_r(G)}
$$
and 
$$
\Theta_1(\pi_1\rtimes_r G)\Theta_1^*=\kappa_2\otimes 1_{\K_r\otimes_\iota C^*_r(G)}
$$
respectively.  The map $w:L^2(G,s)\otimes_\iota C^*_r(G)\to L^2(G,r)\otimes_\iota C^*_r(G)$ is that defined above, and satisfies
$$
w(M_s\otimes 1_{C^*_r(G)})w^*=M_r\otimes 1_{C^*_r(G)}.
$$
\end{lemma}

\begin{proof}
We begin by defining a unitary isomorphism of Hilbert $C_0(G^0)$-modules
$$U:F_0\otimes_{M_r} L^2(G,r)\rightarrow F_1\otimes_{M_r} L^2(G,r)$$
such that $U(\pi_0\otimes 1)U^*=\pi_1\otimes 1$ as follows: a straightforward computation gives identifications
\begin{equation}\label{eq:comp}
    F_0\otimes_{M_r} L^2(G,r)\cong L^2(G^{(2)},r\circ \mathrm{pr}_1),\quad  F_1\otimes_{M_r} L^2(G,r)\cong L^2(G_r\!\times_rG,r\circ \mathrm{pr}_2).
\end{equation}
It is then easy to see that the map $G_r\!\times_rG\rightarrow G^{(2)}$, $(g,h)\mapsto (g,g^{-1}h)$ gives rise to the desired equivalence.

Now recall, that $\K_r$ splits as a tensor product $L^2(G,r)\otimes_{C_0(G^0)} L^2(G,r)^{\mathrm{op}}$.
After making this identification we can define $u$ as 
$$F_0\otimes_{M_r} L^2(G,r)\otimes_{C_0(G^0)} L^2(G,r)^{\mathrm{op}}\stackrel{U\otimes 1}{\rightarrow} F_1\otimes_{M_r} L^2(G,r)\otimes_{C_0(G^0)} L^2(G,r)^{\mathrm{op}}.$$
We now move on to the commutative diagram. 
Notice that $(F_0\otimes_{M_r} \mathrm{K}_r)\rtimes_r G\otimes (L^2(G,r)\rtimes_r G)\cong (F_0\otimes_{M_r}L^2(G,r))\rtimes_r G$,
and similarly 
$(F_1\otimes_{M_r} \mathrm{K}_r)\rtimes_r G\otimes (L^2(G,r)\rtimes_r G)\cong (F_1\otimes_{M_r}L^2(G,r))\rtimes_r G$.
Applying these identifications to the top row of the diagram in the lemma and using identifications similar to the ones in line \eqref{eq:comp} in the bottom row (and slightly abusing notation by still denoting the maps $\Theta_0$ and $\Theta_1$) shows that it will be enough to exhibit a commutative diagram of the form
$$
\xymatrix{
\big(F_0\otimes_{M_r} L^2(G,r)\big)\rtimes_r G \ar[rr]^-{U\rtimes_r G} \ar[d]^{\Theta_0} & &  \big(F_1\otimes_{M_r}L^2(G,r)\big)\rtimes_r G \ar[d]^{\Theta_1} \\
L^2(G^{(2)},s\circ \mathrm{pr})\otimes_\iota C^*_r(G) \ar[rr]^-{W} & & L^2(G^{(2)},r\circ \mathrm{pr})\otimes_\iota C^*_r(G) }
$$
where $W$ is induced by the map 
$G^{(2)}\times_{s\circ \mathrm{pr},r}G\rightarrow G^{(2)}\times_{r\circ \mathrm{pr},r}G$, $(h_1,h_2,g)\mapsto (h_1,h_2,h_2^{-1}g)$.
Similar to the discussion just prior to the present lemma, the modules involved in the diagram all have a canonical dense subspace of compactly supported functions defined on a suitable fibred product of $G^{(2)}$ or $G_r\!\times_rG$ with $G$.
Hence it will be enough to describe a ``dual'' commutative diagram of homeomorphisms

\[\begin{tikzcd}
	{G^{(2)}_{r\circ \mathrm{pr}_1}\!\times_r G} & {(G_r\!\times_rG)_{r\circ \mathrm{pr}_2}\!\times_r G } \\
	{G^{(2)}_{s\circ \mathrm{pr}_2}\!\times_r G} & {G^{(2)}_{r\circ \mathrm{pr}_2}\!\times_r G}
	\arrow["{\theta_1}", from=2-2, to=1-2]
	\arrow["{\theta_0}", from=2-1, to=1-1]
	\arrow["\jupiter", from=1-2, to=1-1]
	\arrow["\omega", from=2-2, to=2-1]
\end{tikzcd}\]
The map $\jupiter$ inducing $U\rtimes_r G$ is given by
$$\jupiter(h_1,h_2,g)=(h_1,h_1^{-1}h_2,g)$$
and the map $\omega$ inducing $W$ is given by
$$\omega(h_1,h_2,g)=(h_1,h_2,h_2^{-1}g).$$
If we set $\theta_0(h_1,h_2,g):=(h_1,h_2,h_1h_2g)$ and
$\theta_1(h_1,h_2,g):=(h_1,h_1^{-1}h_2,h_1g)$, then the diagram commutes. It follows that the induced diagram of Hilbert modules commutes. Moreover, direct checks show that $\Theta_0$ intertwines the representations $\pi_0\rtimes_r G$ and $\kappa_2\otimes 1_{L^2(G,s)\otimes_\iota C_r^*(G)}$ and $\Theta_1$ intertwines $\pi_1\rtimes_r G$ and $\kappa_2\otimes 1_{L^2(G,r)\otimes_\iota C_r^*(G)}$.
\end{proof}

Using this lemma, we may finally compute an explicit formula for both $\psi_1$ and its descent $j_G(\psi_1)$.  Define  
$$
F_{\K}:= I \big((F_0\otimes_{M_r}\K_r)\oplus (F_1\otimes_{M_r}\K_r)\big)
$$
with the $I \K_r$-valued inner product given by adding the componentwise inner products pointwise for each $t\in[0,1]$.
Define $F$ to be the collection of all triples $(\xi_0,\eta,\xi_1)$ in the direct sum $F_0\oplus F_{\K}\oplus F_1$ such that\footnote{Note that even though $\K_r$ is typically non-unital, $\xi_i\otimes 1_{\K_r}$ still makes sense as an element of $F_i\otimes_{M_r} \K_r$ for $i\in \{0,1\}$ - for example, one can show that if $(u_j)_{j\in J}$ is an approximate unit for $\K_r$, then $(\xi_i\otimes u_j)_{j\in J}$ is Cauchy in $F_i\otimes_{M_r}\K_r$, and then define $\xi_i\otimes 1_{\K_r}$ to be its limit.}
$$
(\xi_0\otimes 1_{\K_r} ,0) =\eta(0) \quad\text{and} \quad \eta(1)=(0, \xi_1\otimes 1_{\K_r}).
$$
Let $D(M_r)$ be as in line \eqref{d phi}, and define a $D(M_r)$-valued inner product on $F$ by
$$
\langle (\xi_0^{(0)},\eta^{(0)},\xi_1^{(0)})~,~(\xi_0^{(1)},\eta^{(1)},\xi_1^{(1)})\rangle := \Big( \langle\xi_0^{(0)},\xi_0^{(1)}\rangle_{F_0}~,~ \langle \eta^{(0)},\eta^{(1)}\rangle_{F_\K}~,~\langle\xi_1^{(0)},\xi_1^{(1)}\rangle_{F_1} \Big)
$$  
Direct checks then show that $F$ is a well-defined Hilbert $G$-$D(M_r)$-module.

Now, with $u$ as in Lemma \ref{unitary} consider the unitary 
$$
V:=\begin{pmatrix} 0 & u^* \\ u & 0 \end{pmatrix}\in \mathcal{L}\big((F_0\otimes_{M_r}\K_r)\oplus (F_1\otimes_{M_r}\K_r)\big),
$$
which commutes with the direct sum action of $G$.  This is a self-adjoint unitary, so by the spectral theorem we can write $V=p-q$\footnote{There are also concrete formulas: $p=\frac{1}{2}\begin{pmatrix} 1& u^* \\ u & 1\end{pmatrix}$ and $q=\frac{1}{2}\begin{pmatrix} 1& -u^* \\ -u & 1\end{pmatrix}$.}, where $p$ and $q$ are orthogonal projections that commute with the $G$-action and satisfy $p+q=1$.  For $t\in [0,1]$, define $V_t=p+e^{i\pi t}q$, so $\{V_t\}_{t\in [0,1]}$ is a path of $G$-invariant unitaries connecting $V_0=1$ and $V_1=V$.  For $f\in C_0(G^{(2)})$, $\eta\in F_\K$, and $t\in [0,1]$, define 
$$
(M_\K(f)\eta)(t):=V_t\big(\pi_0(f)\otimes 1_{\K_r},0\big)V_t^*.
$$
This defines a $G$-equivariant representation $\pi_\K:C_0(G^{(2)})\to F_\K$.  Moreover, it is compatible with the representations $\pi_0$ and $\pi_1$ of $C_0(G^{(2)})$ on $F_0$ and $F_1$ respectively, in the sense that for $f\in C_0(G^{(2)})$ and $(\xi_0,\eta,\xi_1)\in F$, the formula
$$
\pi_F(f)(\xi_0,\eta,\xi_1):=(\pi_0(f)\xi_0,\pi_\K(f)\eta,\pi_1(f)\xi_1)
$$
defines a Hilbert $G$-$D(M)$-module representation $\pi_F:C_0(G^{(2)})\to \mathcal{L}(F)$.  This representation takes values in the compact operators on $F$, and thus we get a well-defined Kasparov element 
$$
\widetilde{\psi}:=[\pi_F,F,0]\in KK^G(C_0(G^{(2)}),D(M_r)).
$$
The short exact sequence of line \eqref{c d ses} and split exactness of $KK^G$-theory then gives a canonical isomorphism 
\begin{equation}\label{dm vs cm}
KK^G(C_0(G^{(2)}),D(M_r))\cong KK^G(C_0(G^{(2)}),C_0(G))\oplus KK^G(C_0(G^{(2)}),C(M_r)).
\end{equation}
Write $\psi$ for the image of $\widetilde{\psi}$ in $KK^G(C_0(G^{(2)}),C(M_r))$ under the canonical quotient map arising from the direct sum decomposition above.

\begin{lemma}\label{p tower}
The element $\psi$ fits into the canonical phantom tower in $KK^G$ as the map labeled $\psi_1$
\begin{equation}\label{ghost}
\xymatrix{ \K_r  \ar[rr]|\circ^{\iota} & & C(M_r) \ar[dl]^-{e^0} &    \\
& \ar[ul]^-{M_r} C_0(G) &  & \ar[ll]^-{\delta_1}  C_0(G^{(2)}). \ar[ul]^-{\psi_1} }
\end{equation}
\end{lemma}

\begin{proof}
As we already noted, the proof of \cite[Lemma 3.2]{Meyer2008} shows that $\psi_1$ always exists, and is uniquely determined by the property $\psi_1\otimes [e^0]=\delta_1$.  We thus need to show that $\psi\otimes [e^0]=\delta_1$.  

We have canonical quotient maps $f^j:D(M_r)\to C_0(G)$ defined by $f^j:(a_0,b,a_1)\mapsto a_j$ for $j\in \{0,1\}$.  Let $s:C_0(G)\to D(M_r)$, $a\mapsto (a,\mathsf{c}(M_r(a)),a)$ be the canonical splitting of the short exact sequence from line \eqref{c d ses}.  Clearly $f^0\circ s=f^1\circ s$, whence the map 
$$
\cdot\otimes ([f^0]-[f^1]) : KK^G(C_0(G^{(2)}),D(M))\to KK^G(C_0(G^{(2)}),C_0(G))
$$
vanishes on $s_*(KK^G(C_0(G^{(2)}),C_0(G)))$.  Let $i:C(M_r)\to D(M_r)$, $(a_0,b)\mapsto (a_0,b,0)$ be the canonical inclusion.  Then according to the isomorphism in line \eqref{dm vs cm} we have $\widetilde{\psi}=i_*(\psi)\oplus s_*(\alpha)$ for some $\alpha\in KK^G(C_0(G^{(2)}),C_0(G))$, so the above-discussed vanishing of $\cdot\otimes ([f^0]-[f^1])$ on the image of $s_*$ gives
$$
\widetilde{\psi}\otimes ([f^0]-[f^1])=i_*(\psi)\otimes ([f^0]-[f^1])=\psi\otimes ([f^0\circ i]-[f^1\circ i]).
$$
On the other hand, we clearly have $f^0\circ i=e^0$ and $f^1\circ i=0$, so the above implies that 
$$
\widetilde{\psi}\otimes ([f^0]-[f^1])=\psi\otimes [e^0]
$$
It thus suffices to show that $\widetilde{\psi}\otimes ([f^0]-[f^1])=\delta_1$.  For this, we note that $\widetilde{\psi}[f^j]=[F_j,\pi_j,0]$ for $j\in \{0,1\}$ so we need to show that 
$$
\delta_1=[\pi_0,F_0,0]-[\pi_1,F_1,0];
$$
this is exactly the formula in line \eqref{delta1 diff}, so we are done.
\end{proof}

Our next goal is to compute the image of the diagram in line \eqref{ghost} under descent.  Unfortunately, this necessitates more notation.  Let $\iota:C_0(G^0)\to C^*_r(G)$ denote the canonical inclusion, and let $C(\iota)$ and $D(\iota)$ be the corresponding $C^*$-algebras from lines \eqref{c phi} and \eqref{d phi}.  Define $X_D$ to consist of all triples $(\xi_0,\eta,\xi_1)$ in 
$$
L^2(G,s)\oplus I \big(L^2(G,s)\otimes_\iota C^*_r(G)\big)\oplus L^2(G,s)
$$
such that $\xi_i\otimes 1_{C^*_r(G)}=\eta(i)$ for $i\in \{0,1\}$.  Direct checks based on Lemma \ref{desc mor} shows that this is canonically a Morita equivalence $D(M_r)\rtimes_r G$-$D(\iota)$-bimodule.   Similarly, if $X_C$ consists of all pairs in $L^2(G,s)\oplus I\big(L^2(G,s)\otimes_\iota C^*_r(G)\big)$ such that $\xi_0\otimes 1_{C^*_r(G)}=\eta(0)$ and $\eta(1)=0$, we see that $X_C$ is a Morita equivalence $C(M_r)\rtimes_r G$-$C(\iota)$ bimodule.  Moreover, if $X_G$ is the Morita equivalence $C_0(G)\rtimes_rG$-$C_0(G^0)$ bimodule from Lemma \ref{mor1}, then the following diagram (built from the general short exact sequence of line \eqref{c d ses}) is easily seen to commute in $KK$
\begin{equation}\label{c d mor}
\xymatrix{ 0 \ar[r] \ar@{=}[d] & C(M_r)\rtimes_r G \ar[d]_{\cong}^-{X_C} \ar[r] & D(M_r)\rtimes_r G \ar[d]_{\cong}^-{X_D} \ar[r] & C_0(G)\rtimes_r G \ar[r] \ar[d]_{\cong}^-{X_G}  & 0 \ar@{=}[d] \\
0 \ar[r] & C(\iota) \ar[r] & D(\iota) \ar[r] & C_0(G^0) \ar[r] & 0. }
\end{equation}

For ease of notation let $E_s$ and $E_r$ denote the Hilbert $C_r^*(G)$-modules $L^2(G,s)\otimes_\iota C_r^*(G)$ and $L^2(G,r)\otimes_\iota C_r^*(G)$, respectively.
Let $E$ be the collection of triples $(\xi_0,\eta,\xi_1)$ in 
\begin{equation}\label{epsi}
L^2(G,s)\oplus I \big(E_s\oplus E_r\big)\oplus L^2(G,r)
\end{equation}
such that $(\xi_0\otimes 1_{C^*_r(G)},0)=\eta(0)$ and $(0,\xi_1\otimes 1_{C^*_r(G)})=\eta(1)$; this is a Hilbert $D(\iota)$-module in the natural way.  It is moreover equipped with a left $C_0(G)$-action defined as follows.   Let $w: E_s\to E_r$ be the unitary isomorphism from Lemma \ref{unitary}, and define
$$
v:=\begin{pmatrix} 0 & w^* \\ w & 0 \end{pmatrix} \in \mathcal{L}(E_s\oplus E_r).
$$
As in the discussion defining $\widetilde{\psi}$, we may write $v=p-q$ for complementary projections $p$ and $q$, and define $v_t:=p+e^{i\pi t}q$.  Then the formula 
\begin{equation}\label{pie}
\pi_E:=M_s\oplus v_t(M_s\otimes 1_{C^*_r(G)} ,0)v_t^* \oplus M_r
\end{equation}
defines the desired $C_0(G)$-action on $E$.  Define $\widetilde{\Psi}:=[\pi_E,E,0]\in KK(C_0(G),D(\iota))$, and define $\Psi$ to be the image of $\widetilde{\Psi}$ in $KK(C_0(G),C(\iota))$ under the canonical quotient map arising from the direct sum decomposition 
$$
KK(C_0(G),D(\iota))\cong KK(C_0(G),C(\iota))\oplus KK(C_0(G),C_0(G^0))
$$
that in turn arises from the split short exact sequence in line \eqref{c d ses}.  

The next lemma is the last main ingredient needed to compute $\mu_1$.  To state it, let $\iota:\Sigma C^*_r(G)\to C(\iota)$ and $e^0:C(\iota)\to C_0(G^0)$ be the canonical maps associated to the mapping cone.  Let also $[s]$ and $[r]$ respectively denote the elements $[M_s,L^2(G,s),0]$ and $[M_r,L^2(G,r),0]$ of $KK(C_0(G),C_0(G^0))$.

\begin{lemma}\label{p tower desc}
After applying the canonical Morita isomorphisms $C(M_r)\rtimes_r G \stackrel{X_C}{\sim} C(\iota)$ discussed above, the Morita isomorphism $\K_r\rtimes_r G\stackrel{X_{G,\K}}{\sim} C^*_r(G)$ of Lemma \ref{desc mor}, and $C_0(G)\rtimes_r G\stackrel{X_G}{\sim} C_0(G^0)$ and $C_0(G^{(2)})\rtimes_r G\stackrel{X_{G^{(2)}}}{\sim} C_0(G)$ of Lemma \ref{mor1}, the image of the commutative diagram \eqref{ghost} above under descent $j_G:KK^G\to KK$ identifies with 
\begin{equation}\label{ghost2}
\xymatrix{ C^*_r(G)  \ar[rr]|\circ^{\iota} & & C(\iota) \ar[dl]^-{e^0} &    \\
& \ar[ul]^-{\iota} C_0(G^0) &  & \ar[ll]^-{[s]-[r]}  C_0(G) \ar[ul]^-{\Psi} }.
\end{equation}
\end{lemma}

\begin{proof}
That the left triangle in line \eqref{ghost} has image equal to the left triangle in line \eqref{ghost2} under descent (and modulo the given Morita equivalences) follows from the commutative diagram of short exact sequences in $KK$
\begin{equation}\label{two cones}
\xymatrix{ 0 \ar[r] \ar@{=}[d] &\Sigma \K_r\rtimes_r G  \ar[d]_{\cong}^-{X_{G,\K}} \ar[r] & C(M_r)\rtimes_r G \ar[d]_{\cong}^-{X_C} \ar[r] & C_0(G)\rtimes_r G \ar[r] \ar[d]_{\cong}^-{X_G}  & 0 \ar@{=}[d] \\
0 \ar[r] & \Sigma C^*_r(G) \ar[r] & C(\iota) \ar[r] & C_0(G^0) \ar[r] & 0. }
\end{equation}

We next claim that $j_G(\psi)\otimes_{C(M_r)\rtimes_rG} [X_C]= [X_{G^{(2)}}]\otimes_{C_0(G)}\Psi$.  Thanks to the commutative diagram in line \eqref{c d mor} above, it suffices to show that 
\begin{equation}\label{desid}
j_G(\widetilde{\psi})\otimes_{D(M_r)\rtimes_rG} [X_D]= [X_{G^{(2)}}]\otimes_{C_0(G)}\widetilde{\Psi}.
\end{equation}
This will moreover show that the bottom horizontal arrow in line \eqref{ghost2} is correctly labeled, as it is clear that if $f^0,f^1:D(\iota)\to C_0(G^0)$ are the canonical $*$-homomorphisms, then $\widetilde{\Psi}\otimes ([f^0]-[f^1])=[s]-[r]$; thus to complete the proof it suffices to establish the identity in line \eqref{desid}.

The Kasparov product $j_G(\widetilde{\psi})\otimes_{D(M_r)\rtimes_rG} [X_D]$ is represented by the triple
$$\big(\pi_F\rtimes_r G\otimes 1,F\rtimes_r G\otimes X_D, 0\big).$$
Our first goal is to identify this triple with the triple
$$
\big(\kappa_2\otimes  (1\oplus v_t(1,0)v_t^*\oplus 1),L^2(G^{(2)},\mathrm{pr})\otimes_{\Tilde{\pi}_E}E,0\big), 
$$
where we emphasise that we are using the representation $\Tilde{\pi}_E=(M_s,I(M_s\otimes 1,M_r\otimes 1),M_r)$ (as opposed to $\pi_E$ defined in line (\ref{pie})).

In fact we will deal with the ambient modules of $E$ and $F$ respectively, which allows us to treat each component separately to improve readability.
We first deal with the first and third components. For these we have isomorphisms
\begin{equation}\label{desc delta 01}
    (F_0\rtimes_r G)\otimes_{\kappa_1}L^2(G,s)\rightarrow L^2(G^{(2)},\mathrm{pr})\otimes_{M_s} L^2(G,s)
\end{equation}
\begin{equation}\label{desc delta 11}
    (F_1\rtimes_r G)\otimes_{\kappa_1}L^2(G,s)\rightarrow L^2(G^{(2)},\mathrm{pr})\otimes_{M_r} L^2(G,r)
\end{equation}
identifying the first and third components of $F\rtimes_r G\otimes X_D$ and $L^2(G^{(2)},\mathrm{pr})\otimes_{\Tilde{\pi}_E} E$, respectively.
These isomorphisms are produced in a similar fashion so will only explain the procedure for (\ref{desc delta 11}):
We apply the isomorphism $\omega$ from Lemma \ref{desc mor} to the transformation groupoid $G_r\!\times_{r} G$ of the left action of $G$ on itself (in place of $G$) to obtain an isomorphism
$$L^2(G_r\!\times_{r} G,S)\otimes_{C_0(G)} C_0(G)\rtimes_{\mathrm{lt}}G\stackrel{\cong}{\rightarrow} F_1\rtimes_r (G_r\!\times_{r} G)\cong F_1\rtimes_r G,$$
where $S:G_r\!\times_{r} G\rightarrow G$ denotes the source map of $G_r\!\times_{r} G$ given by $S(g,h)=g^{-1}h$.
Using this we get a chain of identifications
\begin{align*}
(F_1\rtimes_r G) \otimes_{\kappa_1} L^2(G,s)& \cong
L^2(G_r\!\times_{r} G, S)\otimes_{M_s} C_0(G)\rtimes_rG\otimes_{\kappa_1} L^2(G,s)\\
&\cong L^2(G_r\!\times_{r} G,S)\otimes_{M_s} L^2(G,s)\\
&\cong L^2(G^{(2)},\mathrm{pr}) \otimes_{M_s} L^2(G,s)
\end{align*}
A tedious but routine calculation shows that this isomorphism intertwines the actions $\pi_1\rtimes G \otimes 1$ and $\kappa_2\otimes 1$.

It remains to identify the middle components of $F\rtimes_r G\otimes X_D$ and $L^2(G^{(2)},\mathrm{pr})\otimes_{\Tilde{\pi}_E} E$, respectively.
For every $t\in[0,1]$ we have the following chain of isomorphisms, where we use
Lemma \ref{desc mor} in line 2 and the isomorphisms $\Theta_0$ and $\Theta_1$ from Lemma \ref{unitary} in line 4.
\begin{align*} &(F_0\otimes_{M_r}\K_r\oplus F_1\otimes_{M_r} \K_r)\rtimes_r G\otimes_{Ad\, \omega} E_s \\ \cong & (F_0\otimes_{M_r}\K_r\oplus F_1\otimes_{M_r} \K_r)\rtimes_r G\otimes_{id} (L^2(G,r)\rtimes_r G) \\
 \cong & \big((F_0\otimes_{M_r}\K_r)\rtimes_r G \otimes_{id}L^2(G,r)\rtimes_r G\big)\oplus\big((F_1\otimes_{M_r}\K_r)\rtimes_r G \otimes_{id}L^2(G,r)\rtimes_r G\big)\\
 \cong & L^2(G^{(2)},\mathrm{pr})\otimes_{M_s \otimes 1}E_s \oplus L^2(G^{(2)},\mathrm{pr})\otimes_{M_r \otimes 1} E_r\\
 \cong & L^2(G^{(2)},\mathrm{pr})\otimes_{(M_s \otimes 1,M_r \otimes 1)}(E_s\oplus E_r)
\end{align*}
The commutative diagram in Lemma \ref{unitary} together with the construction of the families of unitaries $(V_t)_t$ and $(v_t)_t$ imply that $(\Theta_0\oplus \Theta_1)((V_t\rtimes_r G\otimes 1_{L^2(G,r)\rtimes G})=(1_{L^2(G^{(2)},\mathrm{pr})}\otimes v_t)(\Theta_0\oplus \Theta_1)$ and hence the chain of isomorphisms above intertwines the representations $(V_t(\pi_0\otimes 1,0)V_t^*)\rtimes_r G\otimes 1$ and $\kappa_2\otimes v_t(1,0)v_t^*$ of $C_0(G^{(2)})\rtimes_r G$. This completes the identification of the Kasparov triples.

Finally, we apply a standard trick in $KK$-theory to replace the representing module by a non-degenerate one, i.e. we pass to the module
$$\overline{\big(\kappa_2\otimes  (1,v_t(1,0)v_t^*,1)\big)\big(L^2(G^{(2)},\mathrm{pr})\otimes_{\Tilde{\pi}_E} E\big)}.
$$
The latter module however is easily seen to be isomorphic to
$$L^2(G^{(2)},\mathrm{pr})\otimes_{\pi_E}E,
$$
in such a way that the isomorphism intertwines the representations $\kappa_2\otimes  (1,v_t(1,0)v_t^*,1)$ and $\kappa_2\otimes 1_E$.
\end{proof}

Finally, we are ready to give our concrete formula for the comparison map $\mu_1$. 
 
\begin{proposition}\label{mu1 final}
The map $\mu_1$ from Lemma \ref{mu1 descr} agrees with the canonical homomorphism from Theorem \ref{mu1 exist}.
\end{proposition}

\begin{proof}
According to the description in Theorem \ref{mu1 exist}, it will suffice to show that for any compact open bisection $V$ in $G$, if $\eta$ is as in Theorem \ref{p:relKthunital} we have that $\eta([1_V]\otimes \Psi)=[1_V^*]$ in $V_\infty(C_0(G^0),C^*_r(G))/\approx$.  

Now, if we write $\lambda_V:\C\to \mathcal{L}(\pi_E(1_V)E)$ for the unital scalar representation, then 
$$ 
[1_V]\otimes \widetilde{\Psi} = [\lambda_V, \pi_E(1_V)E , 0].
$$ 
Recalling (see lines \eqref{epsi} and  \eqref{pie}) that $\pi_E=M_s\oplus v_t(M_s\otimes 1_{C^*_r(G)},0)v_t^*\oplus M_r$, we compute that $M_s(1_V)\cdot L^2(G,s)\cong 1_{s(V)}C_0(G^0)$ as a right $C_0(G^0)$-module, and that $M_r(1_V)\cdot L^2(G,r)\cong 1_{r(V)}C_0(G^0)$ as a right $C_0(G^0)$-module.  On the other hand $v_t(M_s(1_V)\otimes 1_{C^*_r(G)},0)v_t^*\cdot (L^2(G,s)\otimes_\iota C^*_r(G)\oplus L^2(G,r)\otimes_\iota C^*_r(G))$ is isomorphic to $p_t(C^*_r(G)\oplus C^*_r(G))$ as a right $C^*_r(G)$-module, where 
$$
p_t:=\frac{1}{2}\begin{pmatrix} (1+\cos(\pi t))1_{s(V)} & -i\sin(\pi t)1_V^* \\ i\sin(\pi t)1_V & (1-\cos(\pi t)) 1_{r(V)}\end{pmatrix}\in M_2(C^*_r(G)).  
$$
It follows that $[1_V]\otimes \widetilde{\Psi}\in KK(\C,D(\iota))$ is represented by the class of the projection 
$$
\begin{pmatrix} 1_{s(V)} & 0 \\ 0 & 0 \end{pmatrix} \oplus (p_t)_{t\in [0,1]} \oplus \begin{pmatrix} 0 & 0 \\ 0 & 1_{r(V)} \end{pmatrix}\in M_2(D(\iota))
$$
On the other hand, note that $p_t=u_tp_0u_t^*$, where
$$
u_t:=\frac{1}{2}\begin{pmatrix} 2-1_{s(V)}(1-e^{i \pi  t}) & 1_{V}^*(1-e^{i \pi t }) \\ 1_V(1-e^{i\pi t}) & 2 - 1_{r(V)}(1-e^{i \pi t}) \end{pmatrix},
$$
so $(u_t)_{t\in [0,1]}$ defines a continuous path of unitaries in $M_2(C^*_r(G))$ (or in the unitization of this $C^*$-algebra if it is not unital).  It follows from the definition of $\eta$ given in line (\ref{etadef}) that 
$$
\eta([1_V]\otimes \widetilde{\Psi} )=\Bigg[\begin{pmatrix} 1_{s(V)} & 0 \\ 0 & 0 \end{pmatrix} u_1\Bigg].
$$
Computing, 
$$
\begin{pmatrix} 1_{s(V)} & 0 \\ 0 & 0 \end{pmatrix} u_1=\begin{pmatrix} 0 & 1_{V}^* \\ 0 & 0 \end{pmatrix},
$$
so we see that $\eta([1_V]\otimes \widetilde{\Psi})=[1_{V}^*]$.
\end{proof}

\subsection{The HK-conjecture in low dimensions}\label{asdim}

In this subsection, we apply the previous result and Theorem \ref{main vanishing} to deduce consequences for the HK conjecture.

As discussed in Subsection \ref{pyss sec}, a recent result of Proietti and Yamashita in \cite{Proietti:2021wz} established the existence of a convergent spectral sequence
\begin{equation}\label{Equation:Spectral sequence}
E_{p,q}^2=H_p(G,K_q(A))\Rightarrow K_{p+q}(A\rtimes_r G)
\end{equation}
for any $G$-algebra $A$, provided that $G$ is a second countable ample groupoid with torsion free isotropy, which satisfies the strong Baum-Connes conjecture.

Combining this with our results from previous sections we obtain the following application to the HK-conjecture.  The reader should compare this to \cite[Remark 4.5]{Proietti:2021wz}, which establishes a similar result on the HK conjecture under a vanishing hypothesis on $H_k(G)$ for $k\geq 3$; the main difference between Theorem \ref{Cor:HK conjecture up to dimension 2} below and \cite[Remark 4.5]{Proietti:2021wz} is that the former gives a concrete criterion when vanishing holds.

\begin{theorem}\label{Cor:HK conjecture up to dimension 2}
	Let $G$ be a second countable principal ample groupoid with dynamic asymptotic dimension at most two.
	Then there is a short exact sequence
	$$0\rightarrow H_0(G)\stackrel{\mu_0}{\rightarrow} K_0(C_r^*(G))\rightarrow H_2(G)\rightarrow 0,$$
	and $\mu_1:H_1(G)\rightarrow K_1(C_r^*(G))$ is an isomorphism.
	If moreover $H_2(G)$ is free (e.g.\ if it is finitely generated), then the HK-conjecture holds for $G$, i.e.
	$$K_0(C_r^*(G))\cong H_0(G)\oplus H_2(G)\text{ and }K_1(C_r^*(G))\cong H_1(G).$$
\end{theorem}

\begin{proof}
	First of all we can apply the spectral sequence (\ref{Equation:Spectral sequence}) since $G$ is principal and any groupoid with finite dynamic asymptotic dimension is in particular amenable (this follows from the proof of \cite[Theorem A.9]{Guentner:2014bh}), and hence satisfies the strong Baum-Connes conjecture by the main result of \cite{Tu:1999bq}.
	Since $H_n(G)=0$ for all $n\geq 3$ by Theorem \ref{main vanishing} the spectral sequence (\ref{Equation:Spectral sequence}) collapses on the second page and
	we conclude that $K_1(C_r^*(G))\cong H_1(G)$ and that $K_0(C_r^*(G))$ fits into a short exact sequence
	$$0\rightarrow H_0(G)\rightarrow K_0(C_r^*(G))\rightarrow H_2(G)\rightarrow 0.$$
	
If moreover $H_2(G)$ is free abelian (which holds if it is finitely generated, as it is torsion-free by Theorem \ref{main vanishing}), then the sequence above splits and we get a direct sum decomposition $K_0(C_r^*(G))\cong H_0(G)\oplus H_2(G)$.
\end{proof}

With a view towards the classification program for simple nuclear $C^*$-algebras, we obtain the following consequence.
\begin{corollary}
	Let $G$ be a second countable, principal, ample groupoid with compact base space and dynamic asymptotic dimension at most one. Then 
	$$\mathrm{Ell}(C_r^*(G))=(H_0(G),H_0(G)^+,[1_{G^0}],H_1(G),M(G),p),$$
	where $M(G)$ is the set of all $G$-invariant probability measures on $G^0$ and $p:M(G)\times H_0(G)\rightarrow \RR$ is the pairing given by
	$p(\mu,[f]_0)=\int f d\mu$.
\end{corollary}
\begin{proof} We have canonical isomorphisms $\mu_i:H_i(G)\cong K_i(C_r^*(G))$. The isomorphism $\mu_0$ clearly extends to an isomorphism of ordered groups respecting the position of the unit.
	Since $G$ is principal, there is an affine homeomorphism between the set $M(G)$ of $G$-invariant probability measures on $G^0$ and the tracial state space $T(C_r^*(G))$ of the reduced groupoid $C^*$-algebra (see, for instance, \cite[Section 4.1]{LiRenault}).
	Finally, from the definition of $\mu_0$ it is clear that the pairings are compatible.
\end{proof}

\section{Examples and applications} \label{Section:Examples&Applications}
In this final section we discuss several applications of our results for specific classes of groupoids and exhibit some interesting examples.
\subsection{Free actions on totally disconnected spaces}
In \cite[Theorem 1.3]{Conley:2020ta}, Conley et.\ al.\ show that for a large class of countable groups $\Gamma$, any free action on a second countable, locally compact, zero-dimensional space has dynamic asymptotic dimension at most the asymptotic dimension of $\Gamma$, and that the latter is finite.  If $X$ is compact, for example a Cantor set, then the dynamic asymptotic dimension will therefore be exactly equal to the asymptotic dimension of $\Gamma$ by \cite[Theorem 6.5]{Guentner:2014aa}.  The class described by the authors of \cite{Conley:2020ta} is technical and we refer there for details; suffice to say that it includes many interesting examples such as all polycylic groups, all virtually nilpotent groups, the lamplighter group $(\Z/2\Z)\wr \Z$, and the Baumslag-Solitar group $BS(1,2)$.  

Hence for such actions, Theorem \ref{main vanishing} implies that $H_n(\Gamma,\ZZ[X])=0$ for all $n>\text{asdim}(\Gamma)$.

\subsection{Smale spaces with totally disconnected stable sets}
A Smale space consists of a self-homeomorphism $\varphi:X\rightarrow X$ of a compact metric space $X$, such that the space can be locally decomposed into the product of a coordinate whose points get closer together as $\varphi$ is iteratively applied, and a coordinate whose points get farther apart under the map $\varphi$. We refer to \cite{Putnam14} for basic definitions and details.
Given a Smale space $(X,\varphi)$ one can define two equivalence relations on $X$ as follows:
$$x\sim_s y\text{ if and only if } \lim_{n\rightarrow\infty}d(\varphi^n(x),\varphi^n(y))=0, \text{ and}$$
$$x\sim_u y\text{ if and only if } \lim_{n\rightarrow\infty}d(\varphi^{-n}(x),\varphi^{-n}(y))=0.$$
Let $X^s(x)$ and $X^u(x)$ denote the stable and unstable equivalence classes of a point $x\in X$ respectively. 
Upon choosing a finite set $P$ of $\varphi$-periodic points one can construct \'{e}tale principal groupoids $G^u(X,P)$ and $G^s(X,P)$ with unit space $X^s(P)=\bigcup_{x\in P}X^s(x)$ and $X^u(P)=\bigcup_{x\in P}X^u(P)$, respectively. In particular, if these unit spaces are totally disconnected, the groupoids are ample. Note further, that for an irreducible Smale space $(X,\varphi)$, the groupoids $G^u(X,P)$ and $G^s(X,P)$ only depend on $P$ up to equivalence. In particular, the choice of $P$ is irrelevant when computing their homology. Deeley and Strung prove in \cite{Deeley2018} that for an irreducible Smale space one has the estimate $$\dad(G^u(X,P))\leq \dim X.$$ Combining this with our Theorem \ref{main vanishing} and Corollary \ref{Cor:HK conjecture up to dimension 2}, and also \cite[Theorem~4.1]{Proietti2020} we can compute the K-theory of the resulting $C^*$-algebras from Putnam's homology for Smale spaces.
\begin{corollary}
	Let $(X,\varphi)$ be an irreducible Smale space with totally disconnected stable sets. Then $H_n^s(X,\varphi)=0$ for all $n>\dim(X)$; and if $\dim(X) \leq 2$ and $H_2^s(X,\varphi)$ is free abelian (e.g. when it is finitely generated), then
	$$K_0(C^*(G^u(X,P)))\cong H_0^s(X,\varphi)\oplus H_2^s(X,\varphi)\text{ and }K_1(C^*(G^u(X,P)))\cong H_1^s(X,\varphi).$$
\end{corollary}

This result includes most of the previously known examples (see e.g. \cite{Yi20}) that were based on separate computations of the K-theory and homology, and hence provides a more conceptual explanation. 
\subsection{Bounded geometry metric spaces}

A metric space $X$ has \emph{bounded geometry} if for each $r>0$ there is a uniform bound on the cardinalities of all $r$-balls in $X$; important examples come from groups with word metric, or from suitable discretisations of Riemannian manifolds.

Skandalis, Tu, and Yu \cite{Skandalis2002} construct an ample groupoid $G(X)$ which captures the coarse geometry of $X$. In particular, the reduced groupoid $C^*$-algebra $C_r^*(G(X))$ can be canonically identified with the uniform Roe algebra $C_u^*(X)$. Let us briefly recall the construction. Let $\beta X$ denote the Stone-\v{C}ech compactification of $X$, i.e.\ the maximal ideal space of $\ell^\infty(X)$.  For each $r>0$, let $E_r$ be the closure of $\{(x,y)\in X\times X\mid d(x,y)\leq r\}$ inside $\beta X\times \beta X$,\footnote{The authors of \cite{Skandalis2002} take the closure in $\beta(X\times X)$ instead of in $\beta X \times \beta X$, but by \cite[Proposition 10.15]{Roe2003} these closures are canonically homeomorphic, so it does not matter which of them one uses.} which is a compact open set.  Then the \emph{coarse groupoid} $G(X)$ of $X$ has as underlying set $\bigcup_{r=0}^\infty E_r$.  The operations are the restriction of the pair groupoid operations from $\beta X\times \beta X$, and the topology is the weak topology coming from the union $\bigcup_{r=0}^\infty E_r$, i.e.\ a subset $U$ of $G$ is open if and only if $U\cap E_r$ is open for each $r$.  Then $G(X)$ is a principal, ample, $\sigma$-compact groupoid with compact base space homeomorphic to $\beta X$, see \cite[Theorem 10.20]{Roe2003}.

Our first goal is to identify the groupoid homology $H_*(G(X))$ with the uniformly finite homology of $X$ introduced by Block and Weinberger in  \cite[Section 2]{Block1992}.
We begin by recalling the relevant definitions. Let $C_n(X)$ denote the collection of all bounded functions $c:X^{n+1}\to \Z$ such that there exists $r>0$ such that if $c(x_0,...,x_n)\neq 0$, then the diameter of the set $\{x_0,...,x_n\}$ is at most $r$.  For each $i\in \{0,...,n\}$, let $\partial^i:X^{n+1}\to X^n$ be defined by $\partial^i(x_0,...,x_n):=(x_0,...,\widehat{x_i},...,x_n)$.  Define $\partial^i_*:C_n(X)\to C_{n-1}(X)$ by 
$$
(\partial^i_*c)=\sum_{\partial^i y=x}c(x)
$$  
and define $\partial:C_n(X)\to C_{n-1}(X)$ by $\partial:=\sum_{i=0}^n (-1)^i \partial^i_*$.  Then we have $\partial\circ \partial=0$, so we get a chain complex.  The \emph{uniformly finite homology} of $X$, denoted $H^{\mathrm{uf}}_*(X)$, is by definition the associated homology of this complex.  

Having introduced all the main actors we can now prove the following theorem.

\begin{theorem}\label{Theorem:hmlgy of coarse grpd is uf hmlgy}
	Let $G(X)$ be the coarse groupoid associated to a bounded geometry metric space $X$.  Then there is a canonical isomorphism $H_*(G(X))\cong H_*^{\mathrm{uf}}(X)$.
\end{theorem}

\begin{proof} For brevity let us denote the coarse groupoid by $G$ throughout the proof.
	Now, for $a\in \Z[EG_n]$, define $\overline{a}\in \Z[EG_n]$ by 
	$$
	\overline{a}(g_0,...,g_n):=\left\{\begin{array}{ll} \sum_{h\in G_x} a(h,hg_1,...,hg_n) & g_0=x\in G^{0} \\ 0 & g_0\not\in G^{0} \end{array}\right..
	$$ 
	One can check (we leave this to the reader) that the equivalence classes $[\overline{a}]$ and $[a]$ in $\Z[EG_n]_G$ of $\overline{a}$ and $a$ respectively are the same, and moreover that $\overline{a}$ is the unique element of $[a]$ that is supported on $\{(g_0,...,g_n)\in EG_n\mid g_0\in G^{0}\}$.  
	
	We define maps $\alpha:\Z[EG_n]_G\to C_n(X)$ and $\beta:C_n(X)\to \Z[EG_n]_G$ as follows.  First, 
	$$
	(\alpha [a])(x_0,...,x_n):=\overline{a}((x_0,x_0),...,(x_0,x_n)).
	$$  
	This makes sense using that $G$ contains the pair groupoid $X\times X$.  We note that $\alpha[a]$ is bounded as $\overline{a}$ is.  Moreover, the fact that $\overline{a}$ has compact support implies that it is supported in a set of the form $E_0\times E_{r_1}\times\cdots\times E_{r_n}\cap EG_n$ for compact open sets $E_{r_i}$ as in the definition of $G$.  It follows that $\alpha[a]$ is supported on the set of tuples with diameter at most $2\max\{r_1,...,r_n\}$ and is thus a well-defined element of $C_n(X)$.
	
	To define $\beta$, let first $((x,x_0),(x,x_1),...,(x,x_n))\in EG_n$ where each pair $(x,x_i)$ is in the pair groupoid.  For $c\in C_n(X)$, define 
	$$
	(\beta c)((x,x_0),(x,x_1),...,(x,x_n)):=\left\{\begin{array}{ll} c(x_0,...,x_n) & x=x_0 \\ 0 & x\neq x_0 \end{array}\right..
	$$
	Due to the support condition on elements of $C_n(X)$, there exists $r>0$ such that $c$ is supported in the set $\{(x_0,...,x_n)\in X^{n+1}\mid d(x_i,x_j)\leq r \text{ for all } i,j\}$.  One can check using the bounded geometry condition that this implies that the closure of the support $S$ of $\beta c$ in $EG_n\cap (X\times X)^{n+1}$ canonically identifies with the Stone-\v{C}ech compactification of $S$, and thus that $\beta c$ extends uniquely to a function on the compact open set $\overline{S}$ as it is bounded, and so a function on $EG_n$ by setting it to be zero outside $\overline{S}$.  We also denote $\beta c$ the corresponding class in $\Z[EG_n]_G$.
	
	Now, having built the maps $\alpha$ and $\beta$, note that both define maps of complexes as the face maps in both cases are given by omitting the $i^\text{th}$ element in a tuple.  To see that they are mutually inverse isomorphisms, one computes directly that $\alpha(\beta(c))=c$, and that $\beta(\alpha[a])=[\overline{a}]$; we leave this to the reader.  The result follows.
\end{proof}

Having identified the homology groups of the coarse groupoid with a more classical object, we would now like to apply our main results and draw some consequences for the computation of the K-theory groups of uniform Roe algebras $C_u^*(X)$ which can be canonically identified with $C_r^*(G(X))$.
Since the spectral sequence (\ref{Equation:Spectral sequence}) is only available in the case of second countable groupoids we need to do some additional work. To this end it is useful to consider a slightly different construction of the coarse groupoid.

Following \cite[Section 2.2]{Skandalis2002}, let $\Gamma_X$ denote the collection of all subsets $A\subseteq X\times X$ such that the first coordinate map $r:X\times X\to X$ and second coordinate maps $s:X\times X\to X$ are both injective when restricted to $A$, and such that $\sup_{(x,y)\in A}d(x,y)<\infty$.  As in \cite[Section 3.1]{Skandalis2002}, every $A\in \Gamma_X$ defines a bijection $t_A:s(A)\rightarrow r(A)$ with the property that $\sup_{x\in s(A)}d(x,t_A(x))<\infty$. Every such bijection extends to a homeomorphism $\phi_A:\overline{s(A)}\to \overline{r(A)}$ between the respective closures in $\beta X$.  As in \cite[Definition 3.1]{Skandalis2002}, we write $\mathscr{G}(X)$ for the collection $\{\phi_A\mid A\in \Gamma_X\}$, which is a pseudogroup, i.e.\ closed under compositions and inverses.  As in \cite[Section 3.2]{Skandalis2002}, the coarse groupoid $G(X)$ of $X$ can be realized as the groupoid of germs associated to this pseudogroup (see \cite[Section 2.6]{Skandalis2002} for the construction of the groupoid of germs associated to a pseudogroup and \cite[Proposition 3.2]{Skandalis2002} for the identification of the two constructions).

Now, as in \cite[Section 3.3]{Skandalis2002}, let us say that a sub-pseudogroup $\mathscr{A}$ of $\mathscr{G}(X)$ is \emph{admissible} if $\bigcup_{\phi_A\in \mathscr{A}}\overline{A}=G(X)$.  Define $G_{\mathscr{A}}$ to be the spectrum of the $C^*$-subalgebra of $C_0(G(X))$ generated by $\{\chi_{\overline{A}}\mid \phi_A\in \mathcal{A}\}$, and let $X_\mathscr{A}$ be the spectrum of the $C^*$-subalgebra of $C(\beta X)=\ell^\infty(X)$ generated $\{\chi_{s(\overline{A})}\mid \phi_A\in \mathcal{A}\}$.  The following comes from \cite[Lemma 3.3]{Skandalis2002} and its proof.

\begin{lemma}\label{adm}
	Let $\mathscr{A}$ be an admissible sub-pseudogroup of $\mathscr{G}(X)$.  Then the groupoid operations naturally factor through the canonical quotient maps $G(X)\to G_{\mathscr{A}}$ and $G(X)^{(0)}\to X_\mathscr{A}$, making $G_{\mathscr{A}}$ an \'{e}tale, locally compact, Hausdorff groupoid with base space $X_{\mathscr{A}}$, which is moreover second countable if $\mathscr{A}$ is countable. 
	
	Moreover, the quotient map $p:\beta X=G(X)^{(0)}\to X_\mathscr{A}$ gives rise to an action of $G_\mathscr{A}$ on $\beta X$, and there is a canonical isomorphism of topological groupoids $G(X)\cong \beta X\rtimes G_\mathscr{A}$.
\end{lemma}

\begin{proof}
	The only part not explicitly in \cite[Lemma 3.3]{Skandalis2002} or its proof is the second countability statement.  This follows as if $\mathscr{A}$ is countable, then the $C^*$-subalgebra of $C_0(G(X))$ generated by $\{\chi_{\overline{A}}\mid \phi_A\in \mathcal{A}\}$ is separable.
\end{proof}

\begin{lemma}\label{Lemma: principal pseudogroup}
	Let $X$ be a bounded geometry metric space. Then there exists a countable admissible sub-pseudogroup $\mathscr{A}$ of $\mathscr{G}(X)$ such that $G_\mathscr{A}$ is principal.
\end{lemma}

\begin{proof}
	First, choose a countable admissible sub-pseudogroup $\mathscr{A}'$ of $\mathscr{G}(X)$ as follows.  For each $n\in \N$, a greedy algorithm based on bounded geometry (compare the discussion in  \cite[Section 2.2, part (a)]{Skandalis2002}) gives a finite decomposition of $\{(x,y)\in X\times X\mid d(x,y)\leq n\}$ such that 
	$$
	\{(x,y)\in X\times X\mid d(x,y)\leq n\}=\bigsqcup_{i=1}^{m_n} A_i^{(n)}
	$$
	and so that each $A_i^{(n)}$ is in $\Gamma_X$.  Let $\mathscr{A}'$ be the sub-pseudogroup of $\mathscr{G}(X)$ generated by all the $A_i^{(n)}$.  It is countable (as generated by a countable set) and it is admissible by construction. Given $\phi\in\mathscr{A}'$ we can apply \cite[Proposition~2.7]{Pitts17} to decompose its domain $dom(\phi)=A_{0,\phi}\sqcup A_{1,\phi}\sqcup A_{2,\phi}\sqcup A_{3,\phi}$ into disjoint clopen sets, where $A_{0,\phi}$ is the set of fixed points of $\phi$ and $\phi(A_{i,\phi})\cap A_{i,\phi}=\emptyset$ for $i=1,2,3$.
	Let $\mathscr{A}$ be the sub-pseudogroup generated by $\mathscr{A}'$ and $\{\id_{A_{i,\phi}}\mid \phi\in\mathscr{A}', 0\leq i\leq 4 \}$. Then $\mathscr{A}$ is still countable and admissible. We claim that $G_\mathscr{A}$ is principal. So let $[\phi,\omega]\in G_\mathscr{A}$ such that its source and range are are equal, i.e. $\phi(\omega)=\omega$. We may assume that $\phi\in\mathscr{A}'$ as there is nothing to show if $\phi$ was already the identity function on some clopen set. We then have $\phi|_{A_{0,\phi}}=\id_{A_{0,\phi}}$ and hence $[\phi,\omega]=[\id,\omega]$ as desired.
\end{proof}

Let us now assume that $X$ has asymptotic dimension at most $d$, or equivalently (see \cite[Theorem 6.4]{Guentner:2014aa}) that $G(X)$ has dynamic asymptotic dimension at most $d$.  For each $n\in \N$, let $E_n:=\overline{\{(x,y)\in X\times X\mid d(x,y)\leq n\}}$, where the closure is taken in $\beta X\times \beta X$.  Then $E_n$ is a compact open subset of $G(X)$.  Hence using the assumption that the dynamic asymptotic dimension of $G(X)$ is at most $d$, there exists a decomposition $\beta X=U_0^{(n)}\sqcup \cdots \sqcup U_d^{(n)}$ of $\beta X$ into compact open subsets such that for for each $i\in \{0,...,d\}$ and each $n\in \N$, the subgroupoid of $G(X)$ generated by $\{g\in E_n\mid r(g),s(g)\in U_i^{(n)}\}$ is compact and open.  Note that as each $U_i^{(n)}$ is clopen in $\beta X$, each $U_i^{(n)}$ is the closure of $V_i^{(n)}:=U_i^{(n)}\cap X$ (this follows as clopen sets in $\beta X$ are in one-one correspondence with arbitrary subsets of $X$ in this way).  Let $\mathscr{B}$ denote the sub-pseudogroup of $\mathscr{G}(X)$ generated by $\mathscr{A}$ as in Lemma \ref{Lemma: principal pseudogroup}, and by $\{\id_{V_i^{(n)}}\mid n\in \N,i\in \{0,...,d\}\}$.

Then we have the following result.

\begin{lemma}\label{sc dad}
	Let $X$ be a bounded geometry metric space with asymptotic dimension at most $d$.  Then there is a second countable, \'{e}tale, locally compact, Hausdorff principal groupoid $G$ with dynamic asymptotic dimension at most $d$, and such that $G$ acts on $\beta X$ giving rise to a canonical isomorphism $\beta X\rtimes G\cong G(X)$.
\end{lemma}

\begin{proof}
	We claim that $G=G_{\mathscr{B}}$ works.  Note that $\mathscr{B}$ is countable (as generated by a countable set) and admissible (as it contains $\mathscr{A}$, which is admissible).  Hence most of the statement follows from Lemma \ref{adm}. As we are only adding further identity functions in the passage from $\mathscr{A}$ to $\mathscr{B}$, we also retain principality by the same proof as in Lemma \ref{Lemma: principal pseudogroup}. We only need to show that the dynamic asymptotic dimension of $G_{\mathscr{B}}$ is at most $d$.
	
	For each $n\in \N$, let us write $[E_n]$ for the image of $E_n\subseteq G(X)$ under the quotient map $G(X)\to G_{\mathscr{B}}$.  Then $[E_n]$ is compact and open: indeed, with notation as in the construction of $\mathscr{A}$ its characteristic function is equal to
	$$
	\sum_{i=1}^{m_n} \chi_{\overline{A_i^{(n)}}},
	$$
	and so in  $C_0(G_{\mathscr{B}})\subseteq C_0(G(X))$.  Moreover, $G_{\mathscr{B}}$ is the union of the $[E_n]$ (as $G(X)$ is the union of the $E_n$).  Hence to show that $G_{\mathscr{B}}$ has dynamic asymptotic dimension at most $d$, it suffices to show that for each $n$ we can find an open cover $W_0,...,W_d$ of $X_{\mathscr{B}}$ such that for each $i$ the subgroupoid $G_i$ of $G_{\mathscr{B}}$ generated by  
	$$
	\{g\in [E_n]\mid s(g),r(g)\in W_i\}
	$$
	has compact closure.  For this, let us take $W_i=U_i^{(n)}$, noting that each $U_i^{(n)}$ makes sense as a clopen subset of $X_{\mathscr{B}}$ by construction of $\mathscr{B}$.  Then we have a decomposition 
	$$
	X_{\mathscr{B}}=\bigsqcup_{i=0}^d U_i^{(n)}
	$$
	coming from the corresponding decomposition of $\beta X$.  Finally, note that each $G_i$ is contained in the image of the subgroupoid $\widetilde{G_i}$ of $G(X)$ generated by 
	$$
	\{g\in E_n\mid s(g),r(g)\in U_i^{(n)}\}
	$$
	under the canonical quotient map $G(X)\to G_{\mathscr{B}}$.  As $\widetilde{G_i}$ is compact (by choice of the $U_i^{(n)}$), and as this quotient map is continuous, we are done.
\end{proof}
We can now use this observation to deduce the existence of a convergent spectral sequence also for the (non-second countable) coarse groupoid: 
\begin{prop}
	Let $X$ be a discrete metric space with bounded geometry and finite asymptotic dimension. Then there exists a convergent spectral sequence
	$$E^2_{p,q}=H_p^{\mathrm{uf}}(X)\otimes K_q(\CC)\Rightarrow K_{p+q}(C_u^*(X)).$$
\end{prop}

\begin{proof}

By Lemma \ref{sc dad} we can write $G(X)=G\ltimes \beta X$ for a principal second countable ample groupoid $G$ with finite dynamic asymptotic dimension. Further, we can write $\beta X$ as an inverse limit $\beta X=\lim\limits_{\leftarrow} Y_i$ of $G$-invariant second countable spaces $Y_i$.
Since $G$ is in particular amenable, it satisfies the strong Baum-Connes conjecture.
Hence (\ref{Equation:Spectral sequence}), for each $i\in I$, provides a convergent spectral sequence $$E_{p,q}^2(i)=H_p(G,K_q(C(Y_i)))\Rightarrow K_{p+q}(C(Y_i)\rtimes_r G).$$ The spectral sequence (\ref{Equation:Spectral sequence}) is a special case of the ABC spectral sequence constructed by Meyer in \cite[Theorem~4.3]{Meyer2008}, and hence it is functorial in the coefficient variable. Consequently, the abelian groups $E_{p,q}^d(i)$ together with the differential maps form a directed system of spectral sequences. Hence we obtain a spectral sequence with $E^d_{p,q}:=\lim_i E^d_{p,q}(i)$ in the limit.
Each of the spectral sequences $(E_{p,q}^d (i))$ converges by \cite[Theorem~5.1]{Meyer2008} and as explained on page 172 of \cite{Meyer2008}, the associated filtrations are functorial in the appropriate sense. Hence taking limits again, we obtain an induced filtration of $\lim K_{p+q}(C(Y_i)\rtimes_r G)$.
Now since $E_{p,q}^2(i)=0$ for all $p>\asdim(X)$ and all $i\in I$ by Theorem \ref{main vanishing}, the induced filtration of $\lim K_{p+q}(C(Y_i)\rtimes_r G)$ is finite and hence we obtain a convergent spectral sequence in the limit:
\begin{equation}\label{Equation:Limit spectral sequence}
E^2_{p,q}=\lim_i H_p(G,K_q(C(Y_i)))\Rightarrow \lim_iK_{p+q}(C(Y_i)\rtimes G).
\end{equation}

There are canonical identifications  $\lim_i H_n(G,C(Y_i,\Z))\cong \lim_i H_n(G\ltimes Y_i)=H_n(G(X))$ and by Theorem \ref{Theorem:hmlgy of coarse grpd is uf hmlgy} we can identify the latter group with $H_n^{\mathrm{uf}}(X)$. On the right hand-side we have $\lim_i K_*(C(Y_i)\rtimes_r G)=K_*(C_r^*(G(X)))=K_*(C_u^*(X))$ and hence we are done.
\end{proof}

\begin{corollary}\label{huf cor}
	Let $X$ be a bounded geometry metric space with asymptotic dimension $d$. Then $H_n^{\mathrm{uf}}(X)=0$ for all $n>d$ and $H_d^{\mathrm{uf}}(X)$ is torsion-free. Moreover,
	\begin{enumerate}
		\item if $\mathrm{asdim}(X)\leq 2$ and $H_2^{\mathrm{uf}}(X)$ is free, or finitely generated, then
		$$K_0(C_u^*(X))\cong H_0^{\mathrm{uf}}(X)\oplus H_2^{\mathrm{uf}}(X),\text{ and } K_1(C_u^*(X))\cong H_1^{\mathrm{uf}}(X),$$ 
		\item if $\mathrm{asdim}(X)\leq 3$, $X$ is non-amenable, and $H_3^{\mathrm{uf}}(X)$ is free, or finitely generated, then
		$$K_0(C_u^*(X))\cong H_2^{\mathrm{uf}}(X),\text{ and } K_1(C_u^*(X))\cong H_1^{\mathrm{uf}}(X)\oplus H_3^{\mathrm{uf}}(X).$$
	\end{enumerate}
\end{corollary}
\begin{proof}
	The first case follows from Theorem \ref{main vanishing} and Corollary \ref{Cor:HK conjecture up to dimension 2} as in our earlier examples. For $(2)$ note that if $\mathrm{asdim} (X)\leq 3$ then the only possibly non-zero differentials on the $E^3$-page are the maps $d_{3,2l}^3:H_3^{\mathrm{uf}}(X)\rightarrow H_0^{\mathrm{uf}}(X)$ for $l\geq 0$. 
	The sequence converges on the $E^4$-page and hence there are short exact sequences
	$$0\rightarrow \mathrm{coker}(d_{3,0}^3)\rightarrow K_0(C_u^*(X))\rightarrow H_2^{\mathrm{uf}}(X)\rightarrow 0, \textrm{and }$$
	$$0\rightarrow H_1^{\mathrm{uf}}(X)\rightarrow K_1(C_u^*(X))\rightarrow \ker(d_{3,0}^3)\rightarrow 0$$
	For a non-amenable space $X$ the group $H_0^{\mathrm{uf}}(X)$ vanishes by \cite[Theorem 3.1]{Block1992}. Since $H_3^{\mathrm{uf}}(X)$ is free the result follows.
\end{proof}

\begin{examples}
Let $\Gamma$ be a countable group equipped with a left invariant bounded geometry metric\footnote{Such a metric exists and is essentially unique by \cite[Proposition 2.3.3]{Willett:2009rt}, for example.}.  Let $\Gamma$ act on the group $\ell^\infty(\Gamma,\Z)$ of bounded $\Z$-valued functions on $\Gamma$ via the action induced by the left translation action of $\Gamma$ on itself.   Then it is well-known that $H^{\mathrm{uf}}_*(\Gamma)$ identifies with the group homology $H_*(\Gamma,\ell^\infty(\Gamma,\ZZ))$ of $\Gamma$ with coefficients in $\ell^\infty(\Gamma,\ZZ)$: see for example \cite[last paragraph on page 1515]{Brodzki:2010fu} (this discusses the case of $H_*^{\mathrm{uf}}$ with real coefficients, but the same argument works for integer coefficients).

Now, assume that $\Gamma$ is a $d$-dimensional Poincar\'{e} duality group, for example if $\Gamma$ is the fundamental group of a closed $d$-manifold with contractible universal cover.  Then 
$$
H^{\mathrm{uf}}_d(\Gamma)\cong H_d(\Gamma,\ell^\infty(\Gamma,\ZZ))\cong H^0(\Gamma,\ell^\infty(\Gamma,\ZZ))\cong \ell^\infty(\Gamma,\ZZ)^\Gamma\cong \Z,
$$
where the first isomorphism is the general fact noted above, the second is Poincar\'{e} duality, the third is the definition of the zeroth cohomology group, and the fourth is straightforward.  In particular, $H^{\mathrm{uf}}_d(\Gamma)$ is free.  

This discussion applies in particular if $\Gamma$ is the fundamental group of a closed orientable surface.  In this case $\Gamma$ is quasi-isometric to either the hyperbolic plane or to the Euclidean plane, whence the asymptotic dimension of $\Gamma$ is two, and we may apply the first part of Corollary \ref{huf cor} to conclude that 
$$
K_0(C_u^*(\Gamma))\cong H_0^{\mathrm{uf}}(\Gamma)\oplus \Z,\text{ and } K_1(C_u^*(\Gamma))\cong H_1^{\mathrm{uf}}(\Gamma).
$$
If moreover the underlying surface has genus at least two, then $\Gamma$ is non-amenable, so $H_0^{\mathrm{uf}}(\Gamma)$ vanishes, and $K_0(C_u^*(\Gamma))\cong \Z$.  

The discussion also applies if $\Gamma$ is the fundamental group of a closed, orientable, hyperbolic $3$-manifold.  In this case $\Gamma$ is non-amenable, and $\Gamma$ is quasi-isometric to hyperbolic $3$-space, so of asymptotic dimension three.  We may thus apply the second part of Corollary \ref{huf cor} to conclude that 
$$
K_0(C_u^*(\Gamma))\cong H_2^{\mathrm{uf}}(\Gamma),\text{ and } K_1(C_u^*(\Gamma))\cong H_1^{\mathrm{uf}}(\Gamma)\oplus \Z.
$$
\end{examples}

\subsection{Examples with topological property (T)}\label{Subsection:Property T}
The HK-conjecture asserts that for a principal ample groupoid we have abstract isomorphisms $\bigoplus_{j\in\NN} H_{2j+i}(G)\cong K_i(C_r^*(G))$. If $G$ has homological dimension $1$ one might be tempted to strengthen this conjecture and ask for the canonical maps $\mu_0$ and $\mu_1$ to be isomorphisms. Here, we show that this strong version of the conjecture fails.

In order to exhibit the examples we need some preliminary facts about topological property (T). Topological property (T) for groupoids was introduced in \cite[Definition 3.6]{DellAiera2018}, and we refer the reader there for the definition.
Let $\mathcal{R}_G=\{\pi_x\mid x\in G^0\}$ denote the family of regular representations of $G$, i.e.
$$\pi_x:C_c(G)\rightarrow B(\ell^2(G_x)),\quad \pi_x(f)\delta_g=\sum_{h\in G_{r(g)}} f(h)\delta_{hg}.$$ Then we have the following generalization of \cite[Proposition~4.19]{DellAiera2018}.
\begin{lemma}\label{Lemma:(T) TransGrpd}
	Let $G$ be an ample groupoid with compact unit space acting on a compact space $X$.
	If $G$ has property (T) with respect to $\mathcal{R}_G$ then $G\ltimes X$ has property (T) with respect to $\mathcal{R}_{G\ltimes X}$.
\end{lemma}
\begin{proof}
	Let $p:X\rightarrow G^0$ denote the anchor map of the action.
	Let $(K,c)$ be a Kazhdan pair for $\mathcal{R}_G$. Now let $$L:=\lbrace (g,x)\in G\ltimes X\mid g\in K\rbrace=(G\ltimes X)\cap (K\times X),$$ which is compact. We claim that $L$ is a Kazhdan set for $\mathcal{R}_{G\ltimes X}$. Indeed consider the regular representation $\pi_x^{G\ltimes X}:C_c(G\ltimes X)\rightarrow B(\ell^2(G_{p(x)}))$ associated with an arbitrary point $x\in X$ and let $\xi\in \ell^2(G_{p(x)})$ be a unit vector.
	Since $(K,c)$ is Kazhdan for $\mathcal{R}_G$, there exists a function $f\in C_c(G)$ with support in $K$ such that $\norm{f}_I\leq 1$ and $$\norm{\pi_{p(x)}^G(f)\xi-\pi_{p(x)}^G(\Psi(f))\xi}\geq c,$$
	where $\Psi:C_c(G)\rightarrow C(G^0)$ is given by $\Psi(f)(x)=\sum_{g\in G^x} f(g)$.
	Since $K$ is compact we can cover it with finitely many compact open bisections $V_1,\ldots ,V_n$ and using a partition of unity argument, we can write $f=\sum f_i$ where $supp(f_i)\subseteq V_i$. Then there must be some $1\leq i\leq n$ such that
	$$\norm{\pi_{p(x)}^G(f_i)\xi-\pi_{p(x)}^G(\Psi(f_i))\xi}\geq \frac{c}{n}.$$
	Using that $f_i$ is supported in a bisection one directly verifies that
	$$\pi_{p(x)}^G(f_i)=\pi_{p(x)}^G(\Psi(f_i))\pi_{p(x)}^G(1_{V_i}) \text{ and }\Psi(f_i)=\Psi(f_i)1_{r(V_i)}$$
	and combining this with the previous observation, we conclude that there exists an $i$ such that 
	$$\norm{\pi^G_{p(x)}(1_{V_i})\xi-\pi^G_{p(x)}(1_{r(V_i)})\xi}\geq \frac{c}{n}.$$
	Now let $V_i\ltimes X$ denote the compact open set $(V_i\times X)\cap G\ltimes X$ and let $f':=1_{V_i\ltimes X}$. Then $f'$ is clearly supported in $L$ with $\norm{f'}_I\leq 1$ and since $\pi_x^{G\ltimes X}(f')=\pi_{p(x)}^G(1_{V_i})$ and $\pi_x^{G\ltimes X}(\Psi(f'))=\pi_{p(x)}^G (1_{r(V_i)})$ we conclude that
	$$\norm{\pi_x^{G\ltimes X}(f')\xi-\pi_x^{G\ltimes X}(\Psi(f'))\xi}\geq \frac{c}{n}.$$
	Hence $(L,\frac{c}{n}$) is a Kazhdan pair for $\mathcal{R}_{G\ltimes X}$.
\end{proof}



Now let $\Gamma$ be a residually finite group and $\mathcal{L}=(N_i)_i$ a sequence of finite index normal subgroups. Let $N_\infty$ denote the trivial subgroup of $\Gamma$ and let $\pi_i:\Gamma\rightarrow \Gamma/N_i$ be the quotient map. We denote by $G_\mathcal{L}$ the associated HLS groupoid, i.e. the group bundle $\bigsqcup_{i\in\NN\cup \lbrace\infty\rbrace} \Gamma/N_i$ equipped with the topology generated by the singleton sets $\lbrace (i,\gamma)\rbrace$ for $i\in\NN$ and $\gamma\in \Gamma$, and the tails $\lbrace (i,\pi_i(\gamma))\mid i>N\rbrace$ for each fixed $\gamma\in\Gamma$ and $N\in\NN$. It is well-known that this groupoid is Hausdorff if and only if for each $\gamma\in \Gamma\setminus \{e\}$ the set $\{ i\in\NN\mid \gamma\in N_i\}$ is finite. This is in particular the case if the sequence is nested and has trivial intersection. 

Following a construction of Alekseev and Finn-Sell in \cite{Alekseev2018} we associate a principal groupoid to this data as follows.
Let $X:=\bigsqcup_{i\in \NN\cup\{\infty\}} \Gamma/N_i$. Then $X$ carries a canonical action of the HLS groupoid $G_\mathcal{L}$ given by left multiplication. For $\gamma N_i\in \Gamma/N_i$ let
$$Sh(\gamma N_i)=\bigcup_{j\geq i} \pi_{i,j}^{-1}(\gamma N_i)$$ be the \textit{shadow} of $\gamma N_i$ in $X$. Now let $\widehat{X}$ be the spectrum of the smallest $G_{\mathcal{L}}$-invariant $C^*$-subalgebra $B\subseteq \ell^\infty (X)$ containing 
$$\{\delta_x\mid x\in X\}\cup \{1_{Sh(\gamma N_i)}\mid \gamma\in\Gamma, i\in\NN\}.$$
Since $B$ is $G_\mathcal{L}$-invariant, $\widehat{X}$ also carries an action of $G_\mathcal{L}$ and we can form the transformation groupoid $\mathcal{G}:=G_\mathcal{L}\ltimes \widehat{X}$.
As explained just after Remark 2.2 in \cite{Alekseev2018}, this groupoid is principal. Moreover, $X\subseteq \widehat{X}$ is a dense open $G_\mathcal{L}$-invariant subset with complement $\widehat{X}\setminus X\cong \widehat{\Gamma}_\mathcal{L}=\varprojlim \Gamma/N_i$.
Hence we obtain isomorphisms
$$\mathcal{G}|_{X}\cong \bigsqcup_{i\in\NN\cup\{\infty\}}\Gamma/N_i\ltimes \Gamma/N_i\quad \text{ and }\quad \mathcal{G}|_{\widehat{X}\setminus X}\cong \Gamma\ltimes \widehat{\Gamma}_\mathcal{L}.$$

The following result relies on property $(\tau)$ as defined in \cite[Definition~4.3.1]{Lubotzky94}.

\begin{prop}\label{Prop: Property (T) for principal HLS}
	Suppose $\Gamma$ is a finitely generated, residually finite group and $\mathcal{L}=(N_i)_i$ is a sequence of finite index normal subgroups with property $(\tau)$. Then the following hold:
	\begin{enumerate}
		\item $\mathcal{G}$ has topological property $(T)$ with respect to the family of regular representations in the sense of \cite{DellAiera2018}.
		\item The sequence
		$$K_0(C_r^*(\mathcal{G}|_X))\rightarrow K_0(C_r^*(\mathcal{G}))\rightarrow K_0(C_r^*(\mathcal{G}|_{\widehat{X}\setminus X}))$$
		is not exact in the middle.
	\end{enumerate}
\end{prop}
\begin{proof}
	Since the regular representations of $\mathcal{G}$ extend to $C_r^*(\mathcal{G})$ by definition of the reduced groupoid $C^*$-algebra, the result follows from \cite[Proposition~4.15]{DellAiera2018} and Lemma \ref{Lemma:(T) TransGrpd}. Part (2) follows from (1) and \cite[Proposition~7.14]{DellAiera2018}.
\end{proof}

We can now provide some concrete principal ample groupoids where $\mu_0$ is not surjective.

Let $\Gamma=\FF_2$ and choose a nested sequence $(N_i)_i$ of finite index normal subgroups in $\FF_2$ with property $(\tau)$ such that the associated HLS groupoid is Hausdorff.
\begin{ex} To have a concrete example of such a sequence in mind consider the nested family $(L_i)_i$ of finite index normal subgroups $L_i:=\ker (SL_2(\ZZ)\rightarrow SL_2(\ZZ/5^i))$ of $SL_2(\ZZ)$. Embed $\FF_2$ in $SL_2(\ZZ)$ as a finite index normal subgroup and let $N_i:=L_i\cap \FF_2$. Then $(N_i)_i$ is a nested family of finite index normal subgroups of $\FF_2$ with trivial intersection. Since $\FF_2$ has finite index in $SL_2(\ZZ)$ and $SL_2(\ZZ)$ has property $(\tau)$ with respect to the family $(L_i)_i$ we conclude that $\FF_2$ has $(\tau)$ with respect to the family $(N_i)_i$.
\end{ex}
Let us first compute the homology of the associated groupoid $\mathcal{G}$. Consider the long exact sequence in homology
$$\cdots\rightarrow H_n(\mathcal{G}|_{X})\rightarrow H_n(\mathcal{G})\rightarrow H_n(\mathcal{G}|_{\widehat{X}\setminus X})\rightarrow H_{n-1}(\mathcal{G}|_X)\rightarrow \cdots\rightarrow H_0(\mathcal{G}|_{\widehat{X}\setminus X})$$
corresponding to the decomposition $\widehat{X}=X\sqcup \widehat{X}\setminus X$. Since $\mathcal{G}|_X$ is a disjoint union of principal and proper groupoids, we have $$H_n(\mathcal{G}|_X)=\bigoplus_{i\in\NN\cup\{\infty\}} H_n(\Gamma/N_i\ltimes \Gamma/N_i)=0 \text{ for all }n\geq 1,\text{ and }$$ $$H_0(\mathcal{G}|_X)=\bigoplus_{i\in\NN\cup\{\infty\}} H_0(\Gamma/N_i\ltimes \Gamma/N_i)=\bigoplus_{i\in\NN\cup\{\infty\}} \ZZ$$
From the long exact sequence we conclude that for all $n\geq 2$ the restriction to the boundary induces isomorphisms
$$H_n(\mathcal{G})\cong H_n(\FF_2\ltimes \widehat{\FF_2}_\mathcal{L}).$$
It is a well-known fact that $H_n(\FF_2\ltimes \widehat{\FF_2}_\mathcal{L})\cong H_n(\FF_2,C(\widehat{\FF_2}_\mathcal{L},\ZZ))$
and the homology of the free group $\FF_2$ is well-known to be trivial for all $n\geq 2$.
Hence $H_n(\mathcal{G})=0$ for all $n\geq 2$.

Now by construction $H_0(\FF_2\ltimes \widehat{\FF_2}_\mathcal{L})=C(\widehat{\FF_2}_\mathcal{L},\ZZ)/\langle f-\gamma.f\mid \gamma\in\FF_2\rangle$ 
and from the Pimsner-Voiculescu exact sequence for actions of free groups from \cite[Theorem~3.5]{Pimsner1982} we obtain
$K_0(C(\widehat{\FF_2}_\mathcal{L})\rtimes_r\FF_2)\cong C(\widehat{\FF_2}_\mathcal{L},\ZZ)/im(\beta)$ where
$$\beta:C(\widehat{\FF_2}_\mathcal{L},\ZZ)^2\rightarrow C(\widehat{\FF_2}_\mathcal{L},\ZZ),\ (f_1,f_2)\mapsto f_1-a^{-1}.f_1+f_2-b^{-1}.f_2$$
We clearly have $im(\beta)=\langle f-\gamma.f\mid \gamma\in\FF_2 \rangle $ so that after the identification above, $\mu_0$ is the identity. 

Now consider the commutative diagram

\begin{center}
	\begin{tikzpicture}[description/.style={fill=white,inner sep=2pt}]
	\matrix (m) [matrix of math nodes, row sep=3em,
	column sep=1.5em, text height=1.5ex, text depth=0.25ex]
	{H_0(\mathcal{G}|_X) & H_0(\mathcal{G}) & H_0(\mathcal{G}|_{\widehat{X}\setminus X})\\
		K_0 (C_r^*(\mathcal{G}|_X))&	K_0(C_r^*(\mathcal{G}))& K_0(C_r^*(\mathcal{G}|_{\widehat{X}\setminus X})) \\
	};
	\path[->,font=\scriptsize]
	(m-1-1) edge node[auto] {$ i $} (m-1-2)
	(m-1-2) edge node[auto] {$ p $} (m-1-3)

	(m-2-1) edge node[auto] {$ i_0 $} (m-2-2)
	(m-2-2) edge node[auto] {$  $} (m-2-3)
	
	(m-1-1) edge node[auto] {$  \mu_0 $} (m-2-1)
	(m-1-2) edge node[auto] {$ \mu_0^\mathcal{G}  $} (m-2-2)
	(m-1-3) edge node[auto] {$ \mu_0  $} (m-2-3)
	
	;
	\end{tikzpicture}
\end{center}
The top row is exact in the middle, as it is part of the long exact sequence in homology corresponding to the open invariant subset $X\subseteq \widehat{X}$.
The bottom row however is not exact in the middle by Proposition \ref{Prop: Property (T) for principal HLS}. The map on the right hand side is an isomorphism by our reasoning above.

We claim that the map $\mu_0^{\mathcal{G}}$ is not surjective. Suppose for contradiction that it was. Let $x\in K_0(C_r^*(\mathcal{G}))$ be an element which maps to zero in $K_0(C_r^*(\mathcal{G}|_{\widehat{X}\setminus X})$ but is not in the image of $i_0$. Since $\mu_0^\mathcal{G}$ is surjective, we can find an element $y\in H_0(\mathcal{G})$ such that $\mu_0^\mathcal{G}(y)=x$.
But then by commutativity of the right hand square we have $\mu_0(p(y))=0$ and since $\mu_0$ is injective we conclude that $p(y)$ is zero and hence $y=i(z)$ for some $z\in H_0(\mathcal{G}|_X)$. Moreover, by commutativity of the left square we have $x=\mu_0^\mathcal{G}(y)=i_0(\mu_0(z))$, which contradicts our assumption that $x\not\in \mathrm{im}(i_0)$.

\begin{rem}
At this point, it seems there are three known reasons for failure of the HK conjecture.  The first, due to Scarparo \cite{Scarparo2020} is the presence of torsion in isotropy groups.  The second, due to Deeley \cite{Deeley22} is due to torsion phenomena in $K$-theory; however, Deeley's results do not contradict the ``rational'' HK conjecture one gets after tensoring with $\Q$, analogously to the classical fact that the Chern character is a rational isomorphism between $K$-theory and cohomology.  The third is exotic analytic phenomena connected to the failure of the Baum-Connes conjecture as discussed above (this is admittedly not exactly a failure of the HK conjecture, but it seems to us as evidence that the HK conjecture should sometimes fail when the Baum-Connes conjecture fails).  Based on these counterexamples, the following ``folk conjecture'' (arrived at independently by several people) seems reasonable: if $G$ is an ample, second countable groupoid with torsion free isotropy and satisfying the strong Baum-Connes conjecture, then there are isomorphisms
$$
K_0(C^*_r(G))\otimes \Q\cong \bigoplus_{k\text{ even}}H_k(G;\Q) \quad  \text{and} \quad K_1(C^*_r(G))\otimes \Q\cong \bigoplus_{k\text{ odd}}H_k(G;\Q).
$$
\end{rem}

	\ \newline
{\bf Acknowledgments}. The first author expresses his gratitude to the second and fourth authors, and the Department of Mathematics at the University of \Hawaii~at \Manoa~for their hospitality during a visit in November 2019 where part of this work was carried out.  The authors would also like to thank Robin Deeley for pointing out a mistake in an earlier version of this paper, and the anonymous referee for suggesting the approach to the comparison maps in Subsection \ref{pyss sec}.

	\ \newline
On behalf of all authors, the corresponding author states that there is no conflict of interest.


\begin{thebibliography}{10}

\bibitem{Alekseev2018}
Vadim {Alekseev} and Martin {Finn-Sell}.
\newblock {Non-amenable principal groupoids with weak containment.}
\newblock {\em {Int. Math. Res. Not.}}, 2018(8):2332--2340, 2018.

\bibitem{Ara2020}
Pere Ara, Christian B\"{o}nicke, Joan Bosa, and Kang Li.
\newblock The type semigroup, comparison and almost finiteness for ample
  groupoids.
\newblock {\em Ergodic Theory and Dynam. Systems} 1--40, 2021. doi:10.1017/etds.2021.115

\bibitem{Bartels2018}
Arthur {Bartels}, Wolfgang {L\"uck}, and Holger {Reich}.
\newblock {Equivariant covers for hyperbolic groups}.
\newblock {\em {Geom. Topol.}}, 12(3):1799--1882, 2008.


\bibitem{Bergman:1972aa}
George  Bergman.
\newblock {Boolean Rings of Projection Maps}.
\newblock {\em J. London Math. Soc. (2)}, 2:593--598, 05
  1972.

\bibitem{Block1992}
Jonathan {Block} and Shmuel {Weinberger}.
\newblock {Aperiodic tilings, positive scalar curvature, and amenability of
  spaces.}
\newblock {\em {J. Am. Math. Soc.}}, 14:907--918, 1992.

\bibitem{B20}
Christian B{\"o}nicke.
\newblock {A Going-Down principle for ample groupoids and the Baum-Connes conjecture.}
\newblock {\em  {Adv. Math.}}, 372, 2020.


\bibitem{BP22}
Christian B{\"o}nicke and Valerio Proietti.
\newblock {Categorical} approach to the {Baum-Connes} conjecture for {\'e}tale groupoids.
\newblock available at  	
https://doi.org/10.48550/arXiv.2202.08067, 2022.

\bibitem{Brodzki:2010fu}
Jacek {Brodzki}, Graham {Niblo}, and Nick {Wright}.
\newblock {Pairings, duality, amenability and bounded cohomology.}
\newblock {\em {J. Eur. Math. Soc.}}, 14(5):1513--1518, 2012.


\bibitem{CGSTW}
Jose Carri{\'o}n, James Gabe, Christopher Schafhauser, Aaron Tikuisis, and
  Stuart White.
\newblock Classification of {$^\ast$}-homomorphisms {I}: simple nuclear
  {$C^\ast$}-algebras.
\newblock In preparation.

\bibitem{Conley:2020ta}
Clinton Conley, Steve Jackson, Andrew Marks, Brandon Seward, and Robin
  Tucker-Drob.
\newblock Borel asymptotic dimension and hyperfinite equivalence relations.
\newblock arXiv:2009.06721, 2020.

\bibitem{Crainic2000}
Marius {Crainic} and Ieke {Moerdijk}.
\newblock {A homology theory for \'etale groupoids.}
\newblock {\em {J. Reine Angew. Math.}}, 521:25--46, 2000.

\bibitem{Deeley2018}
Robin  {Deeley} and Karen {Strung}.
\newblock {Nuclear dimension and classification of \({C}^\ast\)-algebras
  associated to Smale spaces.}
\newblock {\em {Trans. Am. Math. Soc.}}, 370(5):3467--3485, 2018.

\bibitem{Deeley22}
Robin {Deeley}.
\newblock {A counterexample to the {HK} conjecture that is principal.}
\newblock arXiv:2106.01527v2, 2022.



\bibitem{DellAiera2018}
Cl\'{e}ment Dell'Aiera and Rufus Willett.
\newblock Topological property {(T)} for groupoids.
\newblock arXiv:1811.07085v2.

\bibitem{ElliottGongLinNiu}
George Elliott, Guihua Gong, Huaxin Lin, and Zhuang Niu.
\newblock The classification of simple separable unital {$\mathcal Z$}-stable
  locally {ASH} algebras.
\newblock {\em J. Funct. Anal.}, 272(12):5307--5359, 2017.

\bibitem{Engel:2019ux}
Alexander Engel.
\newblock Rough index theory on spaces of polynomial growth and contractibility.
\newblock {\em J. Noncommut. Geom.}, 13(2):617--666, 2019. 

\bibitem{Farsi2019}
Carla Farsi, Alex Kumjian, David Pask, and Aidan Sims.
\newblock Ample groupoids: Equivalence, homology and {M}atui's {HK} conjecture.
\newblock {\em {M{\"u}nster J. of Math.}}, 12:411--451, 2019.

\bibitem{Giordano:2003aa}
Thierry Giordano, Ian Putnam, and Christian Skau.
\newblock Affable equivalence relations and orbit structure of {C}antor
  dynamical systems.
\newblock {\em Ergodic Theory Dynam. Systems}, 24(2):441--475, 2004.

\bibitem{Godement:1958aa}
Roger Godement.
\newblock {\em Topologie alg\'{e}brique et th\'{e}orie des faisceaux}.
\newblock Hermann, 1958.

\bibitem{GongLinNiu-1}
Guihua Gong, Huaxin Lin, and Zhuang Niu.
\newblock A classification of finite simple amenable {$\mathcal Z$}-stable
  {$C^\ast$}-algebras, {I}: {$C^\ast$}-algebras with generalized tracial rank
  one.
\newblock {\em C. R. Math. Acad. Sci. Soc. R. Canada}, 42(3):63--450, 2020.

\bibitem{GongLinNiu-2}
Guihua Gong, Huaxin Lin, and Zhuang Niu.
\newblock Classification of finite simple amenable {$\mathcal Z$}-stable
  {$C^\ast$}-algebras, {II}: {$C^\ast$}-algebras with rational generalized
  tracial rank one.
\newblock {\em C. R. Math. Acad. Sci. Soc. R. Canada}, 42(4):451--539, 202.

\bibitem{Gromov:1993tr}
M.~Gromov.
\newblock Asymptotic invariants of infinite groups.
\newblock In G.~Niblo and M.~Roller, editors, {\em Geometric Group Theory},
  volume~2, 1993.

\bibitem{Guentner:2014bh}
Erik Guentner, Rufus Willett, and Guoliang Yu.
\newblock Finite dynamical complexity and controlled operator {K}-theory.
\newblock arXiv:1609.02093, 2016.

\bibitem{Guentner:2014aa}
Erik {Guentner}, Rufus {Willett}, and Guoliang {Yu}.
\newblock {Dynamic asymptotic dimension: relation to dynamics, topology, coarse
  geometry, and \(C^\ast\)-algebras.}
\newblock {\em {Math. Ann.}}, 367(1-2):785--829, 2017.

\bibitem{Haslehurst}
Mitch Haslehurst.
\newblock {Relative $K$-theory for $C^*$-algebras.}
\newblock arXiv:2106.02620, 2021.

\bibitem{Higson2002}
Nigel {Higson}, Vincent {Lafforgue}, and Georges {Skandalis}.
\newblock {Counterexamples to the Baum-Connes conjecture.}
\newblock {\em {Geom. Funct. Anal.}}, 12(2):330--354, 2002.

\bibitem{Higson:2000bs}
Nigel Higson and John Roe.
\newblock {\em Analytic ${K}$-homology}.
\newblock Oxford University Press, 2000.

\bibitem{Kerr2020}
David Kerr.
\newblock Dimension, comparison, and almost finiteness.
\newblock {\em J. Eur. Math. Soc.}, (to appear in print), 2020.

\bibitem{Khoshkam:2004uj}
M.~Khoshkam and G.~Skandalis.
\newblock Corssed products of ${C^*}$-algebras by groupoids and inverse
  semigroups.
\newblock {\em J. Operator Theory}, 51(2):255--279, 2004.

\bibitem{Kirchberg-ICM}
Eberhard Kirchberg.
\newblock Exact {${\rm C}^*$}-algebras, tensor products, and the classification
  of purely infinite algebras.
\newblock In {\em Proceedings of the {I}nternational {C}ongress of
  {M}athematicians, {V}ol.\ 1, 2 ({Z}\"urich, 1994)}, pages 943--954.
  Birkh\"auser, Basel, 1995.
  
  \bibitem{Le-Gall:1999aa}
P.-Y. Le~Gall.
\newblock Th\'{e}orie de {K}asparov \'{e}quivariante et groupo\"{i}des. {I}.
\newblock {\em ${K}$-theory}, 16(4):361--390, 1999.


\bibitem{LiRenault}
Xin Li and Jean Renault.
\newblock Cartan subalgebras in {${\rm C}^*$}-algebras. {E}xistence and
  uniqueness.
\newblock {\em Trans. Amer. Math. Soc.}, 372(3):1985--2010, 2019.

\bibitem{Lubotzky94}
Alexander {Lubotzky}.
\newblock {\em {Discrete groups, expanding graphs and invariant measures.
  Appendix by Jonathan D. Rogawski}}, volume 125.
\newblock Basel: Birkh\"auser, 1994.

\bibitem{Matui2012}
Hiroki {Matui}.
\newblock {Homology and topological full groups of \'etale groupoids on totally
  disconnected spaces.}
\newblock {\em {Proc. Lond. Math. Soc. (3)}}, 104(1):27--56, 2012.

\bibitem{Matui2015}
Hiroki {Matui}.
\newblock {Topological full groups of one-sided shifts of finite type.}
\newblock {\em {J. Reine Angew. Math.}}, 705:35--84, 2015.

\bibitem{Matui2016}
Hiroki {Matui}.
\newblock {\'Etale groupoids arising from products of shifts of finite type.}
\newblock {\em {Adv. Math.}}, 303:502--548, 2016.

\bibitem{Meyer2008}
Ralf {Meyer}.
\newblock {Homological algebra in bivariant \(K\)-theory and other triangulated
  categories. II.}
\newblock {\em {Tbil. Math. J.}}, 1:165--210, 2008.

\bibitem{Meyer:2010vt}
R.~Meyer and R.~Nest.
\newblock Homological algebra in bivariant {$K$}-theory and other triangulated
  categories. {I}.
\newblock In {\em Triangulated Categories}, volume 375 of {\em London Math.
  Soc. Lecture Note Ser.}, pages 236--289. Cambridge University Press, 2010.

\bibitem{Ortega2020}
Eduard Ortega.
\newblock Homology of the {K}atsura-{E}xel-{P}ardo groupoid.
\newblock {\em J. Noncommut. Geom.}, 14:1--23, 2020.

\bibitem{Phillips-documenta}
N.~Christopher Phillips.
\newblock A classification theorem for nuclear purely infinite simple
  {$C^*$}-algebras.
\newblock {\em Doc. Math.}, 5:49--114 (electronic), 2000.

\bibitem{Pimsner1982}
Mihai {Pimsner} and Dan-Virgil {Voiculescu}.
\newblock {K-groups of reduced crossed products by free groups.}
\newblock {\em {J. Oper. Theory}}, 8:131--156, 1982.

\bibitem{Pitts17}
David {Pitts}.
\newblock {Structure for regular inclusions. I}.
\newblock {\em {J. Oper. Theory}}, 78(2):357--416, 2017.

\bibitem{Proietti2020}
Valerio {Proietti} and Makoto {Yamashita}.
\newblock Homology and {$K$}-theory of dynamical systems II. Smale spaces with totally disconnected transversal.
\newblock available at https://doi.org/10.48550/arXiv.2104.10938, 2021, to appear in J. Noncommut. Geom.

\bibitem{Proietti:2021wz}
V.~Proietti and M.~Yamashita.
\newblock Homology and ${K}$-theory of dynamical systems {I}. torsion-free
  ample groupoids.
\newblock {\em Ergodic Theory Dynam. Systems}, 42:2630--2660, 2021.

\bibitem{Putnam1997}
Ian {Putnam}.
\newblock {An excision theorem for the K-theory of \(C^*\)-algebras.}
\newblock {\em {J. Oper. Theory}}, 38(1):151--171, 1997.

\bibitem{Putnam14}
Ian {Putnam}.
\newblock {A homology theory for Smale spaces},
\newblock {\em Mem. Amer. Math. Soc.}, 232(1094):viii+122, 2014.


\bibitem{Roe2003}
John {Roe}.
\newblock {\em {Lectures on coarse geometry}}, volume~31.
\newblock Providence, RI: American Mathematical Society (AMS), 2003.


\bibitem{Scarparo2020}
Eduardo {Scarparo}.
\newblock {Homology of odometers.}
\newblock {\em {Ergodic Theory Dyn. Syst.}}, 40(9):2541--2551, 2020.

\bibitem{Skandalis2002}
Georges {Skandalis}, Jean-Louis {Tu}, and Guoliang {Yu}.
\newblock {The coarse Baum-Connes conjecture and groupoids.}
\newblock {\em {Topology}}, 41(4):807--834, 2002.

\bibitem{stacks-project}
The {Stacks Project Authors}.
\newblock \textit{Stacks Project}.
\newblock \url{https://stacks.math.columbia.edu}, 2020.

\bibitem{Steen78}
Lynn~Arthur Steen and J.~Arthur Seebach, Jr.
\newblock {\em Counterexamples in topology}.
\newblock Springer-Verlag, New York-Heidelberg, second edition, 1978.

\bibitem{TikuisisWhiteWinter}
Aaron Tikuisis, Stuart White, and Wilhelm Winter.
\newblock Quasidiagonality of nuclear {$C^\ast$}-algebras.
\newblock {\em Ann. of Math. (2)}, 185(1):229--284, 2017.

\bibitem{Tu:1999bq}
Jean-Louis Tu.
\newblock La conjecture de {B}aum-{C}onnes pour les feuilletages moyennables.
\newblock {\em ${K}$-theory}, 17:215--264, 1999.

\bibitem{Weibel1994}
Charles {Weibel}.
\newblock {\em {An introduction to homological algebra}}, volume~38.
\newblock Cambridge: Cambridge University Press, 1994.

\bibitem{Willett:2009rt}
Rufus {Willett}
\newblock{ \em{Some notes on property A}}
\newblock In {\em Limits of Graphs in Group Theory and Computer Science}, pages 191--281.
EPFL Press, 2009.

\bibitem{Yi20}
Inhyeop {Yi}.
\newblock {Homology and Matui's HK conjecture for groupoids on one-dimensional
  solenoids}.
\newblock {\em {Bull. Aust. Math. Soc.}}, 101(1):105--117, 2020.

\end{thebibliography}
\end{document}